\documentclass[11 pt]{amsart}
\usepackage{amssymb}
\usepackage{amsmath}
\usepackage{amsthm}
\usepackage{amsfonts}
\usepackage{amssymb,amssymb}
\usepackage{graphicx}
\usepackage{color}
\usepackage{epsfig}
\usepackage[pdfpagelabels,hyperindex]{hyperref}
\usepackage{xcolor}
\DeclareGraphicsExtensions{.pdf,.jpg,.png,.ps,.eps,.gif}

\theoremstyle{plain}
\newtheorem{theorem}{Theorem}\numberwithin{theorem}{section}
{}
\newtheorem{main}{Main~Theorem}{}
\newtheorem{lemma}{Lemma}\numberwithin{lemma}{section}
\newtheorem{proposition}{Proposition}\numberwithin{proposition}{section}
\numberwithin{corollary}{section}
\theoremstyle{definition}
\newtheorem{definition}{Definition}\numberwithin{definition}{section}
\theoremstyle{remark} \theoremstyle{exam} \theoremstyle{ob}
\newtheorem{remark}{Remark}\numberwithin{remark}{section}
\numberwithin{question}{section}
\newtheorem{exam}{Example}\numberwithin{exam}{section}
\numberwithin{ob}{section}
\numberwithin{nt}{section}
\numberwithin{equation}{section}

\begin{document}
\title[Discrete Double Hilbert Transforms Along Polynomial Surfaces]
{ Discrete Double Hilbert Transforms Along Polynomial Surfaces}

\author
{Joonil Kim, \ Hoyoung Song}

\address{Joonil Kim, Department of Mathematics \\
       Yonsei University \\
       Seoul 120-729, Republic of Korea}
\email{jikim7030@yonsei.ac.kr}
\address{Hoyoung Song, Department of Mathematics \\
	Yonsei University \\
	Seoul 120-729, Republic of Korea}
\email{nonspin0070@gmail.com}


 
 \begin{abstract}
 We develop the Hardy-Ramanujan circle method for a two-parameter setting and prove that the exponential sum, with amplitude $1/(t_1t_2)$  over $(t_1,t_2)$  in lattice boxes of any size, is uniformly bounded if and only if all vertices of the Newton polyhedron of the phase polynomial have at least one even component. Additionally, we establish the double arc method to handle the graph case $(t_1,t_2,P(t_1,t_2))$ without any degree constraints, a problem introduced as a challenging issue by Bourgain, Mirek, Stein, and Wright\cite{B4}. We also show the  
$\ell^{p}$ boundedness of the  discrete double Hilbert transforms along polynomial surfaces given the evenness condition for the vertices. 
 \end{abstract}
 \maketitle

\setcounter{tocdepth}{1}

\tableofcontents

\section{Introduction}
For $d\ge 1$, let $  \mathbb{Z}[t]$ be the set of polynomials of $t=(t_1,\cdots,t_d)$ whose coefficients are integers. We write its element $P(t)$ as
$P(t)=\sum_{\mathfrak{m}\in\Lambda(P)} c_{\mathfrak{m}}t^{\mathfrak{m}}$. Here, $\Lambda(P)\subset \mathbb{Z}_{\ge 0}^d$ is the set of all exponents $\mathfrak{m}$ of $t^{\mathfrak{m}}$ in $P(t)$.
 For the last three decades, the Hardy-Ramanujan circle method  has been an indispensable tool  for studying the exponential sums associated with the  following two types of summations at $x\in\mathbb{Z}$:
\begin{align*}
A_N(f)(x)&=\frac{1}{2N+1}\sum_{t\in \mathbb{Z}; |t|\le N}f(x- P(t) ) \\
H_N(f)(x)&= \sum_{t\in \mathbb{Z}\setminus\{0\}; |t|\le N}f(x- P(t) ) \frac{1}{t}
\end{align*}
where $P(t)\in \mathbb{Z}[t]$ and   $f\in \mathcal{S}(\mathbb{R})$. In the late 1980s, Bourgain \cite{B1,B2,B3} utilized the circle method and estimated the exponential sum associated with the discrete average 
 $A_N$  to prove   the $\ell^p(\mathbb{Z})\mapsto \ell^p(\mathbb{Z})$ boundedness of the  discrete maximal average operator $f\rightarrow\sup_N A_N(f)$.  His results  generalized the   Birkoff ergodic theorem \cite{Birkh}, showing that
\begin{align}\label{0101}
 \lim_{N\rightarrow\infty} \frac{\sum_{t\in [1,N]\cap \mathbb{Z}} f(T^{P(t)}x)}{N}\ \text{for a.e. $x\in X$}
\end{align}
where   $f\in L^{\infty}(X)$ and $T:X\rightarrow X$ is a 1-1 and onto measure-preserving map on a probability space $(X,\mathcal{B}(X),\mu)$.  
In 1987  Arkhipov and Oskolkov   \cite{AO} demonstrated the uniform boundedness of the exponential sum associated with the above discrete Hilbert transform $H_N$:  
\begin{align}\label{11}
\sup_{N\in \mathbb{N}\ \text{and coefficients of $P$} }\left| \sum_{t\in \mathbb{Z};\ 1\le |t|\le N} \frac{e^{2\pi i  P(t)}}{t}\right|\le C
\end{align} 
where $C>0$ is independent of coefficients of $P(t)\in \mathbb{R}[t]$ but it depends on only $deg(P)$.       In 1990,  Stein and Wainger \cite{SW2}  replaced  the amplitude  $1/t$ of the above discrete Hilbert transform $H_N$ with the Calderon-Zygmund kernel $K(t)$, and defined the discrete  singular Radon transform  $f\mapsto Rf$  by
 $$Rf(x):=\sum_{t\in \mathbb{Z}^d\setminus\{{\bf 0} \}  }  f(x-(P_1(t),\cdots,P_M(t))K(t)\ \text{where $P_\nu(t)\in   \mathbb{Z}[t_1,\cdots,t_d]$}$$  
for $\nu=1,\cdots,d$. They proved  the $\ell^p(\mathbb{Z}^M)\rightarrow\ell^p(\mathbb{Z}^M)$    $(3/2< p <3)$  boundedness of $R$. In 2002, Magyar, Stein and Wainger in \cite{MSW} generalized a sampling technique to handle a class of discrete multipliers to prove the $\ell^p$ boundedness of the discrete analogue for Stein's maximal  averages over spheres. In 2006,  Ionescu and Wainger \cite{IW} established the  theorem  that serves as an efficient tool for handling a wide class of discrete multipliers, including the discrete Radon transform $R$ in the full range $1<p<\infty$. Subsequently, in 2007, Ionescu, Magyar, Stein and Wainger \cite{IMSW} investigated the non-translation invariant  cases associated with polynomials whose degree $\le 2$.  They  obtained the corresponding ergodic theorems of (\ref{0101}) related with  non-commuting measure preserving maps.   Later,  Mirek and Trojan   \cite{MT} completed  the multi-dimensional  $L^p$ ergodic theorem for $p>1$  concerning $T_1^{P_1(t)}\cdots T_M^{P_M(t)}$ over the cube $[1,N]^d$ in (\ref{0101}).
Recently,  Mirek, Szarek and Kranich   \cite{MSZ} obtained the useful estimate known as the  jump inequality for discrete Radon-transforms. More recently,  Ionescu, Magyar,  Mirek and Szarek \cite{IMMS1} proved some partial cases of the  Furstenberg-Bergelson-Leibman conjecture regarding the pointwise ergodic theorems, extending the result of \cite{IMSW}. For a   survey, see \cite{IMW} and for  a generalization of Ionescu--Wainger multiplier theorem, refer to Tao's work in \cite{Tao}.
\subsection{Multi-Parameter-Theory}  In 1951, Dunford  \cite{Dun} and Zygmund \cite{Zyg}   initiated the study of the multi-parameter  extension of the Birkhoff  ergodic theorem where   $[1,N]$ in (\ref{0101})  is replaced with    $[1,N_1]\times \cdots\times [1,N_d]$.  Recently, Bourgain, Mirek, Stein and Wright \cite{B4}  proved the multi-parameter ergodic theorem for  $T^{P(t)}$ with $t=(t_1,\cdots,t_d)\in [1,N_1]\times \cdots\times [1,N_d]$ in (\ref{0101}). To achieve this, they developed a multi-parameter-circle-method  associated with the non-cubic box $[1,N_1]\times \cdots [1,N_d]$. They estimated  the discrete averages $A_{N}f$ over rectangles $[1,N_1]\times \cdots \times [1,N_d]$ in the oscillation semi-norm. See also \cite{MSW} for related oscillation 
 semi-norm-estimates.  
 Moreover,  they obtained the multi-parameter   ergodic theorems involving  $T_1^{t_1}T_2^{t_2} T_3^{P(t_1,t_2)}$ under the assumption  that degrees of $P(t_1,t_2)$ with respect to $t_1$ and $t_2$ are both bigger than $1$ in \cite{B4}.
An unknown case in the study of the graph  $(t_1,t_2,P(t_1,t_2))$ in \cite{B4}  involves a polynomial expressed as:
  \begin{align}\label{chal1}
  P(t_1,t_2)= R_1(t_2)t_1^1+R_0(t_2).
  \end{align}
In    \cite{B4}, they suggest a challenging model case  $P(t_1,t_2)=t_2^2t_1^1$ (role of $t_1$ and $t_2$ is switched) to explain Conjecture 1.22 of \cite{B4}.  In this paper, we independently develop a
 multi-parameter-circle-method to address a multi-parameter extension of the discrete discrete Hilbert transform $H_N$ for the graph case $(t_1,t_2,P(t_1,t_2))$ for all polynomial $P$, including the case described in  (\ref{chal1}). 
To start, we consider the
 multi-parameter version of  Arkhipov and Oskolkov's theorem of  (\ref{11}) using the following basic setting:
\begin{itemize}
\item the phase $P(t)$ is replaced with   $P(t_1,t_2) $
\item the amplitude $1/t$ is replaced with  $1/(t_1t_2)$ 
\item the range $1\le |t|\le N$  is replaced with the lattice box $R(N_1, N_2)$ where  $$R(N_1, N_2):=\{t\in \mathbb{Z}^2: 1\le |t_1|\le N_1,|t_2|\le N_2\}\ \text{where $ (N_1,N_2)\in \mathbb{N}^2 $}.$$
\end{itemize}  
In particular,  for  the special case  $P(t)=a t_1t_2\in \mathbb{R}[t_1,t_2]$,
Garaev \cite{G} observed that  there exists   $a\in \mathbb{R}$ such that
\begin{align}\label{0kmm}
\lim_{N\rightarrow\infty}\left| \sum_{t\in R(N,N)}\frac{e^{2\pi ia t_1t_2}}{t_1t_2}\right|=\infty.
\end{align}
See also  Oskolkov  \cite{O}.  As an analogue of (\ref{11}) in the   multi-parameter setting, instead of considering $P\in  \mathbb{R}[t_1,t_2]$, we may ask for  conditions  on  vector polynomials  
$\mathcal{P}(t)=(P_1(t),\cdots,P_d(t))$ with $P_\nu(t)\in \mathbb{Z}[t_1,t_2] $, which determine   the following uniform boundedness: 
\begin{align}\label{0188}
\sup_{(N_1,N_2)\in \mathbb{N}^2,\xi\in \mathbb{R}^d}\left| \sum_{t\in R(N_1,N_2)}\frac{e^{2\pi i\xi\cdot \mathcal{P}(t)}}{t_1t_2}\right|\le C.
\end{align}
This generalizes a straightforward two-parameter extension of (\ref{11}), stated as $$ \sup_{N_1,N_2,\text{all coefficients of $P$}} \left|\sum_{t\in R(N_1,N_2)}\frac{e^{2\pi iP(t_1,t_2)}}{t_1t_2}\right|<C\ \text{where $P(t_1,t_2)\in \mathbb{R}[t_1,t_2]$}. $$ This sets up a framework for understanding the conditions under which such uniform boundedness holds in the multi-parameter setting.
In this paper, we address the graph case    $\mathcal{P}(t)=(t_1,t_2,P(t_1,t_2))$ of  (\ref{0188}), defined  as 
\begin{align}\label{12}
H^{\rm{discrete}}_{(N_1,N_2)}(\xi):=\sum_{t\in R(N_1,N_2)}  \frac{e^{-2\pi i ( \xi_1t_1+\xi_2t_2+\xi_3P(t_1,t_2))}}{t_1  t_2}\ \text{for}\ \xi=(\xi_1,\xi_2,\xi_3)\in\mathbb{R}^{3}.
\end{align} 
We aim to determine   the necessary and sufficient condition  for the  boundedness of the discrete multiplier $H^{\rm{discrete}}_{(N_1,N_2)}(\xi)$ uniformly in $\xi$ and $N_1,N_2$.

  \subsection{Main Results}
 Given a two variable polynomial  $P(t_1,t_2)=\sum_{\mathfrak{m}\in\Lambda(P)} c_{\mathfrak{m}} t^{\mathfrak{m}},$ 
 we define a backward Newton polyhedron associated with a  domain $\{t: |t_1|,|t_2|\ge 1\}$ by $$\mathcal{N}={\bf conv}\left(\bigcup_{ \mathfrak{m}\in \Lambda(P)} \mathfrak{m}-\mathbb{R}_+^2 \right)\ \text{
where ${\bf conv}(A)$ indicates the convex hull of   $A$.}$$
\begin{main}
Let $P(t_1,t_2)\in \mathbb{Z}[t_1,t_2]$. Then  there is a constant $C>0$  such that
\begin{align}\label{132}
&\sup_{(N_1,N_2)\in \mathbb{N}^2,\xi \in \mathbb{R}^3}|H^{\rm{discrete}}_{(N_1,N_2)}(\xi_1,\xi_2,\xi_3)|\le C\ \text{if and only if }\nonumber\\
&\qquad\qquad \text{every vertex of   $\mathcal{N}$ has at least one even component.}
\end{align}
Note that the constant $C>0 $ is independent of $N_1,N_2,\xi$, but may depend on the coefficients $c_{\mathfrak{m}}$  of $P(t) $.
Moreover, the above statement holds for $\xi=(0,0,\xi_3)$ as
\begin{align*}
&\sup_{(N_1,N_2)\in \mathbb{N}^2,\xi_3\in \mathbb{R}}|H^{\rm{discrete}}_{(N_1,N_2)}(0,0,\xi_3)|\le C\ \text{if and only if (\ref{132}) holds.} 
\end{align*}

Furthermore,
 if the above evenness condition of $\mathcal{N}$ of (\ref{132}) holds,  then we obtain that
 \begin{align}\label{15p}
 &	H^{\rm{discrete}}(\xi):=\lim_{N_1,N_2\rightarrow \infty}H^{\rm{discrete}}_{(N_1,N_2)}(\xi)\ \text{exists and is uniformly bounded in $\xi\in\mathbb{R}^{3}$. }
 \end{align}  
 \end{main}

\begin{remark}
Garaev's divergence result in (\ref{0kmm}) indicates that  the limit $H^{\rm{discrete}}(\xi)$ in (\ref{15p}) may not exist for some $\xi\in\mathbb{R}^3$, say $(0,0,\xi_3)$, if any vertex of $\mathcal{N}$ has components that are both odd.    
\end{remark}

\begin{definition}\label{def11}
Let   $P(t) \in \mathbb{Z}[t_1,t_2]$. For each $(N_1,N_2)\in \mathbb{N}^2$,  set the operator $f\in \mathcal{S}(\mathbb{R}^{3})\rightarrow {\bf H}^{\rm{discrete}}_{(N_1,N_2)}(f) $ by assigning the   summation to each $x=(x_1,x_2,x_3)\in \mathbb{Z}^{3}$,
\begin{align}\label{18p}
{\bf H}^{\rm{discrete}}_{(N_1,N_2)}(f)(x) :=\sum_{t\in R(N_1,N_2)}  \frac{f(x_1-t_1,x_2-t_2,x_3-P(t))}{t_1t_2}.
\end{align}
Then we use the Fourier inversion formula for $f\in \mathcal{S}(\mathbb{R}^{3})$ to have
$${\bf H}^{\rm{discrete}}_{(N_1,N_2)}(f)(x) = [H^{\rm{discrete}}_{(N_1,N_2)}]^{\vee}*f(x)$$
where  $H^{\rm{discrete}}_{(N_1,N_2)}$ is defined in (\ref{12}). Then we are able to  
  define the discrete double Hilbert transform ${\bf H}^{\rm{discrete}}(f)$ for all $f\in\mathcal{S}(\mathbb{R}^3)$ by taking its value at $x$ is given by
 $${\bf H}^{\rm{discrete}}(f)(x):=\lim_{N_1,N_2\rightarrow \infty} {\bf H}^{\rm{discrete}}_{(N_1,N_2)}(f)(x).$$
Here the pointwise convergence follows from the fact that  the RHS of (\ref{18p})   is rapidly decreasing in $|x_1-t_1|$ and $|x_2-t_2|$ for fixed $x_1,x_2$, regardless that $\mathcal{N}$  satisfies (\ref{132}) or not. 
\end{definition}
If every vertex  $(m_1,m_2)$ in $\mathcal{N}$ has at least one even component, then we shall see that the Fourier multiplier of ${\bf H}^{\rm{discrete}}$ is $H^{\rm{discrete}}$   in  (\ref{15p}), while the existence of a vertex $ (odd,odd)$ in $\mathcal{N}$ may not 
guarantees  that  $H^{\rm{discrete}}$ is the multiplier of ${\bf H}^{\rm{discrete}}$. So we state the theorem regarding ${\bf H}^{\rm{discrete}}$ under the evenness hypothesis:
\begin{main}\label{main} Let $P(t_1,t_2)\in \mathbb{Z}^2[t_1,t_2]$. Suppose that (\ref{132}) holds. Then for any
  $p\in (1,\infty)$, we have $C=C(p,P)$ depending on coefficients of  $P(t_1,t_2)$ 
 such that
 $$\|
 {\bf H}^{\rm{discrete}}f\|_{\ell^p(\mathbb{Z}^3)}\le C  \|
  f\|_{\ell^p(\mathbb{Z}^3)} \ \text{for all  $f\in \ell^p(\mathbb{Z}^3)$}.\ $$  
\end{main}

 \begin{remark}
In Theorem 1.16 of \cite{B4} concerning the multi-parameter ergodic theorem associated with   the graph  case $(t_1,t_2,P(t_1,t_2))$, they assumed that
 the degree of $P$ with respect to each variable $t_1$ and $t_2$ is greater than $1$.  So,  they suggest a model example of the phase $ [\xi_1 t_1+\xi_3t_1t_2^2]$ below Conjecture 1.22 of \cite{B4}, which breaks their degree assumption since $P(t_1,t_2)=t_1t_2^2$ has degree 1 with respect to $t_1$. In this paper, without this degree  constraint,  we successfully  handle the sum   $\sum \frac{e^{2\pi i( \xi_1 t_1+\xi_2t_2+\xi_3P(t_1,t_2))}}{t_1t_2}$. To accomplish this, we introduce the double arc method  in Section 3.2 and  apply it to the most challenging part of the estimation in Section \ref{Sec72}.
\end{remark}

 {\bf Organization}.
     In Section \ref{Sec2}, we define the Newton polyhedron   and  state the  results of the continuous multiple Hilbert transforms along polynomial surfaces. Next, we decompose  the multiplier $H^{\rm{discrete}}_{(N_1,N_2)}$ in terms of its dual faces. In Section \ref{Sec3}, 
   we introduce the two-parameter circle method and outline the proof of the main theorems. In Section \ref{Sec4}, we discuss the two-parameter Weyl sum for the minor arc and give a proof for the minor arc estimate. In Section \ref{Sec500}, we estimate the average Gauss sum and  utilize it for proving the major-minor arc estimate.     In Section \ref{Sec72},
 we treat the hard  case $P(t)=t_1R_1(t_2)+R_0(t_2)$.
   In Section \ref{Sec7}, we prove the estimate for major arcs ($\sharp$ and $\flat$). In Section \ref{Sec8}. 
    we prove the $\ell^p$ boundedness of ${\bf H}^{\rm{discrete}}$ by using the Ionescu--Wainger multiplier method.  In Section \ref{Sec9}, we give   a necessity proof.\\

{\bf Notation}.
 Given two scalars $a,b$,  write $a\lesssim b$  if $a\le Cb$ for some  $C>0$ depending only on the given polynomial $P$.   Notice the bounds involved in $ \lesssim$  are independent of $\xi\in\mathbb{R}^3$ and $j=(j_1,j_2)\in \mathbb{Z}_{\ge 0}^2$.  In additions,
  denote $0\le a\ll b$ if $a/b$ is a sufficiently small number compared with 1. Note that   our positive constants $\epsilon,\eta$ ($\epsilon,\eta\ll 1$) and    $c,C$ may be   different   line by line.\\

{\bf Acknowledgement}. 
The authors would like to express their deep gratitude to Professor Jim Wright for his valuable comments, which helped us address and correct certain flaws in the original draft of this paper.

\section{Newton polyhedron and Dual Face Decomposition}\label{Sec2} For a systematic study,  it is useful to know some terminology of convex geometry.
For further study, see Chapters 1 and 2 of \cite{F}, and Chapter 4 of \cite{K}.
\subsection{Newton Polyhedron}
\begin{definition}[Domains]\label{defi1}
Given a finite  subset $B\subset \mathbb{Q}^2$,   set the region 
\begin{align*} 
D_B:=\{x\in\mathbb{R}^2:   |x^{\mathfrak{b}}|\le 1\ \text{for all $\mathfrak{b}\in B$}\}.
\end{align*}
Here,   $x^{\mathfrak{b}}:=x^{b_1}_1 x^{b_2}_2$   for $\mathfrak{b}=(b_1,b_2)$ and $x_\nu^0:=1$ for   real $x_\nu $. Let ${\bf e}_1=(1,0),{\bf e}_2=(0,1)$. For instance,  the regions $ D_B$ are   $$ D_{\{{\bf e}_1, {\bf e}_2\}}=\{x:|x_\nu|\le 1\},\ \text{and}\ D_{\{-{\bf e}_1, -{\bf e}_2\}}=\{x:|x_\nu|\ge 1\}\ \text{and}\ D_{\{{\bf 0}\}}=\mathbb{R}^2.$$ 
  \end{definition}
\begin{definition}[$\rm{cone}(B)$ and its  dual  $\rm{cone}^{\vee}(B)$]\label{d22}
Given  $B\subset\mathbb{Q}^{2}$, we set      \begin{align*} 
\rm{cone}(B)&:= \left\{\sum_{\mathfrak{b}\in B}\alpha_{\mathfrak{b}}\mathfrak{b}\in \mathbb{R}^2: \alpha_{\mathfrak{b}}\ge 0 \right\} \ \text{and}\  
\rm{cone}^{\vee}(B) := \bigcap_{\mathfrak{b}\in\rm{cone}(B)} \{\mathfrak{q}\in \mathbb{R}^2:\langle\mathfrak{b},\mathfrak{q}\rangle\ge 0\}.\end{align*}  
\end{definition}
\noindent 
\begin{exam}\label{ex21}
Consider the model cones and their duals.
\begin{itemize}
\item
If $B=\{{\bf e}_1,{\bf e}_2\}$, then $\rm{cone}(B)=\mathbb{R}_+^2$ and $\rm{cone}^{\vee}(B)= \mathbb{R}_+^2$.
\item
If $B=\{-{\bf e}_1,-{\bf e}_2\}$,  then $\rm{cone}(B)=-\mathbb{R}_+^2$ and $ \rm{cone}^{\vee}(B)=-\mathbb{R}_+^2.$
\item If $B=\{{\bf 0}\}$, then $\rm{cone}(B)= \{{\bf 0}\}$ and $\rm{cone}^{\vee}(B)=\mathbb{R}^2$.
\end{itemize}
In this paper we shall  mainly work with  $D_B$ for $B=\{-{\bf e}_1,-{\bf e}_2\}$ and $\rm{cone}^{\vee}(B)=-\mathbb{R}_+^2$. 
\end{exam}

\begin{definition}
Let $P(x)$ be a polynomial  of $x\in \mathbb{R}^2$.    Associated with  a pair $(P,D_B)$,   we  define  the Newton polyhedron:
 \begin{align*}
  {\bf N}(P,D_B)&: ={\bf Ch}\left(\bigcup_{ \mathfrak{m}\in \Lambda(P)} \mathfrak{m}+ \rm{cone}(B)  \right)
\end{align*}
where $\Lambda(P)$ is the set of all exponents  of $P$ and ${\bf Ch}(A)$ is the convex hull of $A$.   
\end{definition}
\begin{exam}\label{ded281}
For $D_{\{{\bf e}_1,{\bf e}_2\}}=[-1,1]^2$, $ D_{\{-{\bf e}_1,-{\bf e}_2\}}=\{x: |x_\nu|\ge 1\}$ and $D_{\{{\bf 0}\}}= \mathbb{R}^2$, 
\begin{itemize} 
\item[(1)]
${\bf N}(P,D_{\{{\bf e}_1,{\bf e}_2\}}) ={\bf Ch}[\Lambda(P)+\rm{cone}({\bf e}_1,  {\bf e}_2)] $,
 \item[(2)]
${\bf N}(P, D_{\{-{\bf e}_1,-{\bf e}_2\}}) ={\bf Ch}\left[\Lambda(P)+
\rm{cone}(-{\bf e}_1,-{\bf e}_2)
\right]$,    denoted by $\mathcal{N}$ in Section 1.2.
\item[(3)]  ${\bf N}(P,
  D_{\{{\bf 0}\}}) ={\bf Ch}[\Lambda(P)]$.
  \end{itemize}  
\end{exam} 

\subsection{Continuous  Multiple Hilbert Transforms}
As a continuous version of the discrete multiplier $H^{\rm{discrete}}_{N_1,N_2}(\xi)$ for $\xi=(\xi_1,\xi_2,\xi_3)$, we define
\begin{align}\label{conour}
\mathcal{H}^{\Lambda(P)}_{N_1,N_2} (\xi_1,\xi_2,\xi_3)&:=\int_{1\le |x_1|\le N_1,1\le |x_2|\le N_2} e^{2\pi i( \xi_1x_1+\xi_2x_2+\xi_3 P(x_1,x_2))}  \frac{dx_1}{x_1} \frac{dx_2}{x_2}.
\end{align}
In 2000, Carbery, Wainger and Wright  \cite{CWW} proved that 
$$ \lim_{\epsilon_1,\epsilon_2\rightarrow 0}\int_{\epsilon_1<|t_1|<1}\int_{\epsilon_2<|t_2|<1}   e^{2\pi i\xi\cdot (t_1,t_{2}, P(t))}\frac{dt_1dt_2}{t_1 t_2}.$$
is bounded uniformly in $\xi$ if and only if every vertex in ${\bf N}(P,D_{\{{\bf e}_1,{\bf e}_2\}})$ has at least one even component.
Later, in 2008, Patel  \cite{P2}  established the global domain showing that the uniform boundedness in $\xi\in\mathbb{R}^3$ of
$$ \lim_{\epsilon_1,\epsilon_2\rightarrow 0}\int_{\epsilon_1<|t_1|<1/\epsilon_1}\int_{\epsilon_2<|t_2|<1/\epsilon_2}   e^{2\pi \xi\cdot (t_1,t_{2}, P(t))}\frac{dt_1dt_2}{t_1 t_2}\ \text{holds if and only if} $$
\begin{align}\label{22a}
 [\mathbb{F}\cap \Lambda(P)] \subset (2\mathbb{Z})\times \mathbb{Z}\ \text{or}\  \mathbb{Z}\times(2\mathbb{Z})\ \text{for $\mathbb{F}\in \mathcal{F}({\bf N}(P,D_{\{{\bf 0}\}}))  $ of $\text{rank}(\mathbb{F})\le 1$}
 \end{align}
 where $\text{rank}(\mathbb{F})$ is the number of linearly independent vectors in $\mathbb{F}$.
 Patel in \cite{P1} also expressed  the necessary and sufficient condition for the uniform boundedness in $\xi_3$ of
 $  \lim_{\epsilon_1,\epsilon_2\rightarrow 0}\int_{\epsilon_1<|t_1|<1}\int_{\epsilon_2<|t_2|<1}   e^{2\pi \xi_3P(t)}\frac{dt_1dt_2}{t_1 t_2} $. The condition  involves 
\begin{align}\label{23a}
\text{signs of coefficients $c_{\mathfrak{m}}$ of  $ t^{\mathfrak{m}}$ over all those vertices $\mathfrak{m}=(odd,odd)$.}
\end{align}
For  the  triple Hilbert transforms, we refer \cite{CWW2} and \cite{CHKY} in 2009.  In \cite{CWW2}, Carbey, Wainger and Wright observed that the  $L^p$ boundedness of   triple Hilbert transform along $(t,P(t))$  is not determined by  only its Newton polyhedron but also by coefficients of   $P$. In 2015,    one of the  authors  \cite{K}  outlined the geometric conditions regarding the sets $\Lambda_\nu$ of exponents in $P_\nu$ that determine  $$\sup_{N,\xi\in\mathbb{R}^d}\left|\int_{\prod\{t:|t_\nu|\le N_{\nu}\} } e^{2\pi i\xi\cdot (P_1(t),\cdots,P_d(t))}\frac{dt_1}{t_1}\cdots\frac{dt_n}{t_n}\right|\le C\ \text{for  $C=C(P_1,\cdots,P_d)>0$.}$$ 
\subsection{Results of  Continuous Multipliers $\mathcal{H}^{\Lambda(P)}_{N_1N_2}(\xi)$} Recall (\ref{conour}).
The  condition  (\ref{22a}) involves   edges $\mathbb{F}$ passing through the origin. However,  our  backward Newton polyhedron ${\bf N}(P,D_{\{-{\bf e}_1,-{\bf e}_2\}})$  in (\ref{132})   does not have such edges.
 \begin{theorem}[Continuous Multiplier]\label{th21}
Given $P\in \mathbb{R}[t_1,t_2]$, then we have
\begin{align}\label{j4p}
&\sup_{(N_1,N_2)\in \mathbb{N}^2,\xi\in \mathbb{R}^3}\left|\mathcal{H}^{\Lambda(P)}_{N_1N_2}(\xi)\right|\le C\ \text{if and only if    (\ref{132}) holds,}
 \\
&\quad \sup_{\xi\in \mathbb{R}^3}\left|\lim_{N_1,N_2\rightarrow \infty} \mathcal{H}^{\Lambda(P)}_{N_1N_2}(\xi) \right|\le C\ \text{if and only if      (\ref{132}) holds.  }\label{j4}
  \end{align}
Here, the above constants $C>0$   may depend on coefficients of $P$.
  \end{theorem}
Next, one can treat the   continuous case of $(0,0,\xi_3)$ on   $D_{\{-{\bf e}_1,-{\bf e}_2\}}$:
\begin{theorem}\label{th22}
Given a polynomial $P$, we have 
\begin{align}\label{j4q}
  \sup_{(N_1,N_2)\in \mathbb{N}^2,\xi_3\in \mathbb{R}}\left|  \mathcal{H}^{\Lambda(P)}_{N_1N_2}(0,0,\xi_3) \right|\le C\ \text{if and only if    (\ref{132}) holds.  }
  \end{align}
\end{theorem}
It is observed that the evenness condition (\ref{132}) in (\ref{j4p}), (\ref{j4}) and (\ref{j4q}) aligns with that in Main Theorem 1. Additionally, we note that the condition (\ref{132}) differs from (\ref{23a}), which pertains to the coefficients $c_{\mathfrak{m}}$ of $t^{\mathfrak{m}}$, as (\ref{23a}) does not lead to the cancellation effect in the global integral of (\ref{j4q}). We will present a concise proof of the sufficient part of Theorems \ref{th21} and \ref{th22} in Section \ref{Sec52}.

\subsection{Dual Face Decomposition}
\begin{definition}\label{n45}[${\bf N}(P,D_B)$ in terms of $\rm{cone}^{\vee}(B)$]
Given a nonzero $\mathfrak{q} $ in $ \mathbb{R}^2$ and  $r\in \mathbb{R}$, we set a   line $\pi_{\mathfrak{q},r}$ and its  upper-half-space $\pi^+_{\mathfrak{q},r}$:
$$\pi_{\mathfrak{q},r}=\{x\in\mathbb{R}^2: \mathfrak{q}\cdot x= r\}\ \text{and}\ \pi^+_{\mathfrak{q},r}=\{x\in\mathbb{R}^2: \mathfrak{q}\cdot x\ge r\}. $$
Then one can see that there exists a set $\{\mathfrak{q}_1,\cdots,\mathfrak{q}_M \}$  of nonzero vectors such that 
 \begin{align} \label{0hj}
  {\bf N}(P,D_B)&: =\bigcap_{i=1}^M \pi_{\mathfrak{q}_i,r_i}^+\ \text{and}\    \rm{cone}^{\vee}(B)=\rm{cone}(\{\mathfrak{q}_1,\cdots,\mathfrak{q}_M \}).
\end{align}
\end{definition}
\begin{proposition}
Let $B=\{-{\bf e}_1,-{\bf e}_2\}$ and consider the Newton polyhedron in (\ref{0hj}). Then
all faces  in ${\bf N}(P,D_B)$ and their dual faces are represented  in terms of $\{\pi_{\mathfrak{q}_i,r_i}\}_{i=1}^M$:
\begin{itemize}
\item Edges (facets):\ $\mathbb{F}_i=\pi_{\mathfrak{q}_i,r_i}\cap {\bf N}(P,D_B)$ and its dual edge  $\mathbb{F}_i^*=\rm{cone}(\mathfrak{q}_i)$
\item  Vertices:
$\mathbb{F}_{i,i+1}=\pi_{\mathfrak{q}_i,r_i}\cap  \pi_{\mathfrak{q}_{i+1},r_{i+1}}$ and its dual vertex: $\mathbb{F}_{i,i+1}^*=\rm{cone}(\mathfrak{q}_i,\mathfrak{q}_{i+1})$.
\end{itemize}
Moreover, it leads the dual-face-decomposition of the whole dual  $\rm{cone}^{\vee}(B)$:
\begin{align}\label{cd1}
\rm{cone}^{\vee}(B)=\bigcup_{i=1}^M \mathbb{F}_{i,i+1}^{*}.
\end{align}
Denote the set of all edges or vertices by $\mathcal{F}^1({\bf N}(P,D_B))$ or $\mathcal{F}^0({\bf N}(P,D_B))$ respectively.   
\end{proposition}
\begin{proof}
  For its formal proof, we refer  \cite{F}  and Proposition 4.3 of \cite{K}.
\end{proof}
Let $\mathbb{Z}_+$ denote the set of non-negative integers.  For $j=(j_1,j_2)\in \mathbb{Z}_+^2$,  we define     
\begin{align}\label{8b}
 H_{j}^{\rm{discrete}}(\xi_1,\xi_2,\xi_3):=\sum_{t_1\sim 2^{j_1},t_2\sim 2^{j_2}} e^{-2\pi i(\xi_1t_1+\xi_2t_2+\xi_3P(t_1,t_2))}\frac{1}{t_1}\frac{1}{t_2}.
\end{align}
Here $t_\nu\sim 2^{j_\nu}$ indicates $2^{j_\nu-1}<|t_\nu|\le 2^{j_\nu}$ for $\nu=1,2$. 
For sufficiency part of Main Theorem 1, we shall prove that for some $C>0$ independent of $\xi$ and $N=(N_1,N_2)$,
\begin{align}\label{au5}
\sum_{j\in \mathbb{Z}_+^2\ \text{and} \ 1\le 2^{j_1}\le N_1,1\le 2^{j_2}\le N_2}| H_{j}^{\rm{discrete}}(\xi_1,\xi_2,\xi_3)  |\le C.
\end{align}
This implies (i) uniform boundedness and (ii) pointwise convergence of 
$$ \lim_{N_1,N_2\rightarrow\infty} H_{(N_1,N_2)}^{\rm{discrete}}(\xi_1,\xi_2,\xi_3) \ \text{for every}\ (\xi_1,\xi_2,\xi_3)\in  \mathbb{R}^3.$$
\begin{lemma}\label{lemdc}[Dual Face Decomposition]
 Recall that $\mathbb{F}^*$ is a dual face of $\mathbb{F}\in \mathcal{F}({\bf N}(P,D_B))$ where $B=\{-{\bf e}_1,-{\bf e}_2\}$. Then it holds that 
\begin{align*}
LHS \ \text{in (\ref{au5})}\le  \sum_{\mathbb{F}\in\mathcal{F}^0({\bf N}(P,D_B))}
 \sum_{j\in\mathbb{Z}^2(\mathbb{F}) }\left|H_{j}^{\rm{discrete}}(\xi_1,\xi_2,\xi_3)\right|
\end{align*}
where
\begin{align}\label{au0}
\mathbb{Z}^2(\mathbb{F}) :=\{j\in (-\mathbb{F}^{*}) \cap \mathbb{Z}^2: 1\le 2^{j_1}\le N_1\ \text{and}\ 1\le 2^{j_2}\le N_2\}.
\end{align}
\end{lemma}
\begin{proof}
Let $B=\{-{\bf e}_1,-{\bf e}_2\}$.  In the second line of Example \ref{ded281}, we  observed  $ [0,\infty)^2 =-\rm{cone}^{\vee}(B). $
By this with   (\ref{cd1}), one can decompose the exponent set $ [0,\infty)^2\cap \mathbb{Z}^2$ of $j$'s:
$$ [0,\infty)^2\cap \mathbb{Z}^2=-\rm{cone}^{\vee}(B)\cap \mathbb{Z}^2=\bigcup_{\mathbb{F}\in \mathcal{F}^0({\bf N}(P,D_B))} (-\mathbb{F}^*)\cap \mathbb{Z}^2. $$
The restriction to $1\le 2^{j_\nu}\le N_\nu$ on both sides implies the result of the above lemma.
\end{proof}
Henceforth, to establish (\ref{au5}) it suffices to prove that
\begin{align}\label{au1}
\sum_{j\in \mathbb{Z}^2(\mathbb{F})}\left| H_{j}^{\rm{discrete}}(\xi_1,\xi_2,\xi_3)\right|\le C\ \text{for a fixed   $\mathbb{F}\in \mathcal{F}^0({\bf N}(P,D_B))$}
\end{align}
since there exists only finitely many vertices $\mathbb{F}$ in ${\bf N}(P,D_B)$.
We will apply the following dominance property for the phase function.\begin{lemma}  
Suppose that  $\mathbb{F}=\{\mathfrak{m}\}\in\mathcal{F}^0( {\bf N}(P,D_B))$  and  $j\in (-\mathbb{F}^{*})\cap \mathbb{Z}^2 $ where  $\mathbb{F}^*$ represents dual faces in (\ref{cd1}). Then  it holds that     
 \begin{align}
  2^{j\cdot \mathfrak{m}}&\ge  2^{j\cdot \mathfrak{n}}\ \text{for all $\mathfrak{n}\in {\bf N}(P,D_B)$}\ \text{and}\ 2^{j\cdot \mathfrak{m}}  \ge C  |P(2^{j_1}x_1,2^{j_2}x_2)|   \ \text{for $x_1,x_2\sim 1$} \label{jq501}.
\end{align}
where  $C$   relies on  the coefficients of $P$, but independent of  $j$ and $x$.   
 \end{lemma}
   \begin{proof}[Proof of (\ref{jq501})]
We express a vertex $\mathbb{F}\in \mathcal{F}^0( {\bf N}(P,D_B) )$ and its dual face as
   $$\mathbb{F}=\bigcap_{\nu=1}^{2} \pi_{\mathfrak{q}_\nu,r_\nu} \ \text{and}\  \mathbb{F}^{*} =\rm{cone}(\mathfrak{q}_1,\ \mathfrak{q}_{2}).$$
Let $\mathfrak{m}\in \mathbb{F}\cap \Lambda(P)$ and $\mathfrak{n}\in {\bf N}(P,D_B)$.  Then
$
 \mathfrak{q}_\nu\cdot (\mathfrak{n}-\mathfrak{m})\ge 0\ \text{for all $\nu=1,2$} $ since
 $\mathfrak{m}\in  \pi_{\mathfrak{q}_\nu,r_\nu}$ and $\mathfrak{n}\in \pi_{\mathfrak{q}_\nu,r_\nu}^+$.
As  $j=\alpha_1(-\mathfrak{q}_1)+ \alpha_{2}(-\mathfrak{q}_{2}) \in (-\mathbb{F}^{*}) \cap \mathbb{Z}^2$   for  $\alpha_1, \alpha_{2}\ge0$, one obtains  $ j\cdot (\mathfrak{n}-\mathfrak{m})\le 0$. So, $ 2^{ j\cdot \mathfrak{m} }\ge 2^{ j\cdot \mathfrak{n}}$ showing the first inequality of (\ref{jq501}). The second inequality follows  from $\Lambda(P)\subset {\bf N}(P,D_B)$ and the triangle inequality 
  \begin{align*} 
  |P({\bf 2}^{j}x)|\le  \sum_{\mathfrak{n}\in \Lambda(P)} |c_{\mathfrak{n}}|2^{j\cdot\mathfrak{n}}|x^{\mathfrak{n}}|\le  \max\{|c_{\mathfrak{n}}|\}  \sum_{\mathfrak{n}\in \Lambda(P)}  2^{j\cdot\mathfrak{n}}
  \end{align*}
  where $\Lambda(P)$ is finite.
  \end{proof}
 
\subsection{Reduction}\label{Sec25n}
To establish $\sum_{j \in \mathbb{Z}^2(\mathbb{F})}|H_{j}^{\text{discrete}}(\xi)|\le C$ as stated in (\ref{au1}), we outline the fundamental reductions as follows.
\begin{itemize}
\item[(1)]  
Without loss of generality, we may assume    $j\in\mathbb{Z}^2(\mathbb{F})$  with $j_1\ge j_2$ in (\ref{au0}). 
\item[(2)]   We may assume that the set $\Lambda(P)$ of exponents of $P(t_1,t_2)$ does not include either ${\bf e}_1$ or ${\bf e}_2$.  Viewing $P(t_1,t_2) $ as
$\sum R_m(t_2)t_1^m$, for each $s\ge 0$,   we define
\begin{align*}
\ \ \ \ \ \mathbb{Z}^{s}[t_1,t_2]=\{P(t_1,t_2)\in \mathbb{Z}[t_1,t_2]: \text{$P(t_1,t_2)=\sum_{m=0}^{m=s}R_m(t_2)t_1^m$ with $R_s(t_2)\not\equiv 0$}\}.
\end{align*}
This set consists of polynomials   $P(t)$  with degree $s$ in the first variable $t_1$.
If $P(t_1,t_2)\in\mathbb{Z}^0[t_1,t_2]$, then $P(t_1,t_2)=R_0(t_2)$ and
$H_{j}^{\rm{discrete}}(\xi)\equiv 0$. So,
by excluding    $\mathbb{Z}^0[t_1,t_2]$ from   $\mathbb{Z}[t_1,t_2]$ we partition \begin{align*}
 \mathbb{Z}[t_1,t_2]=\mathbb{Z}^1[t_1,t_2]\cup \mathbb{Z}^{\ge 2}[t_1,t_2]\ \text{where}\
\mathbb{Z}^{\ge 2}[t_1,t_2]=\bigcup_{s\ge 2}\mathbb{Z}^s[t_1,t_2]
\end{align*}
and handle $P(t)$ in $\mathbb{Z}^1[t_1,t_2]$ or $\mathbb{Z}^{\ge 2}[t_1,t_2]$ separately when proving (\ref{au1}). Thus,  our $P(t)$ contains a monomial $c_{mn}t_1^mt_2^n$ such that $m\ge 1$. By this
 combined with (\ref{jq501}), we can assume that
\begin{align*}
2^{j\cdot \mathfrak{m}}\ge 2^{j_1}\ \text{for all $j\in\mathbb{Z}^2(\mathbb{F})$ with $\mathfrak{m}=\mathbb{F}$.}
\end{align*}
 \item[(3)]

The single sum estimate $\sum_{j_1}|H_{j_1,j_2}^{\rm{discrete}}(\xi)| \le C$ holds uniformly in $j_2$ and $\xi$,  leading to the  one-parameter estimate
$$\sum_{t_2\sim 2^{j_2}}\frac{1}{|t_2|} \sum_{j_1}\left|\sum_{t_1\sim 2^{j_1}} \frac{e^{2\pi i(\xi_1t_1+\xi_2t_2+\xi_3P(t_1,t_2)})}{t_1}\right|\le C. $$ 
\item[(4)]  To estimate  the double sum of (\ref{au1}), we   restrict $2^{j_2}> C(P)$ as  follows
\begin{align*}
\sum_{j \in \mathbb{Z}^2_+}|H_{j}^{\rm{discrete}}(\xi)|&=\sum_{ 2^{j_2} >C(P)}\sum_{j_1\in \mathbb{Z}_+}|H_{j}^{\rm{discrete}}(\xi)|+\sum_{1\le 2^{j_2}\le C(P)}\sum_{j_1\in \mathbb{Z}_+}|H_{j}^{\rm{discrete}}(\xi)|  \\
&\le \sum_{2^{j_2} >C(P)}\sum_{j_1\in \mathbb{Z}_+}|H_{j}^{\rm{discrete}}(\xi)|+O(1)\ \text{by (3) above,}
\end{align*} 
where,  given $P(t)=\sum_{(m,n)\in\Lambda(P)} c_{mn}t_1^mt_2^n$, the constant $C(P)$ is defined by \begin{align}\label{0122}
	C (P):=  d^2\sum_{(m,n)\in\Lambda(P)} (|c_{mn}|+1)^{100}\max\left\{\left|\frac{c_{mn'}}{c_{mn}}\right|:  (m,n),(m,n')\in\Lambda(P)\right\}.
\end{align}
Write $P(t)=\sum R_m(t_2)t_1^m$ with $R_m(t_2)=\sum_n c_{mn}t_2^n$. Then, given non-constant $R_m(t_2)$, from the assumption $ 2^{j_2}\ge C(P)$, it follows that   for some $N\ge 1$, \begin{align}\label{0122m}
|R_m(t_2)|=|\sum_{n=0}^{n=N}c_{mn}t_2^n|\approx |c_{mN}2^{j_2N}| \ \text{for  $t_2\sim 2^{j_2}$}.\end{align}
 We  utilize this observation  for the estimate of  $j_1\ge 100dj_2$.
\item[(5)] Building on (4) above, we redefine the set of $j$ in (\ref{au0}) as:
$$ \mathbb{Z}^2(\mathbb{F})=\{j\in (-\mathbb{F}^*)\cap \mathbb{Z}^2:2^{j_1}\ge 2^{j_2}\ge C(P)\}. $$ 
For our estimate,  we partition $\mathbb{Z}^2(\mathbb{F})$ into  
$\mathbb{Z}^2_{\approx}(\mathbb{F})$ and $ \mathbb{Z}^2_{\gg}(\mathbb{F})$:
\begin{align*}
\mathbb{Z}^2_{\approx}(\mathbb{F})&:=\{j \in \mathbb{Z}^2(\mathbb{F}): j_1\le  100dj_2\},\\
  \mathbb{Z}^2_{\gg}(\mathbb{F})&:=\{j \in \mathbb{Z}^2(\mathbb{F}): j_1>  100dj_2\}.
  \end{align*}
\end{itemize}
 
\section{Two parameter Circle Method}\label{Sec3}
In this section, we introduce the frequency decomposition for our two-parameter-circle method and sketch the outline of our estimate.
\subsection{Dirichlet Approximation}
To compute the exponential sum $H_{j}^{\rm{discrete}}(\xi)$ in (\ref{12}), it suffices to  consider  $\xi\in [0,1)^3$ since  $P(t)\in \mathbb{Z}[t_1,t_2]$. To treat the graph case $\xi_1t_1+\xi_2t_2+\xi_3P(t)$,  we approximate the frequency $\xi_3 $ by $a_3/q$ by using Dirichlet approximation within the size $2^{-j\cdot \mathfrak{m}} 2^{j_1/10}$  whereas we use  uniform-sized  approximation  for 
 $\xi_1,\xi_2$   to  $a_1/q,a_2/q$  within the   size $1/(2q)$ until we establish the $\ell^2$ boundedness in Section 7.  However, for the $\ell^p$ estimate in  Section 8, we  will utilize   the   approximation to $\xi$ by $a/q$ with $a=(a_1,a_2,a_3)$ such that $\text{gcd}(a,q)=1$.

\begin{lemma}\label{lem001}
 Fix a vertex $\mathfrak{m}=(m_1, m_2)\in\mathbb{F}$.  Given $j\in \mathbb{Z}(\mathbb{F})$, we set   $ M_j:=2^{j\cdot \mathfrak{m}}2^{-\frac{1}{10}j_1}$ where $2^{j\cdot \mathfrak{m}}\ge 2^{j_1}$. Then given $M_j$ and  $\xi_3\in\mathbb{R}$,  there are $q\in \mathbb{Z} \cap [1,M_j]$ and $a_3\in \mathbb{Z}$:
\begin{align}\label{002}
\left|\xi_3-\frac{a_3}{q}\right|<\frac{1}{qM_j}\ \text{where}\ gcd(q,a_3)=1.
\end{align}
\end{lemma}
\begin{proof}
Apply the Dirichlet approximation to $\xi_3,M_j$ to obtain   $a_3,q$ satisfying (\ref{002}).
\end{proof}
\begin{definition}\label{df31}
We employ a smooth cutoff function  $\phi$ supported in $\{u\in\mathbb{R}^1:|u|\le 1\}$ for $\phi(u)\equiv 1$ in $|u|<\frac{1}{2}$ and $\phi^c=1-\phi$.  Next set
$\chi(x)=\phi(x/2)-\phi(x)$.
 Given $A\subset \mathbb{R}^n$ with $n=1,2$, we denote by $\chi_A$ the characteristic function on $A$.
\end{definition}
\begin{lemma}\label{le32}
Fix $j\in\mathbb{Z}(\mathbb{F})$.    Then   it holds that for every $\xi_3\in\mathbb{R}$,
\begin{align*}
\sum_{ 1\le q\le 2^{j_1/10}} \sum_{a_3: (a_3,q)=1} \phi\left(\frac{\xi_3-a_3/q }{2^{-j\cdot \mathfrak{m}}2^{j_1/10}}\right)&\le  1.
\end{align*}
If $1\le q\le 2^{j_1/10}$, then
\begin{align*}
\sum_{a_3:gcd(q,a_3)=1}\phi\left(\frac{\xi_3-a_3/q }{2^{-j\cdot \mathfrak{m}}2^{j_1/10}}\right)&\le\sum_{a_3:gcd(q,a_3)=1}\phi\left(\frac{\xi_3-a_3/q }{1/(10q^2)}\right) \le 1
\end{align*}
where the condition $\phi\left(\frac{\xi_3-a_3/q }{1/(10q^2)}\right)\ne 0$ for $q$ given by $q,q'$ and  $q<q'$, implies that  $5q<q'$. 
\end{lemma}
\begin{proof}
The first part follows from a small width $2^{-j\cdot\mathfrak{m}}2^{j_1/10}\le 2^{-9j_1/10}$ and a large difference $|a_3/q-a_3'/q|>1/(qq')\ge  2^{-j_1/5}$ for $a_3/q\ne a_3'/q'$. In the second part, the first inequality follows from $2^{-j\cdot \mathfrak{m}}2^{j_1/10}\le 1/(10q^2)$ and the second from $|a_3/q-a_3'/q|>1/q^2$. The lacunary property of $q$'s follows from $1/(qq')\le |\frac{a_3}{q}-\frac{a_3'}{q'}|\le \frac{1}{10q^2}+\frac{1}{(10q')^2}$.
\end{proof}
\begin{definition}  
Let $q$ be a positive integer.  Given  $ a_\nu\in \{1,\cdots,q\}$   with  $\nu=1,2$,  we denote by $\Gamma_{a_\nu/q}$ the characteristic functions  as $$\Gamma_{a_\nu/q}=\begin{cases}
 \chi_{ \left[\frac{a_\nu}{q}-\frac{1}{2q}, \frac{a_\nu}{q}+\frac{1}{2q}\right)}\ \text{if $a_\nu\in \{1,\cdots,q-1\}$ and}\\
 \chi_{\left[0,\frac{1}{2q}\right)\cup \left[1-\frac{1}{2q},1\right)}\ \text{if $a_\nu=q$.}
 \end{cases}
 $$ 
To decompose the first two frequency variables $(\xi_1,\xi_2) \in [0,1)^2$ of $\xi\in [0,1)^3$,  we use
\begin{align}\label{kse}
\chi_{[0,1)^2}(\xi_1,\xi_2)\equiv \sum_{a_1,a_2=1}^{q}\Gamma_{a_1/q}\left(\xi_1\right)\Gamma_{a_2/q}\left(\xi_2\right)
\end{align}
where we note that for each $(\xi_1,\xi_2)\in [0,1)^2$ and $q$, there is a unique lattice $(a_1,a_2)$:
\begin{align*} 
-\frac{1}{2q}\le 	 \xi_\nu-\frac{a_\nu}{q} <\frac{1}{2q}\   \ \text{where we also denote}\ \beta_\nu:=\xi_\nu-\frac{a_\nu}{q}(\nu=1,2).
\end{align*}  

\end{definition}
\subsection{Basic Decomposition}
In view of the first inequality in Lemma \ref{le32}, we are able to define the three cutoff functions:  \begin{align}\label{33}
\phi^{\rm{minor}}_j(\xi)&= 1-\sum_{1\le q\le 2^{j_1/10}} \sum_{a_3: (a_3,q)=1}\phi\left(\frac{\xi_3-a_3/q }{2^{-j\cdot \mathfrak{m}}2^{j_1/10}}\right)\nonumber\\
\phi^{\rm{major,minor}}_j(\xi)&= \sum_{2^{\eta j_2}\le q\le 2^{j_1/10}} \sum_{a_3: (a_3,q)=1}  \phi\left(\frac{\xi_3-a_3/q }{2^{-j\cdot \mathfrak{m}}2^{j_1/10}}\right)\\
\phi^{\rm{major}}_j(\xi)&= \sum_{ 1\le q\le 2^{\eta j_2}} \sum_{a_3: (a_3,q)=1} \phi\left(\frac{\xi_3-a_3/q }{2^{-j\cdot \mathfrak{m}}2^{j_1/10}}\right) \nonumber
\end{align} 
where we shall determine $\eta$ to be $10^{-3}$ later. From now, we refer to the supports of  $\phi^{\rm{minor}}_j,  \phi^{\rm{major,minor}}_j$ and $\phi^{\rm{major}}_j$ as   minor arc,  major-minor arc and  major arc, respectively. According to $\xi$ contained in the three arcs, we split $H_j^{\rm{discrete}}(\xi)$ as
\begin{align*}
H_j^{\rm{discrete}}(\xi)=H_j^{\rm{discrete}}(\xi) \phi^{\rm{minor}}_j(\xi)+H_j^{\rm{discrete}}(\xi) \phi^{\rm{major,minor}}_j(\xi)  +H_j^{\rm{discrete}}(\xi) \phi^{\rm{major}}_j(\xi).
\end{align*}
By multiplying  $\chi_{[0,1)^2}(\xi_1,\xi_2)= \sum_{a_1,a_2=1}^{q}\Gamma_{a_1/q}\left(\xi_1\right)\Gamma_{a_2/q}\left(\xi_2\right) $ of  (\ref{kse}),     we    rewrite  the summation over $a_3$ in (\ref{33}) as the summation over $a=(a_1,a_2,a_3)$:\begin{align}
\sum_{a_3: (a_3,q)=1}\phi\left(\frac{\xi_3-a_3/q }{2^{-j\cdot \mathfrak{m}}2^{j_1/10}}\right)&=\sum_{a: (a_3,q)=1}\phi_{j,q,a}(\xi)   \nonumber\\
\ \text{where}\ \phi_{j,q,a}(\xi) &:=\Gamma_{a_1/q}\left(\xi_1\right)\Gamma_{a_2/q}\left(\xi_2\right)\phi\left(\frac{\xi_3-a_3/q }{2^{-j\cdot \mathfrak{m}}2^{j_1/10}}\right). \label{36a}
\end{align}
To show (\ref{au1}), we shall show that for a fixed   $\mathbb{F}\in \mathcal{F}^0({\bf N}(P,D_B))$,
\begin{align*} 
\sum_{j\in \mathbb{Z}^2(\mathbb{F})}\left| H_{j}^{\rm{discrete}}(\xi)\right|\left( \phi^{\rm{minor}}_j(\xi)+ \phi^{\rm{major,minor}}_j(\xi) +\phi^{\rm{major}}_j(\xi) \right) \le C.
\end{align*}

\subsection{Sketch of Non-Major-Arc Estimate and Double Arc Method}
We shall prove
\begin{align}\label{nnmm}
\sum_{j_1:j\in \mathbb{Z}^2(\mathbb{F})}|H_j^{\rm{discrete}}(\xi) |
(\phi^{\rm{minor}}_j(\xi)+ \phi^{\rm{major,minor}}_j(\xi)) \le C2^{-cj_2}
\end{align}
in Sections 4 through 6.
For the minor arc case, we will make use of the Weyl-sum estimate. In 
 the major-minor arc case, we will apply the average Gauss-sum inequality expressed as \begin{align}\label{906}
\frac{1}{2^{j_2}}\sum_{t_2\sim 2^{j_2}}\left|\frac{1}{q}\sum_{t_1=1}^q e^{2\pi i\left(\frac{a_1}{q}t_1+\frac{a_3}{q}P(t_1,t_2)\right)}\right|=O(q^{-c}).\end{align} 
These estimates enables us to obtain  the main estimation in Sections 4 and 5:
\begin{align*}
&\sum_{j_1:j\in \mathbb{Z}^2(\mathbb{F})}|H_j^{\rm{discrete}}(\xi) (\phi^{\rm{minor}}_j(\xi)+ \phi^{\rm{major,minor}}_j(\xi)) |\le C2^{-cj_2} \ \text{for all $P\in \mathbb{Z}^{\ge 2}[t_1,t_2]$},\\
&\sum_{j_1:j\in \mathbb{Z}_{\approx}^2(\mathbb{F})}|H_j^{\rm{discrete}}(\xi) (\phi^{\rm{minor}}_j(\xi)+ \phi^{\rm{major,minor}}_j(\xi)) |\le C2^{-cj_2} \ \text{for all $P\in \mathbb{Z}^{1}[t_1,t_2]$}.
\end{align*}
Then there remains  the hard part: 
\begin{align}\label{kds}
\sum_{j_1:j\in \mathbb{Z}_{\gg}^2(\mathbb{F})}|H_j^{\rm{discrete}}(\xi) (\phi^{\rm{minor}}_j(\xi)+ \phi^{\rm{major,minor}}_j(\xi)) |\le C2^{-cj_2} \ \text{for $P\in \mathbb{Z}^1[t_1,t_2]$}.
\end{align}
This is the case which lacks both of the Weyl-sum estimate and the  average Gauss sum estimate (\ref{906}) (when $q\ge 2^{100dj_2}$). To overcome this,  
write $P\in \mathbb{Z}^{1}[t_1,t_2]$ as
\begin{align*} 
P(t)=t_1R_1(t_2)+R_0(t_2)
\end{align*}  
and  express the phase $\xi_1t_1+\xi_2t_2+\xi_3P(t)$ of $H_{j}^{\rm{discrete}}(\xi)$ as  the  polynomial of   $t_1$:
$$\xi(t_2)t_1+[\xi_2t_2+\xi_3R_0(t_2)]\ \text{where $\xi(t_2)=\xi_3R_1(t_2)+\xi_1$}.$$
We then use the  new Dirichlet arcs   to the coefficient $\xi(t_2)$ of $t_1$ so that
\begin{align}\label{memp}
\begin{cases}
\text{new arc}: |\xi(t_2)-\frac{a_j(t_2)}{q_j(t_2)}|<\frac{1}{q_j(t_2)2^{(1-\frac{1}{100})j_1}}\ \text{of}\ \xi(t_2) \\
\text{old arc}: |\xi_3-a_3/q|<\frac{1}{2^{(1-1/10)j_1}}\ \text{of}\ \xi_3.
\end{cases}
\end{align}
Subsequently, within the new arc, we observe that the exponential sum over $t_1$ exhibits a strong decay except when $q_j(t_2)=1$. This, together with  (\ref{memp}), allows us to determine
$$\sharp\{ t_2\sim 2^{j_2}\ \text{satisfying (\ref{memp}) and $q_j(t_2)=1$ for some $j_1$}\}\le\text{deg}(P),$$ which reduces the matters   to the one parameter case to obtain (\ref{kds}) in Section 6. \\
\subsection{Sketch of Major-Arc Estimate}
Based on   $|(\xi_3-a_3/q)2^{\mathfrak{m}\cdot j}/2^{\eta j_2}|$, we divide 
 $$H_j^{\rm{discrete}}(\xi) \phi^{\rm{major}}_j(\xi)=H_j^{\rm{discrete}}(\xi) \phi^{\rm{major,\sharp}}_j(\xi)+H_j^{\rm{discrete}}(\xi) \phi^{\rm{major,\flat}}_j(\xi)$$ where
\begin{align}\label{34p}
\phi^{\rm{major,\sharp}}_j(\xi)&= \sum_{ 1\le q\le 2^{\eta j_2}} \sum_{a_3: (a_3,q)=1}\phi_{j,q,a}(\xi)  \phi^c\left( \frac{ \left(\xi_3-\frac{a_3}{q}\right)2^{\mathfrak{m}\cdot j} }{2^{\eta j_2}} \right)\nonumber\\
\phi^{\rm{major,\flat}}_j(\xi)&= \sum_{ 1\le q\le 2^{\eta j_2}} \sum_{a_3: (a_3,q)=1}\phi_{j,q,a}(\xi) \phi\left( \frac{ \left(\xi_3-\frac{a_3}{q}\right)2^{\mathfrak{m}\cdot j} }{2^{\eta j_2}} \right).
\end{align}
We switch the above two discrete sums   to the related oscillatory integrals:
\begin{itemize}
\item[(1)]  Change $H_j^{\rm{discrete}}(\xi) \phi^{\rm{major,\sharp}}_j(\xi)$ to the Hilbert transforms  along curves and  exploit  a decay  $2^{-cj_2}$ arising from  the high frequency   $ (\xi_3-a_3/q) 2^{j\cdot \mathfrak{m}} \ge 2^{\eta j_2}$.
\item[(2)] Change $H_j^{\rm{discrete}}(\xi) \phi^{\rm{major,\flat}}_j(\xi)$ to the double Hilbert transforms    along surfaces if $ (\xi_3-a_3/q) 2^{j\cdot \mathfrak{m}} \le 2^{\eta j_2}$. This will be done  in Section \ref{Sec7}.
\end{itemize}

\begin{table} 
	\centering
	\caption{Layout of Two-Parameter Circle Method.}
	\label{t2}
	\begin{tabular}{c|c|c|c}
		\noalign{\smallskip}\noalign{\smallskip}
		\hline\hline
		Variation& Range of $q$ & Frequency $\beta_3:=\xi_3-\frac{a}{q} $ &Outcome \\
		\hline
		$\text{Major}^{\flat}$  & $ q<2^{\eta j_2} $ & Low $|\beta_3|2^{j\cdot \mathfrak{m}}<2^{\eta j_2} $ & $\int \frac{dt_1}{t_1}\frac{dt_2}{t_2} $\\
		\hline
		$\text{Major}^{\sharp}$&$q<2^{\eta j_2} $  &High $| \beta_3|2^{j\cdot \mathfrak{m}}\ge 2^{\eta j_2} $ & $\sum_{t_2}\frac{1}{t_2} \int \frac{1}{t_1}dt_1$ \\
		\hline
		Major-Minor & $2^{\eta j_2}\le q\le 2^{j_1/10} $ & $|\beta_3|2^{j\cdot \mathfrak{m}}<2^{j_1/10}$ &$\sum_{t_2}\frac{1}{t_2}\int \frac{1}{t_1}dt_1 $ \\
		\hline
		Minor &  $2^{j_1/10}  <q$ & $|\beta_3|2^{j\cdot \mathfrak{m}}<2^{j_1/10}$ &$\sum_{t_1,t_2} \frac{1}{t_1t_2} $\\
		\hline
		\hline
	\end{tabular}
\end{table} 
The scheme for  the above four arcs is listed in 
 Table \ref{t2}.

\subsection{Reduction to $t_1$-Integral and $t_2$-Sum}\label{Sec34}
 \begin{definition}
Define a smooth even cutoff function for each $\ell\ge 1$ in $ \mathbb{Z}$ as
\begin{align}\label{aria}
\chi_\ell(s)=\begin{cases} 1\ \text{if $2^{\ell-1}+1/2 \le |s|\le 2^{\ell}$}\\
0 \ \text{if $|s|\le  2^{\ell-1}$ or $ |s|\ge 2^{\ell}+1/2$}.
\end{cases} 
\end{align}
Note that $|\partial_s\chi_{\ell} |$ is supported on $\{s\in \mathbb{R}:|s|\in (2^{\ell-1},2^{\ell-1}+1/2)\cup (2^{\ell},2^{\ell}+1/2)\}$ and whose value is $\approx 1$.    
 We rewrite $H^{\rm{discrete}}_{j}(\xi)$ as  $$H^{\rm{discrete}}_{j}(\xi) =\sum_{t\in \mathbb{Z}^2} e^{-2\pi i ( \xi_1t_1+\xi_2t_2+\xi_3P(t_1,t_2))}\frac{\chi_{j_1}(t_1)\chi_{j_2}(t_2)}{t_1t_2}.$$ 
 \end{definition}
 Away from minor arc,    we  now represent  $H^{\rm{discrete}}_{j}(\xi)$ as the the $t_1$-integral and $t_2$-sum.
\begin{lemma}\label{lem002} 
 For   each fixed $t_2$, we  set
\begin{align}
\mathcal{H}_{j_1}^{t_2}(\beta_1,0,\beta_3)&:=\int e^{-2\pi i(\beta_1x_1+\beta_3 P(x_1,t_2))}\chi_{j_1}\left(x_1\right)\frac{dx_1}{x_1},\label{010}\\
S^{t_2}\left(\frac{(a_1,0,a_3)}{q}\right)&:=\frac{1}{q}\sum_{\ell_1=1}^q e^{-2\pi i \left(\frac{a_1}{q}\ell_1+\frac{a_3}{q}P(\ell_1,t_2)\right)}.\label{011}  
\end{align}
Suppose that  $1\le q\le 2^{j_1/10}$ and $\text{gcd}(q, a_3)=1$ and 
$$|\xi_1-a_1/q|\le 1/(2q)\ \text{and}\ | \xi_3-a_3/q|\le 2^{-j\cdot \mathfrak{m}}2^{j_1/10}. $$
  Then   $H_j^{\rm{discrete}}(\xi)$ in (\ref{8b}) is expressed as
\begin{align}
  H^{\rm{discrete}}_{j}(\xi)   =  \label{640} 
 \sum_{t_2\sim 2^{j_2}} \frac{e^{2\pi i\xi_2t_2}}{t_2}  S^{t_2}\left(\frac{(a_1,0,a_3)}{q}\right)\mathcal{H}_{j_1}^{t_2}(\beta_1,0,\beta_3) +O(2^{-8j_1/10})
\end{align}
where   $\xi_1= \beta_1+\frac{a_1}{q}, \xi_3= \beta_3+\frac{a_3}{q}.$
\end{lemma}

\begin{proof}
For each fixed $t_2$,  rewrite from (\ref{aria}),
 \begin{align}\label{008}
H^{\rm{discrete}}_{j}(\xi)= \sum_{t_2\sim 2^{j_2}} \frac{e^{2\pi i\xi_2t_2}}{t_2}\left(\sum_{t_1 \in \mathbb{Z}} e^{-2\pi i (\xi_1t_1+\xi_3 P(t_1,t_2))}\frac{\chi_{j_1}(t_1)}{t_1}\right)+O(2^{-j_1}).
 \end{align}
Change   variable $t_1=\mu_1q+\ell_1$ in the sum  over $t_1$ to write  
\begin{align}\label{0059}
&\sum_{t_1 \in \mathbb{Z}} e^{-2\pi i (\xi_1t_1+\xi_3 P(t_1,t_2))}\frac{\chi_{j_1}(t_1)}{t_1}\nonumber\\
&= \sum_{\ell_1=1}^q   e^{-2\pi i \left( \frac{a_1}{q}\ell_1+\frac{a_3}{q} P( \ell_1 ,t_2)\right)}     \sum_{\mu_1} e^{2\pi i \left(\beta_1(q\mu_1+\ell_1)+\beta_3 P((q\mu_1+\ell_1),t_2)\right)}\frac{\chi_{j_1}(q\mu_1+\ell_1)}{q\mu_1+\ell_1}.
\end{align}
Regard $\mu_1$ as continuous variable.  By (\ref{jq501}) with $|\beta_1|\le 1/(2q)$ and $|\beta_3|\le \frac{1}{2^{j\cdot \mathfrak{m}}2^{-j_1/10}}$    in the support condition $ \phi_{j,q,a}(\xi)  $ ($q \le 2^{j_1/10}$) in  the above,  one can compute
\begin{align*}
 |\partial_{\mu_1} \big(\beta_1(q\mu_1+\ell_1)+ \beta_3 P((q\mu_1+\ell_1),t_2)\big)|  
&\le \frac{1}{2}+ C \beta_3 q 2^{j\cdot \mathfrak{m}-j_1}  \le \frac{1}{2}+2^{-7j_1/10}.
\end{align*}
Thus under the condition $1\le q\le 2^{j_1/10}$, one can use the van der Corput theorem to replace the above   $\sum_{\mu_1}$ with $\int dx_1$. Next apply the change of variable $qx_1+\ell_1\rightarrow x_1$ in that integral as
\begin{align*}
 &\sum_{\mu_1} e^{-2\pi i \left(\beta_1 (q\mu_1+\ell_1)+ \beta_3 P(q\mu_1+\ell_1,t_2)\right)}\frac{\chi_{j_1}(q\mu_1+\ell_1)}{q\mu_1+\ell_1}\\
 & =\int e^{-2\pi i \left(\beta_1 (q\mu_1+\ell_1)+ \beta_3P( qx_1+\ell_1,t_2)\right)}\frac{\chi_{j_1}(q x_1+\ell_1)}{q x_1+\ell_1}dx_1+O( 2^{-  9j_1/10 })\\
 &=\frac{1}{q} \mathcal{H}^{t_2}_{j_1}(\beta_1,0,\beta_3) +O( 2^{-9 j_1/10 })  
\end{align*}
where the error  follows from
$\int \left|\frac{d}{dx_1} \left(\frac{\chi_{j_1}(q x_1+\ell_1)}{q x_1+\ell_1} \right)\right|dx=O( 2^{- 9 j_1/10 })$.  Thus, by invoking (\ref{010}),(\ref{011}) and (\ref{0059}), one has 
\begin{align*}
&\sum_{t_1 \in \mathbb{Z}} e^{-2\pi i (\xi_1t_1+\xi_3 P(t_1,t_2))}\frac{\chi_{j_1}(t_1)}{t_1}\\
&\qquad = \frac{1}{q}\sum_{\ell_1=1}^q   e^{-2\pi i \left(\frac{a_1}{q}\ell_1+  \frac{a_3}{q} P( \ell_1 ,t_2)\right)}   \mathcal{H}_{j_1}^{t_2}(\beta_1,0,\beta_3)  +O(2^{-8j_1/10}) \\
&\qquad = S^{t_2}\left(\frac{(a_1,0,a_3)}{q}\right)\mathcal{H}_{j_1}^{t_2}(\beta_1,0,\beta_3)+O(2^{-8j_1/10}).
\end{align*} 
Therefore, we insert this into (\ref{008}) to obtain  Lemma \ref{lem002}. 
\end{proof}

\section{Weyl Sum and Minor Arc Estimate}\label{Sec4}
\subsection{Two-Parameter Weyl Sum}
 \begin{proposition}\label{prop21}
Consider a polynomial $Q(t_1,t_2)  =\sum_{(m,n)\in \Lambda(Q)}  c_{mn}t_1^mt_2^n \in \mathbb{R}[t_1,t_2]$  with $\text{deg}(Q)=d$ and $\Lambda(Q)\not\subset \{0\}\times\mathbb{Z}$.
Suppose that there exists $ (m,n)\in \Lambda(Q) $ which is   a vertex $\mathbb{F}$ of $ {\bf N}(Q,D_{\{-{\bf e}_1,-{\bf e}_2\}})$,  with   $j\in (-\mathbb{F}^*)$  satisfying that $j_1\ge j_2\ge 1$, and there exists  a Dirichlet approximation pair $(q, a_{mn})$ to the coefficient $c_{mn}$ of $t_1^mt_2^n$ of $Q(t_1,t_2)$ and $\eta>0$:
\begin{align}\label{42fg}
|c_{mn}- \frac{a_{mn}}{q}|\le \frac{1}{q^2}\ \text{and}\ \text{gcd}(a_{mn},q)=1 \ \text{satisfying}
\ 2^{j_1\eta }\le q\le \frac{2^{j_1
		m}2^{j_2n}}{2^{j_1\eta}}. 
\end{align}
  Assume  that   $2^{j_1}\le 2^{100 d j_2}$.
Then there are $C,\delta$ that may depend on $d$ and  $\eta$ such that
 \begin{align}\label{023}
 \left|\sum_{t_1\sim 2^{j_1}}\sum_{t_2\sim 2^{j_2}}e^{2\pi i Q(t_1,t_2)}\right|\le C2^{-\delta j_1} 2^{j_1}2^{j_2}.
\end{align}
\end{proposition}
  \begin{remark}\label{rmk41}
Note that $t_1\sim 2^{j_1}$ and $t_2\sim 2^{j_2}$ in the above Weyl sum can be replaced with  $   [2^{j_{1-1}},2^{j_1}]\times [2^{j_{2}-1}, 2^{j_{2}}]$. Moreover, it can be replaced with $ [2^{j_{1}-1}, p]\times [2^{j_{2}-1}, q]$ for any $p\in [2^{j_{1}-1},2^{j_1}]$ and $q\in [2^{j_{2}-1},2^{j_2}]$.
\end{remark} 
\begin{remark}\label{rmk42}
Without comparability condition $2^{j_2}\le 2^{j_1}\le 2^{Kj_2}$ for some   $K>1$, the  decay estimate (\ref{023}) can fail. Take $c_{10}=2^{2j_2-j_1/2}$, $c_{12}=-2^{-j_1/2}$ and set $j_1\gg j_2$. Then $$\sum_{t_1\sim 2^{j_1},t_2\sim 2^{j_2}} e^{2\pi i (c_{10}t_1+c_{12}t_1t_2^2)}\approx \sum_{t_2=\pm 2^{j_2}} \sum_{2^{j_1-1}< |t_1|\le 2^{j_1}} e^{2\pi i (c_{10}+c_{12}t_2^2)t_1}=2^{j_1}\gg 2^{-\delta j_1} 2^{j_1}2^{j_2}.$$
\end{remark}
\subsection{Proof of  (\ref{023})}\label{Sec43}
To prove (\ref{023}), it suffices to consider the summation
over only positive $t_1,t_2$.
Let $(m,n)$ be a vertex of $  {\bf N}(Q,D_{\{-{\bf e}_1,-{\bf e}_2\}})$ and $j\in(-\mathbb{F}^*)$. Let $m=0$ such that $(0,n)$ is a vertex $\mathfrak{m}$ of $  {\bf N}(Q,D_{\{-{\bf e}_1,-{\bf e}_2\}})$. Then there is no element $(m',n')$ in $  {\bf N}(Q,D_{\{-{\bf e}_1,-{\bf e}_2\}})\setminus (0,n)$ with $n'\ge n$. To see this, observe that such a vector $(m',n')$  would satisfy  $(m',n')\cdot j>(0,n)\cdot j$ which is a contradiction to the first inequality of (\ref{jq501}).  Thus one can express $Q(t_1,t_2)=c_{0n}t_2^n +S_{n-1}(t_1)t_2^{n-1}+\cdots+S_1(t_1)t_1+S_0(t_1)$. Fix $t_1$ and apply the approximation hypothesis (\ref{42fg}) of the coefficient $c_{0n}$  to the one-parameter-Weyl sum result. Then one has
$\sum_{t_2\sim 2^{j_2}} e^{2\pi i Q(t_1,t_2)} =O(2^{j_2}2^{-\delta' j_2})$. This with  $2^{j_1}\le 2^{100dj_2}$ implies (\ref{023}).  So, we assume that $m\ge 1$.
Let $(m,n)=(1,0)$. Then one can see that $Q(t_1,t_2)$ is of the form $c_{10}t_1+R_0(t_2)$. For this case, use $\sum_{\ell=1}^q e^{2\pi i  \frac{a}{q}  \ell}=0$ and (\ref{42fg}) such that $c_{10}=\frac{a}{q}+\beta$ with $|\beta|\le 1/q^2$ and $2^{j_1\eta}<q<2^{j_1-j_1\eta}$ to have 
\begin{align}\label{43k}
 | \sum_{t_1\sim 2^{j_1}} e^{2\pi i c_{10}t_1}|&\le  q+\sum_{\mu=1}^{[\frac{2^{j_1}}{q}]}|\sum_{\ell=1}^q e^{2\pi i (\frac{a}{q}+\beta)\mu q} (e^{2\pi i (\frac{a}{q}+\beta) \ell}-e^{2\pi i  \frac{a}{q}  \ell})|\nonumber \\
&=q+O(\frac{2^{j_1}}{q})=O(2^{j_1} 2^{-  j_1\eta}).
\end{align}
Hence, it suffices to consider the vertices $(m,n)$ with $m\ge 1$ and $n\ge 1$.
  Define $D_{\ell}^1Q(t_1,t_2):=Q(t_{1}+\ell,t_{2})-Q(t_1,t_2)$ and $D_{r}^2Q(t_1,t_2):=Q(t_{1},t_{2}+r)-Q(t_1,t_2)$. Next, denote $D_{r_2}^2D_{r_1}^2=D^2_{r_1r_2}$. By switching the order of summations, observe that
	\begin{align}\label{0911}
		\left|\sum_{t_2\in [a,b]\cap \mathbb{Z}} e^{2\pi i Q(t_1,t_2)}\right|^2=\sum_{r\in [a-b,b-a]\cap \mathbb{Z}} \sum_{t_2\in [a-r,  b-r]\cap[a,b]\cap \mathbb{Z}}e^{2\pi i D_r^2Q(t_1,t_2)}.
	\end{align}
From now on, we omit $\mathbb{Z}$ in the subscript of the summation.
	For  $s\in \mathbb{N}$, let $L_{s,r}=[2^{s-1}-r,2^{s}-r]\cap[2^{s-1},2^{s}]\cap \mathbb{Z}$.  Then we rewrite the sum of (\ref{023}) as
\begin{align}\label{0302}
W:=\left|\sum_{t_1\sim 2^{j_1},t_2\sim 2^{j_2}}e^{2\pi i Q(t_1,t_2)}\right|=\left|\sum_{t_1\in L_{j_1,0}}\left(\sum_{t_2\in L_{j_2,0}} e^{2\pi i Q(t_1,t_2)}\right)\right|.
\end{align}

{\bf Differencing in $t_2$} Given $s\in \mathbb{N}$,  set $I_{s}=[-2^{s-1},2^{s-1}]\cap \mathbb{Z}$ and denote their   products by: $$I_{s}^i=\underbrace{I_{s}\times \cdots \times I_{s}}_{\text{$i$-factors}}\ \text{where}\ I_s^0=\{0\}.
$$
By repeating  the Schwarz's inequality for the summations $\sum_{t_1,(r_1,\cdots,r_i)}$   and  the differencing  (\ref{0911}), we  estimate $W$ of (\ref{0302})  as
	\begin{align*}
		W=&\left|\sum_{t_1\in L_{j_1,0}}\left(\sum_{t_2\in L_{j_2,0}} e^{2\pi i Q(t_1,t_2)}\right)\right|\le |L_{j_1,0}|^{1/2}\left|\sum_{t_1\in L_{j_1,0}}\sum_{r_1\in I_{j_2}}\left( \sum_{t_2\in L_{j_2,r_1}}e^{2\pi i D^2_{r_1}Q(t_1,t_2)}\right)\right|^{1/2}\\
		&\le |L_{j_1,0}|^{1/2}|L_{j_1,0}\times I_{j_2}|^{1/2^2}\left| \sum_{t_{1}\in L_{j_1,0}}\sum_{(r_1,r_2)\in I_{j_2}^2} \left(\sum_{t_2\in L_{j_2,r_1+r_2}}e^{2\pi i  D^2_{r_1r_2}Q(t_1,t_2)}\right)\right|^{1/2^2}\\
		&\le \cdots\le \prod_{i=1}^{n}|L_{j_1,0}\times I_{j_2}^{i-1}|^{1/2^{i}}\times |U|^{1/2^n} 
\end{align*}
where $U$ is given by
$$
 U :=  \sum_{t_{1} \in L_{j_1,0}}\sum_{(r_1,\cdots,r_n) \in I_{j_2}^n} \left(\sum_{t_2\in L_{j_2,r_1+r_2+\cdots+r_n}}e^{2\pi i  D^2_{r_1\cdots r_n}P(t_1,t_2)}\right).
$$
{\bf Differencing in $t_1$}. 
By changing the order of the summation $\sum_{t_1}\sum_{t_2}$ in $U$ above,   	\begin{align*}
		U = \sum_{(r_1,\cdots,r_n,t_{2}) \in \mathcal{I}^n_{j_2}} \left(\sum_{t_{1} \in L_{j_1,0}}e^{2\pi i  D^2_{r_1\cdots r_n}Q(t_1,t_2)}\right) 
	\end{align*}
where $\mathcal{I}^n_{j_2}=\{(r_1,\cdots,r_n,t_{2}): (r_1,\cdots,r_n)\in I^n_{j_2}\ \text{and}\ t_2\in L_{j_2,r_1+r_2+\cdots+r_n}\}$. 
Repeat the Schwarz's inequality for $\sum_{(r_1,\cdots,r_n,t_{2}),(\ell_1,\cdots,\ell_i)}$ and  differencing   $|\sum_{t_1}|^2$  like (\ref{0911}) to get
	\begin{align*}
		|U|&\le | \mathcal{I}^n_{j_2}|^{1/2}  \left|\sum_{(r_1,\cdots,r_n,t_{2}) \in \mathcal{I}^n_{j_2}}\sum_{\ell_1\in I_{j_1}} \left(\sum_{t_1\in L_{j_1,\ell_1}}e^{2\pi i  D_{\ell_1}^1D^2_{r_1\cdots r_n}Q(t_1,t_2)}\right)\right|^{1/2}\\
		& \le \cdots\le \prod_{i=1}^{m-1} |\mathcal{I}^n_{j_2} \times I_{j_1}^{i-1}|^{1/2^{i}}|V_{mn}|^{1/2^{m-1}}.
	\end{align*} 
Here, the main term $V_{mn}$ is given by
	\begin{align*}
		V_{mn} &:= \sum_{(r_1,\cdots,r_n,t_{2}),(\ell_1,\cdots,\ell_{m-1})\in  \mathcal{I}^n_{j_2}\times I_{j_1}^{m-1}}   \left(\sum_{t_1\in L_{j_1,\ell_1+\ell_2+\cdots+\ell_{m-1}}}e^{2\pi i  D^1_{\ell_1\cdots\ell_{m-1}}D^2_{r_1\cdots r_n}Q(t_1,t_2)}\right)
	\end{align*}
where for the case $m=1$,
	$$V_{1n}=\sum_{(r_1,\cdots,r_n,t_{2}) \in  \mathcal{I}^n_{j_2} }   \left(\sum_{t_1\in L_{j_1,0}}e^{2\pi i   D^2_{r_1\cdots r_n}Q(t_1,t_2)}\right)=U.$$
	Therefore, by combining the above differencings in $t_1$ and $t_2$,
	\begin{align*}
		&W\le   \prod_{i=1}^{n}|L_{j_1,0}\times I_{j_2}^{i-1}|^{1/2^{i}} \times \left(   \prod_{i=1}^{m-1}|\mathcal{I}^n_{j_2} \times I_{j_1}^{i-1}|^{1/2^{i}}\right)^{1/2^n} |V_{mn}|^{1/2^{m+n-1}}. 
	\end{align*}
  By using $\sum_{i=1}^{n}(i-1)/2^{i}=1-(n+1)/2^{n}$,   
	\begin{align*}
	&	\prod_{i=1}^{n}|L_{j_1,0}\times I_{j_2}^{i-1}|^{1/2^{i}} \times \left(   \prod_{i=1}^{m-1}|\mathcal{I}^n_{j_2} \times I_{j_1}^{i-1}|^{1/2^{i}}\right)^{1/2^n} 
		\le   C2^{j_1(1-\frac{m}{2^{m+n-1}})} 2^{j_2(1-\frac{n+1}{2^{m+n-1}})}.  
	\end{align*}
Then we have 
\begin{align}
W\le C|V_{mn}|^{1/2^{m+n-1}}2^{j_1(1-\frac{m}{2^{m+n-1}})} 2^{j_2(1-\frac{n+1}{2^{m+n-1}})} .\label{me1}
\end{align}
	By the geometry of the vertex $(m,n)$ in $ {\bf N}(Q,D_B) $,    we have  $M:=m!(n-1)!$ so  that
	\begin{align}\label{9sd}
	D^1_{\ell_1\cdots\ell_{m-1}}D^2_{r_1\cdots r_n}Q(t_1,t_2)=c_{mn} M( \ell_1\cdots\ell_{m-1}r_1\cdots r_{n} )t_1.
	\end{align}
Set  $\mathcal{R}(h):=\{(r_1,\cdots,r_n),(\ell_1,\cdots,\ell_{m-1})\in I^n_{j_2}\times I_{j_1}^{m-1}:h=\ell_1\cdots\ell_{m-1}r_1\cdots r_{n} M \}$ where $|h|\le M 2^{(m-1)j_1+nj_2}$.  Note that the   $\mathcal{R}(0)$ occurs when one of $r_1,\cdots,r_n$ or $\ell_1,\cdots,\ell_{m-1}$ is zero, which gives 
a gain of $2^{-j_1}$ or $2^{-j_2}$ when counting the number of $r$ and $\ell$. So it suffice to treat $\mathcal{R}(h)$ for $h\ne 0$.  According to $h$, we split the sum $V_{mn}$ as
	\begin{align*}
		&|V_{mn}| \le \sum_{t_{2}\in I_{j_2}} \sum_{h=-M2^{(m-1)j_1+nj_2}}^{ M2^{(m-1)j_1+nj_2} }  \sum_{(r_1,\cdots,r_n),(\ell_1,\cdots,\ell_{m-1})\in\mathcal{R}(h)} \left|\sum_{t_1\in L_{j_1,\ell_1+\ell_2+\cdots+\ell_{m-1}}}e^{2\pi i c_{mn}h t_1}\right|.
		\end{align*}
	As $|L_{j_1,\ell_1+\ell_2+\cdots+\ell_{m-1}}|\le 2^{j_1}$ with $c_{mn}=\frac{a_{mn}}{q}+\beta_{mn} $, there exists a $w_0$ such that	 
		\begin{align*}
		|V_{mn}|&\le  2^{j_2}  \max\left\{\frac{2M2^{(m-1)j_1+nj_2}}{q},1\right\} \sum_{ w_0q<h\le w_0q+q}|\mathcal{R}(h)|\min\left\{2^{j_1}, \frac{1}{  \| (\frac{a_{mn}}{q}+\beta_{mn})h\|} \right\}
\end{align*}
where    $\|x\|=\text{dist}(x,\mathbb{Z})$. 
Note $ |\mathcal{R}(h)|=O(h^{\epsilon'})=O(2^{j_1\epsilon/2})$ for  $\epsilon'\le \frac{\epsilon}{2(n+m-1)} \ll \eta$ where we take $\epsilon =\eta/(100d)$. By   $h\rightarrow qw_0+h$ for $1\le h\le q$ in the above, one can use the fact $c=\omega_0q\beta_{mn}$ is fixed and $|\beta_{mn}h|\le 1/q$ due to $|\beta_{mn}|\le 1/q^2$ to obtain the first two inequalities below
\begin{align}\label{as123}
		|V_{mn}|& \le   2^{j_2}   2^{\epsilon j_1/2} \max\left\{\frac{2M2^{(m-1)j_1+nj_2}}{q},1\right\} \sum_{ 1<h\le  q} \min\left\{2^{j_1}, \frac{1}{  \|c+ \frac{a_{mn}}{q}h+\beta_{mn}h\|} \right\}\notag\\ 
		    &\le C 2^{j_2}  2^{\epsilon j_1/2} \max\left\{\frac{2M2^{(m-1)j_1+nj_2}}{q},1\right\} \cdot (2^{j_1}+q\log q)\\
		    &\le C    2^{\epsilon j_1/2} \left( 2^{j_1m+j_2(n+1)} 2^{-j_1\eta/2} +2^{j_1+j_2}\right)\nonumber
	\end{align}
	since $2^{\eta j_1}<q<2^{j_1m+j_2n-\eta j_1}$.   By (\ref{as123}) and (\ref{me1}),  	\begin{align*}
		W&\le  C 2^{\epsilon  j_1/2^{m+n} } \left( 2^{j_1m+j_2(n+1)} 2^{-j_1\eta/2} +2^{j_1+j_2}  \right)^{1/2^{m+n-1}}  \times 2^{j_1(1-\frac{m}{2^{m+n-1}})} 2^{j_2(1-\frac{n+1}{2^{m+n-1}})}\\
		& \le C 2^{\epsilon  j_1/2^{m+n} }  \left(2^{ -\frac{\eta j_1}{ 2^{m+n} }  }   2^{j_1+j_2}   +    2^{j_1(1-m)/2^{m+n-1}} 2^{-j_2n/2^{m+n-1}}    2^{j_1+j_2}    \right)
		\end{align*}
		which is bounded by $2^{-cj_1}2^{j_1+j_2}$ for the case $m\ge 2$.		For the case $m=1$, one can see that the last line is bounded by $2^{-cj_1}2^{j_1+j_2}$ with		
		 $c=\frac{\eta}{200d\times 2^{1+n}} $ thanks to $2^{j_1}\le 2^{100 d   j_2}$ and $\epsilon=\eta/(100d)$.	
		 Therefore, we proved  (\ref{023}) for the case $2^{j_1}\le 2^{100 d   j_2}$.

\subsection{Minor Arc Estimate} 
Recall $ \mathbb{Z}^2(\mathbb{F})=\{j\in (-\mathbb{F}^*)\cap \mathbb{Z}^2:2^{j_1}\ge 2^{j_2}\ge C(P)\}$   in (\ref{0122}).  We claim that 
\begin{proposition}\label{minor arc}
Suppose that
(i) $P\in \mathbb{Z}^{\ge 2}[t_1,t_2]$ and $j\in \mathbb{Z}^2(\mathbb{F})$, or (ii) $P\in \mathbb{Z}^{\le 1}[t_1,t_2]$ and
 $j\in \mathbb{Z}^2_{\approx}(\mathbb{F})$.
Then one can find  $C>0$ and $c>0$   independent of $\xi $ such that
\begin{align}\label{0053}
&  |H^{\rm{discrete}}_{j}(\xi_1,\xi_2,\xi_3) |\phi^{\rm{minor}}_j(\xi)\le C2^{-c j_1}. 
\end{align}
 This implies that
  \begin{align}
 \sum_{j_1:j\in \mathbb{Z}^2(\mathbb{F})}  |H^{\rm{discrete}}_{j}(\xi_1,\xi_2,\xi_3) |\phi^{\rm{minor}}_j(\xi)&\le C2^{-cj_2}\ \text{if $P\in \mathbb{Z}^{\ge 2}[t_1,t_2]$},\nonumber\\
  \sum_{j_1:j\in \mathbb{Z}_{\approx}^2(\mathbb{F})}  |H^{\rm{discrete}}_{j}(\xi_1,\xi_2,\xi_3) |\phi^{\rm{minor}}_j(\xi)&\le C2^{-cj_2}\ \text{if $P\in \mathbb{Z}^{1}[t_1,t_2]$}. \label{414b}\end{align}
  \end{proposition}
   
  \begin{lemma}\label{4013}
If (i) or (ii) of Proposition \ref{minor arc} holds, then there are $C,c>0$ such that
\begin{align}\label{p44}
\phi^{\rm{minor}}_j(\xi)|\sum_{t_1\sim 2^{j_1}}\sum_{t_2\sim 2^{j_2}} e^{2\pi i (\xi_1t_1+\xi_2t_2+\xi_3 P(t_1,t_2))}|\le C2^{j_1+j_2} 2^{-cj_1}.
\end{align}
\end{lemma}
\begin{proof}[Proof of (\ref{p44})]
Let $\mathfrak{m}=(m,n)=\mathbb{F}$ be a vertex of ${\bf N}(P,D_B)$. Then if  $j\in(-\mathbb{F}^*)$, it holds that $2^{j\cdot \mathfrak{m}}\ge 2^{j_1}$ since $\Lambda(P)$ contains the point whose first component is $\ge 1$. 
By the  minor arc condition $\phi^{\rm{minor}}_j(\xi)\ne 0$ and by the Dirichlet approximation (\ref{002}),  there are  $q_3=q_3(j,\xi_3)$ and $a_3=a_3(j,\xi_3)$ such that 
\begin{align}\label{419p}
\left|\xi_3-\frac{a_3}{q_3}\right|<\frac{1}{q_32^{j\cdot \mathfrak{m}}2^{-\frac{1}{10}j_1}}\le \frac{1}{q_3^2}\ \text{for}\ gcd(q_3,a_3)=1 \ \text{and}\   2^{j_1/10}\le q_3\le 2^{Cj_1}
\end{align}
where we write $2^{Cj_1}=2^{j\cdot \mathfrak{m}}2^{-\frac{1}{10}j_1}$. Set  $Q(t_1,t_2)=\xi_1t_1+\xi_2t_2+\xi_3 P(t_1,t_2)$.Then by $\mathfrak{m}\ne (1,0),(0,1)$ in the reduction (2) in Section 2.5,   one can see that $\mathfrak{m}$ is a vertex of ${\bf N}(Q,D_B)$
with $c_{\mathfrak{m}}\xi_3$ being a coefficient of the monomial of the exponent $\mathfrak{m}$ in $Q(t_1,t_2)$.  Then from $2^{j_1}\ge 2^{j_2}\ge |c_{\mathfrak{m}}|^{100}$ in (\ref{0122}) we have $|c_{\mathfrak{m}}|\le 2^{j_1/100}$ . By this and (\ref{419p}),
one can use Lemma  1 of \cite{SW2} to find   $\tilde{a}_3/\tilde{q}_3$ to approximate the coefficient $\xi_3c_{\mathfrak{m}}$ of the monomial $t^{\mathfrak{m}}$ of $Q(t_1,t_2)$
$$\left| \xi_3 c_{\mathfrak{m}}-\frac{\tilde{a}_3}{\tilde{q_3}}\right|\le   \frac{1}{\tilde{q}_3^2}
\ \text{where}\ \text{gcd}(\tilde{a}_3,\tilde{q}_3)=1\ \text{and}\ \frac{2^{j_1(1/10-1/100)}}{2}\le |\tilde{q}_3|\le 2 2^{j_1(C-1/100)}.$$
{\bf  Case  1}.  Let $P\in \mathbb{Z}^{1}[t_1,t_2]$ and
 $2^{j_1}<2^{100dj_2}$. Then apply Proposition \ref{prop21} for the phase polynomial $Q(t_1,t_2)$,  whose coefficient $\xi_3 c_{\mathfrak{m}}$ of $t^{\mathfrak{m}}$ has the   good approximation above, to obtain  (\ref{p44}).\\
{\bf Case 2}.  Let $P\in \mathbb{Z}^{\ge 2}[t_1,t_2]$. If
 $2^{j_1}<2^{100dj_2}$, apply  Proposition \ref{prop21} to obtain the desired bound as above.  There remains the case $2^{j_1}\ge 2^{100d j_2}$.   One can write $P(t)=\sum_{i=0}^M R_i(t_2)t_1^i$ with $M\ge 2$.  Fix $i=m\ge 2$. Then from  (\ref{0122m}), one observes that
\begin{align*}
R_m(t_2)=\sum c_{mn}t_2^n \approx c_{mN}t_2^N\ \text{non-vanishing   for all $t_2\sim 2^{j_2}$}
\end{align*}
which forms a coefficient $\xi_3R_m(t_2)$ of $t_1^m$ in the  polynomial $\xi_1t_1+\xi_2t_2+\xi_3P(t)$ of $t_1$: $$Q_{t_2}(t_1)=\sum_{i=2}^{M} \xi_3 R_i(t_2) t_1^i+( \xi_1+\xi_3 R_1(t_2))t_1^1+(\xi_2t_2+\xi_3R_0(t_2))t_1^0$$
and   $|R_m(t_2)|\le  2^{j_2d}\le 2^{j_1/100}$. This combined with (\ref{419p}) enables us to apply Lemma 1 of \cite{SW2}. Then we have    $\tilde{a}_3/\tilde{q}_3$   approximating the coefficient $R_m(t_2)\xi_3$ of $Q_2(t_1)$ as
$$\left|R_m(t_2)\xi_3-\frac{\tilde{a}_3}{\tilde{q}_3} \right|<\frac{1}{\tilde{q}_3^2}\ \text{with}\ 
 \text{gcd}(\tilde{a}_3,\tilde{q}_3)=1\ \text{and}\ \frac{2^{j_1(1/10-1/100)}}{2}\le |\tilde{q}_3|\le 2 2^{j_1(C-1/100)}.$$
Therefore, for each fixed $t_2\sim 2^{j_2}$ we   apply  the 1-parameter Weyl sum estimate to have
 $$ \sum_{t_1\sim 2^{j_1}} e^{2\pi i  Q_2(t_1)}=O(2^{j_1-cj_1})$$
 which implies (\ref{p44}).  \end{proof}
\begin{proof}[Proof of Proposition \ref{minor arc}]
Let $\phi^{\rm{minor}}_j(\xi)\ne 0$ and $Q(t_1,t_2)=\xi_1t_1+\xi_2t_2+\xi_3 P(t_1,t_2)$. By Lemma \ref{4013} and Remark \ref{rmk41},  for the cases (i) and (ii) of the Proposition \ref{minor arc}, we obtain  \begin{align}\label{4gg}
\sup_{t_1\le 2^{j_1},t_2\le 2^{j_2}}\left|\sum_{2^{j_1-1}\le {t_1}'\le t_1}\sum_{2^{j_2-1}\le {t_2}'\le  {t_2}}e^{2\pi i  Q({t_1}',{t_2}')}\right| =O(2^{j_1+j_2} 2^{-cj_1}).
\end{align}
Then  (\ref{0053}) follows from  (\ref{4gg}) and the following summation by parts estimate,
\begin{align}
|H_{j_1,j_2}^{++}(\xi)| &\le C\frac{1}{2^{j_1}}\frac{1}{2^{j_2}}\sup_{t_1\le 2^{j_1},t_2\le 2^{j_2}}\left|\sum_{2^{j_1-1}\le {t_1}'\le t_1}\sum_{2^{j_2-1}\le {t_2}'\le  {t_2}}e^{2\pi i  Q({t_1}',{t_2}')}\right|\label{0057}
\end{align}
where
$$  H_{j}^{++}(\xi):= \sum_{2^{j_2-1}\le  t_2 \le 2^{j_2}} \sum_{2^{j_1-1}\le t_1\le  2^{j_1}} e^{2\pi i  Q(t_1,t_2)}\frac{1}{t_1t_2}.$$  
\begin{proof}[Proof of (\ref{0057})]
Let $E(t_1,t_2)=\sum_{2^{j_1-1}\le {t_1}'\le t_1}e^{2\pi i Q({t_1}',t_2)}$  and  write $H^{++}_{j_1,j_2}(\xi)$ as
$$  H^{++}_{j_1,j_2}(\xi):= \sum_{2^{j_2-1}\le  t_2 \le 2^{j_2}}\frac{H_{j_1}(t_2)}{t_2}\ \text{with}\ H_{j_1}(t_2)=\sum_{2^{j_1-1}\le t_1\le  2^{j_1}} (E(t_1,t_2)-E(t_1-1,t_2))\frac{1}{t_1}$$ 
where $ E(2^{j_1-1}-1,t_2)=0$. Then  apply the summation by parts to $\sum_{t_1}$: 
\begin{align*}
H_{j_1}(t_2)& =\left(\sum_{2^{j_1-1}\le t_1\le 2^{j_1}-1} E(t_1,t_2)\left(\frac{1}{t_1}-\frac{1}{t_1+1} \right)\right)  +E(2^{j_1},t_2)\frac{1}{2^{j_1}}. \end{align*}
By taking $\sum_{2^{j_2-1}\le t_2\le 2^{j_2}}\frac{1}{t_2}$ on both sides above, we have
 \begin{align}
 H_{j_1,j_2}^{++}(\xi_1,\xi_2,\xi_3) &=\sum_{2^{j_2-1}\le t_2\le 2^{j_2}}\frac{1}{t_2}\left(\sum_{2^{j_1-1}\le t_1\le 2^{j_1}-1} E(t_1,t_2)(\frac{1}{t_1}-\frac{1}{t_1+1})\right)\nonumber\\
 &+  \frac{1}{2^{j_1}} \left( \sum_{2^{j_2-1}\le t_2\le 2^{j_2}}\frac{1}{t_2}E(2^{j_1},t_2)\right).\label{9494}
\end{align}
The first term above is bounded by $\sum_{t_1\sim 2^{j_1}}\max_{t_1\sim 2^{j_1}} | \sum_{2^{j_2-1}\le t_2\le 2^{j_2}}\frac{1}{t_2} E(t_1,t_2)| 2^{-2j_1}.$ With respect to $t_2$, we shall apply the summation by parts again
 to have
$$\left|\sum_{2^{j_2-1}\le t_2\le 2^{j_2}}\frac{1}{t_2} E(t_1,t_2)\right|\le \left|\sum_{2^{j_2-1}\le t_2\le 2^{j_2}-1}\left(\frac{1}{t_2}-\frac{1}{t_2+1}\right) U(t_1,t_2)\right|+ \left|U(t_1,2^{j_2})\frac{1}{2^{j_2}} \right|$$
where
$$U(t_1,t_2)=\sum_{2^{j_2-1}\le t_2'\le t_2}\sum_{2^{j_1-1}\le {t_1}'\le t_1}e^{2\pi i Q({t_1}',t_2')}=O(2^{j_1+j_2}2^{-\delta j_1}).$$
So, we have
$$\sum_{t_1\sim 2^{j_1}}\max_{t_1\sim 2^{j_1}} | \sum_{2^{j_2-1}\le t_2\le 2^{j_2}}\frac{1}{t_2} E(t_1,t_2)| 2^{-2j_1}\le \frac{1}{2^{j_1+j_2}} \max_{t_1,t_2}|U(t_1,t_2)|$$
which is a desired bound for the first term of (\ref{9494}). One can use a similar summation by parts for the second term of (\ref{9494})  to complete the proof of (\ref{0057}). 
\end{proof}
Therefore, we complete the proof of Proposition \ref{minor arc}.
\end{proof}
\section{Average of Gauss Sum and Major-Minor Arc Estimate}\label{Sec500}
\subsection{Average of Gauss Sum}\label{Sec51}
As  discussed in Lemma \ref{lem002}, the exponential sum 
 $$\phi^{\rm{major,minor}}_j(\xi) H_j^{\rm{discrete}}(\xi)=  \sum_{2^{j_2/10}\le q\le 2^{j_1/10}} \sum_{a: (a_3,q)=1}  \phi_{j,q,a}(\xi) H_j^{\rm{discrete}}(\xi)$$ 
  is expressed as the multiplication of the Gauss sum and the piece of the continuous Hilbert transform in (\ref{640}). Thus one  needs to estimate the Gauss sum  in (\ref{011}),
$$S^{t_2}\left(\frac{(a_1,0,a_3)}{q}\right)=\frac{1}{q}\sum_{t_1=1}^q e^{2\pi i\left(\frac{a_1}{q}t_1+\frac{a_3}{q}P(t_1,t_2)\right)}\ \text{where  $\text{gcd}(q,a_3)=1$ }$$
having 
 a decay $q^{-c}$  for each fixed $t_2$. However,  for $P(t_1,t_2)= R_n(t_2)t_1^n+R_{n-1}(t_2)t_1^{n-1}+\cdots+R_0(t_2)t_1^0$,
one cannot have such a decay if $ \text{gcd}(a_3R_n(t_2), q)$ is relatively large. 
Instead, we are able to take the  
   average the Gauss sum over $t_2\sim 2^{j_2}$ in (\ref{640}) as: $$\frac{1}{2^{j_2}}\sum_{t_2\sim 2^{j_2}}\left|\frac{1}{q}\sum_{t_1=1}^q e^{2\pi i\left(\frac{a_1}{q}t_1+\frac{a_3}{q}P(t_1,t_2)\right)}\right|=O(q^{-c}) $$
   with one exception $P(t)=t_1R_1(t_2)+R_0(t_2)$ and $q\gg 2^{j_2C}$.
Its proof is based on the following counting lemma.
\begin{lemma}[Number of Congruent Solutions]\label{8001}
	Let $p$ be a prime and let $G(t)=b_{n}t^{n}+b_{n-1}t^{n-1}+\cdots +b_{1}t+b_{0}\in \mathbb{Z}^n[t]$. Suppose that   $b_i\not\equiv 0  \pmod{p^{m_0}}$ for some $0\le i\le n$ and $m_0\ge 1$.  Set 
\begin{align}\label{sf99}
K_{\alpha}:=\left\{t\in[1,p^{\alpha}]:G(t)\equiv 0 \pmod{p^{\alpha}}\right\}.
\end{align}
Then there are $C_n>0$ and $ \rho_n\in (0,1)$ independent of $\alpha$ and  values of $b_0,\cdots,b_n$,  such that $$ \sharp\left(K_{\alpha}\right)\leq C_{n}p^{m_0-1} p^{(1-\rho_{n})\alpha}.$$
\end{lemma}
\begin{remark}
The above lemma appears to be well-known,but we provide a detailed proof here  to make our paper self-contained. In particular, we may utilize this lemma  for future research in  multi-parameter and higher-dimensional discrete settings.
\end{remark}
\begin{proof}	
 We shall
	use an induction   argument for $n$.  It is true for $n=1$ because $b_1t+b_0\equiv 0 \pmod{p^{\alpha}}$ has  at most $p^{m_0-1}$ solutions  in $[1,p^{\alpha}]$ when $b_1\not\equiv 0\pmod{p^{m_0}}$.   Assume the result for $\text{deg}(G)\le n-1$:  If $H(t)=  c_{s}t^{s}+c_{s-1}t^{s-1}+\cdots +c_{1}t+c_{0}$ with  $c_i\not\equiv 0  \pmod{p^{m_0}}$ and $0\le i\le s\le n-1$, it holds that for every $k\in \mathbb{N}$,
		$$\exists \rho_{s}\in (0,1)\ \text{such that}\ \sharp\left( \left\{t\in[1,p^{k}]:H(t)\equiv 0 \pmod{p^{k}}\right\}   \right)\leq C_{s}p^{m_0-1}p^{(1-\rho_{s})k}.$$    We now prove the case $  \text{deg}(G)=n $.  We may assume that there is a solution $a\in [1,p^{\alpha}]$ satisfying $G(t)\equiv 0\pmod{p^{\alpha}}$. This implies that  then there is $h(t)$  satisfying  $G(t)\equiv (t-a)h(t)\pmod{p^{\alpha}}$. In other word, if $t\in K_{\alpha}$, then it holds that $p^{\alpha}|(t-a)h(t)$. Equivalently,  there exists  $ 0\le k\le \alpha$ such that
	$$   p^{\alpha-k}|(t-a)\ \text{and}\ p^k|h(t).$$Hence, we can split $K_{\alpha}$ as \begin{align*}
		&K_{\alpha}=\bigcup_{k}(\{t\in[1,p^{\alpha}]:t\equiv a \pmod{p^{\alpha-k}}\}\cap \{t\in[1,p^{\alpha}]:h(t)\equiv 0 \pmod{p^{k}}\}).\end{align*} 
	There are at most $p^{k}$ solutions in $[1,p^{\alpha}]$ for $(t-a)\equiv 0\pmod{p^{\alpha-k}}$. By the assumption, combined with the expression $$h(t)\equiv b_nt^{n-1}+(a b_n+b_{n-1})t^{n-2}+(a^2 b_n+a b_{n-1}+b_{n-2})t^{n-3}+\cdots  \pmod{p^{k}}$$  where   $b_n\not\equiv 0  \pmod{p^{m_0}}$ or    $a b_n+b_{n-1}\not\equiv 0  \pmod{p^{m_0}}$ or $a^2 b_n+a b_{n-1}+b_{n-2} \not\equiv 0  \pmod{p^{m_0}}$ etc,
 there are also at most $C_{n-1}p^{m_0-1}p^{(1-\rho_{n-1})k}$ solutions of $h(t)\equiv 0\pmod{p^{k}}$ in each  interval of   length  $p^k$,  contained  in $[1,p^{\alpha}]$. Since there are $p^{\alpha-k}$ number of such intervals, there are at most $C_{n-1}p^{m_0-1}p^{(1-\rho_{n-1})k}\times p^{\alpha-k}$ solutions  in $[1,p^{\alpha}]$. Thus  
	\begin{align*} 
	\sharp(K_{\alpha})&\le  \sum_{k=1}^{[(1-\rho_{n-1})\alpha]}\sharp\left(\{t\in[1,p^{\alpha}]:t-a\equiv 0 \pmod{p^{\alpha-k}}\}\right) \\
	&+\sum_{k=[(1-\rho_{n-1})\alpha]}^{\alpha}\sharp\left(\{t\in[1,p^{\alpha}]:h(t)\equiv 0 \pmod{p^{k}}\}\right)\\
		&\leq \sum_{k=1}^{[(1-\rho_{n-1})\alpha]}p^k+C_{n-1}p^{m_0-1}\sum_{k=[(1-\rho_{n-1})\alpha]+1}^{\alpha}p^{(1-\rho_{n-1})k}\cdot p^{\alpha-k}\\
		&\leq \alpha p^{(1-\rho_{n-1})\alpha}+C_{n-1}p^{m_0-1}\alpha p^{(1-\rho_{n-1}(1-\rho_{n-1}))\alpha}\le
		 C_np^{m_0-1}p^{(1-\rho_{n})\alpha}.
	\end{align*}
	Here, we can make  $C_n>0$ and $\rho_n\in (0,1)$  independent of $\alpha$, since  $\alpha p^{-\epsilon \alpha}\ll 1$ for large $\alpha$. We may choose $\rho_n=1/n$.
	Therefore, we have proved Lemma \ref{8001}.
\end{proof}

\begin{proposition}\label{lem055}\text{$[\rm{Average\  of\  Gauss\ Sum}]$}
For each fixed $|t_2|\sim 2^{j_2}\ge C(P)$  in  (\ref{0122}),  consider the Gauss sum, $S^{t_2}\left(\frac{(a_1,0,a_3)}{q}\right): =\frac{1}{q}\sum_{t_1=1}^q e^{2\pi i\left(\frac{a_1}{q}t_1+\frac{a_3}{q}P(t_1,t_2)\right)}$ in (\ref{011}) 
where $(1,0)\notin \Lambda(P)$ as in  (2) of Section \ref{Sec25n}. Suppose that 
  $$\text{$\exists\ 0<\eta_1<1$ such that}\  \text{gcd}(a_3,q)\le q^{1-\eta_1}\ \text{or   $\exists 0<\eta_2<1$ such that}\ \text{gcd}(a_1,q)\le q^{1-\eta_2}.$$
  Then it holds that
\begin{itemize}
\item If $P(t)\in \mathbb{Z}^{\ge 2}[t_1,t_2]$, then there are  $\delta, C>0$ independent of $q$ and $j_2$ such  that
\begin{align}\label{110}
&\frac{1}{2^{j_2}}  \sum_{t_2\sim 2^{j_2}} |S^{t_2}\left(\frac{(a_1,0,a_3)}{q}\right) |  \le Cq^{-\delta}.
\end{align}
\item If  $P(t)\in \mathbb{Z}^{1}[t_1,t_2]$. Then
with an additional assumption that $q< 2^{100dj_2}$,  \begin{align}\label{110a}
&\frac{1}{2^{j_2}}  \sum_{t_2\sim 2^{j_2}} |S^{t_2}\left(\frac{(a_1,0,a_3)}{q}\right) |  \le Cq^{-\delta}. 
\end{align}
\end{itemize}\end{proposition}
\begin{remark}
Assume that $P(t)=t_1R_1(t_2)+R_0(t_2)\in \mathbb{Z}^1[t_1,t_2]$ and $q\ge 2^{100 dj_2}$. Then whenever  $|\{ t_2\sim 2^{j_2}: a_3R_1(t_2)+a_1\equiv 0 (\text{mod}\ q)\}|\ne 0$,  
\begin{align*} 
\frac{1}{2^{j_2}}  \sum_{t_2\sim 2^{j_2}} |S^{t_2}\left(\frac{(a_1,0,a_3)}{q}\right) | &=
\frac{1}{2^{j_2}}  \sum_{t_2\sim 2^{j_2}} \frac{1}{q}|\sum_{t_1=1}^q e^{2\pi i\frac{1}{q}t_1(a_3R_1(t_2)+a_1)}  |\nonumber\\
& =\frac{1}{2^{j_2}} |\{ t_2\sim 2^{j_2}: a_3R_1(t_2)+a_1\equiv 0 (\text{mod}\ q)\}|
\ge 2^{-j_2}.\end{align*}
This tells us that the average of the Gauss sum possibly has no bound such as  $O(q^{-c})$  if
 $q\gg 2^{Kj_2}$ for sufficiently large $K>0$ and $P(t)\in \mathbb{Z}^1[t_1,t_2]$. 
 We shall treat this case in Section \ref{Sec72}.
\end{remark}
\begin{proof}[Proof of (\ref{110})]
Assume that $\text{gcd}(a_3,q)\le q^{1-\eta}$.  
Write  $P(t_1,t_2)=\sum_{\mathfrak{m}\in \Lambda(P)}  c_{\mathfrak{m}}t^{\mathfrak{m}}$  as
\begin{align}\label{0551}
P(t_1,t_2)&=t_1^m R_m(t_2)+ \sum_{m'<m}R_{m'}(t_2)t_1^{m'}  \ \text{where}\\
R_m(t_2)&=c_{mn}t_2^n+c_{m, n-1}t_2^{n-1}+\cdots+c_{m,0}\ \text{and}\  R_{m'}(t_2)=\sum_{n'} c_{m'n'} t_2^{n'}.\nonumber
\end{align}
Our case  for (\ref{110}) is  $m\ge 2$ in the above. From $c_{mn}\ne 0$ and from
 $t_2\ge 2^{j_2-1}\ge  C(P)\gg \sum_{\mathfrak{m},\mathfrak{n}\in \Lambda(P)}\left|\frac{c_{\mathfrak{m}}}{c_{\mathfrak{n}}}\right| $ in the hypothesis,   it follows that
$$ |c_{mn}t_2^{n} |\ge 5 \ (c_{m,n-1}t_2^{n-1}+\cdots+ c_{m,0}t_2^{0})\ \text{showing that}\   |R_m(t_2)|\sim |c_{mn}t_2^n|.$$
One can assume that  
\begin{align}\label{ott}
 q^{1/(20(m+n))}  \ge |c_{mn}|
\end{align}
since $ |c_{mn}|\ge q^{1/(20(m+n))}$ would give a bound $C  q^{-1/(20(m+n))}$ in (\ref{110}) with  $C=|c_{mn}|$. For estimating (\ref{110}), we decompose $\frac{1}{2^{j_2}}   \sum_{t_2\sim 2^{j_2}}  |S^{t_2}\left(\frac{(a_1,0,a_3)}{q}\right) |  =A+B$  where
\begin{align*}   
 &A:=\frac{1}{2^{j_2}}  \sum_{t_2\sim 2^{j_2}:   1\le \text{gcd}(q,R_m(t_2)) \le q^{\eta/2} } |S^{t_2}\left(\frac{(a_1,0,a_3)}{q}\right)|,\\ &B:= \frac{1}{2^{j_2}}  \sum_{t_2\sim 2^{j_2}:     \text{gcd}(q,R_m(t_2)) \ge q^{\eta/2} }|S^{t_2}\left(\frac{(a_1,0,a_3)}{q}\right)|.
\end{align*}
{\bf Estimate of $A$}.  For a fixed $t_2$, put $\text{gcd}(q,R_m(t_2))=p\le q^{\eta/2}$  and $\text{gcd}(q,a_3)=h\le q^{1-\eta}$. 
Then $\text{gcd}(q,a_3R_m(t_2))=ph_0\le ph\le q^{1-\eta/2}$.
So, there are relatively prime $\tilde{q}$ and $\tilde{a}$ such that  $q=ph_0\tilde{q}$ and $a_3 R_m(t_2)=ph_0\tilde{a} $. Then     it holds that  $\frac{a_3}{q}R_m(t_2)=\frac{\tilde{a}}{\tilde{q}} $ with
  $\text{gcd}(\tilde{a}, \tilde{q})=1$ and  $q\ge \tilde{q}=q/(ph_0)\ge q^{\eta/2}$.  This enables us to express (\ref{0551}) as the polynomial containing $ \frac{\tilde{a}}{\tilde{q}}  t_1^m$ and apply the  standard Wely sum estimate for each $t_2$ in the sum $A$ above as
\begin{align*}
S^{t_2}\left(\frac{(a_1,0,a_3)}{q}\right) &=\frac{1}{q}\sum_{t_1 =1}^q e^{2\pi i \left(  \frac{\tilde{a}}{\tilde{q}}  t_1^m +  \sum_{m'<m}    \frac{a_3 R_{m'}(t_2) }{q}  t_1^{m'}  +\frac{a_1}{q}t_1 \right)} =O\left(\frac{1}{q^\delta}\right) 
\end{align*}
which yields $A=O(1/q^\delta)$.\\
{\bf Estimate of $B$}.  Let $D(q)=\{p\in \mathbb{Z}: p| q\}$. Then we see that 
\begin{align*}
\sharp(D(q))\le C_{\epsilon} q^{\epsilon}\ \text{ for sufficiently small $\epsilon>0$. }
\end{align*}
This implies that for that small $\epsilon$,
\begin{align}
  B &\le \frac{1}{2^{j_2}}   \sum_{p\in D(q);  q^{\eta/2}\le p\le q} \sharp( \{t_2\sim 2^{j_2}: \text{gcd}(q,R_m(t_2))=p\}) \nonumber \\
  & 
  \le \frac{q^{\epsilon} }{2^{j_2}}  \max_{q^{\eta/2}\le p\le q}\sharp( \{t_{2}\sim 2^{j_2}: p|R_m(t_2)\}) \label{a099} 
  \end{align}
where if $q^{\eta/2}\le p\le q$, then one can observe that  \begin{align}\label{t}
	q\le p^{2/\eta}\le R_m(t_2)^{2/\eta}\le 2^{(2n/\eta)j_2}.
\end{align} To handle $\ref{a099}$, we shall prove that for $p$ satisfying $q^{\eta/2}\le p\le q$
\begin{align}\label{0055}
	\sharp(\{t_{2}\sim 2^{j_2}: p| R_m(t_{2})\})\le C2^{j_2}/q^c
\end{align}
for some $C,c>0$ depending only on the degree and the coefficients of $R_m(t_2)$ and $\eta$.
Set $p:= p_1^{\alpha_1} \cdots p^{\alpha_M}_M$ where $p_1,\cdots,p_M$ are prime numbers. Then we shall consider the following two cases.\\
{\bf Case 1}.  We   treat the case that   $p_i\ge 2^{j_2}$ for some $p_i\in\{p_1,\cdots,p_M\}$.   Then we have
\begin{align}
\sharp( \{t_{2}\sim 2^{j_2}: p|R_m(t_2)\}) 
 &\le \sharp(\{1\le t_{2}\le p_i: p_i| R_m(t_{2})\})\nonumber \\
 &=\sharp(\{1\le t_{2}\le p_i:  R_m(t_{2})\equiv 0\ (\rm{mod}\ {p_{i}})\})\le n.\label{sf8}\end{align} 
\begin{proof}[Proof of (\ref{sf8})]
Observe that   $\deg(R_m(t_2))=n\le C(P)<  2^{j_2}\le p_i$ from the hypothesis, and  that $|c_{mn}|\le |q|^{1/(20n)}\ll       2^{j_2}\le  p_i$ because of (\ref{ott}) and $q\le   2^{(2n/\eta)j_2}$ in (\ref{t}). Thus,
we can apply the Lagrange theorem regarding the number of   congruent solutions to obtain the last inequality of (\ref{sf8}). 
	\end{proof} 
(\ref{t}) and	(\ref{sf8})  imply that  \begin{align*}
		\sharp(\{t_{2}\sim 2^{j_2}: p| R_m(t_{2})\})\le n \le  C2^{j_2}/q^c
	\end{align*}
	for some $C,c>0$ depending only on the degree and the coefficients of $R_m(t_2)$ and $\eta$.\\
	\noindent
	{\bf Case 2}. There remains the case $p_i\le 2^{j_2}$ for all 
	$p_i\in \{p_1,\cdots,p_M\}$.   To show (\ref{0055}), due to $q^{\eta/2}\le p=p_1^{\alpha_1} \cdots p^{\alpha_M}_M$,   observe that for some 
$1\le K\le M$ and $0\le \beta_i\le \alpha_i$, 
 $$  \tilde{p}:=p_1^{\beta_1} \cdots p_K^{\beta_K}\in [q^{b}, 2^{j_2}]\ \text{for some $b>0$ not depending on $q$ and $j_2$}$$ 
 where these $p_1,\cdots,p_K$ are rearrangement of $p_1,\cdots,p_M$ forming $p$.
   Then, by decomposing $[1,2^{j_2}]$ into the intervals of $[ s\tilde{p}+1, (s+1)\tilde{p}]$ with $s=0,1,2,\cdots$, we  have the following upper bound
$$  \sharp(\{t_{2}\sim 2^{j_2}: p| R_m(t_{2})\})\le   \sharp(\{t_{2}\in [1,\tilde{p}]: \tilde{p}| R_m(t_{2})\})\left(\frac{2^{j_2}}{\tilde{p}}+1\right)$$ 
since $R_m(s\tilde{p}+t_2)\equiv R_m(t_2)\ (\rm{mod}\ \tilde{p})$.
By the ring isomorphism version    $\mathbb{Z}/(\tilde{p}\mathbb{Z})\cong \mathbb{Z}/(p_1^{\beta_1}\mathbb{Z})\times\cdots\mathbb{Z}/(p_1^{\beta_K}\mathbb{Z})$ of the Chinese remainder theorem, we obtain that
$$   \sharp(\{t_{2}\in [1,\tilde{p}]: \tilde{p}| R_m(t_{2})\})=\prod_{i=1}^K   \sharp(\{t_{2}\in [1,p_i^{\beta_i}]: p_i^{\beta_i}| R_m(t_{2})\}).$$
To show (\ref{0055}), we claim that
\begin{align*}
 \prod_{i=1}^K   \sharp(\{t_{2}\in [1,p_i^{\beta_i}]: p_i^{\beta_i}| R_m(t_{2})\})\le C|\tilde{p}||\tilde{p}|^{-c}.
 \end{align*}
To prove this, it suffices to assume that $|\tilde{p}|\gg 1$. Say,
$$  \tilde{p}\ge   e^{C_n^{\frac{100}{\rho_n}} }$$
where $C_n$ and $\rho_n$ are constants in Lemma \ref{8001}.
For $p_1,\cdots,p_K$ and  $c_{mn}$ above, we can take the maximal $\gamma_1,\cdots,\gamma_K$ such that  $p_1^{\gamma_1}\cdots p_K^{\gamma_K}$ divides  $|c_{mn}|$. Then from $$p_1^{\gamma_1+1},\cdots, p_K^{\gamma_K+1}
\not| c_{mn}$$
one can observe that for $m_0= \gamma_i+1$ where $i=1,\cdots,K$,   
$$\text{$c_{mn}\not\equiv 0\ (\text{mod}\ p_i^{m_0} )$ in Lemma \ref{8001}}.$$
 So, we can apply (\ref{sf99}) in Lemma \ref{8001} for the polynomial $$R_m(t_2)=c_{mn}t_2^n+c_{m, n-1}t_2^{n-1}+\cdots+c_{m,0}$$ to obtain the first inequality below
\begin{align}
\prod_{i=1}^K   \sharp(\{t_{2}\in [1,p_i^{\beta_i}]: p_i^{\beta_i}| R_m(t_{2})\})\nonumber&\le 
C_n^K   \prod_{i=1}^K  p_i^{\gamma_i}p_i^{\beta_i(1-\rho_n)}  \\
\nonumber&\le 
C_n^K   \prod_{i=1}^K  p_i^{\gamma_i}p_i^{\beta_i(1-\rho_n)} \prod_{i=K+1}^M p_i^{0}p_i^{\beta_i(1-\rho_n)}\\
&\le C_n^K |c_{mn}|\tilde{p}^{1-\rho_n}\le 
|c_{mn}|\tilde{p}^{1-\frac{\rho_n}{2}}.\label{0400}
\end{align}
where the second follows from $ \prod_{i=K+1}^M p_i^{0}p_i^{\beta_i(1-\rho_n)}\ge 1$ and the third from $p_1^{\gamma_1}\cdots p_K^{\gamma_K}$ dividing $|c_{mn}|$.
For the last inequality  of (\ref{0400}), we shall show that $C_n^K\lesssim  \tilde{p}^{\frac{\rho_n}{2}} $. It suffices to consider the case $\tilde{p}\gg 1$  as above. Then one can observe  $$K\le \text{number of distinct prime divisors of $\tilde{p}$ }\lesssim \frac{\log \tilde{p}}{\log (\log \tilde{p})}\le \frac{ \log \tilde{p}}{  \frac{100}{\rho_n}\log C_n}.$$   Thus
 $  C_n^K\le   \tilde{p}^{\frac{\rho_n}{100}}$.  This yields the last inequality of (\ref{0400}).
Using (\ref{0400}) with $\tilde{p}\ge q^b$, one has $c,C>0$:   
$$  \sharp(\{t_{2}\sim 2^{j_2}: p| R_m(t_{2})\})\le  2|c_{mn}|\tilde{p}^{1-\frac{\rho_n}{2}} 2^{j_2}\tilde{p}^{-1}\le C 2^{j_{2}} q^{- c}$$
which proves (\ref{0055}). \end{proof}
\begin{proof}[Proof of (\ref{110a})]
We  prove (\ref{110a}) under the assumption  $h=\text{gcd}(a_3,q)\le q^{1-\eta}$.   Our case is that $q<2^{100d j_2}$ and   $m=1$ in (\ref{0551}):
\begin{align*} 
P(t_1,t_2)&=t_1^1 R_1(t_2)+ R_0(t_2) \ \text{and}\ S^{t_2}(\frac{a}{q})  =\frac{1}{q}\sum_{t_1=1}^q e^{2\pi i \frac{1}{q} \left(t_1  (a_3R_1(t_2)+a_1)+ a_3R_0(t_2)\right)}.
\end{align*}
Note that
\begin{align}\label{q}
\frac{1}{2^{j_2}}  \sum_{t_2\sim 2^{j_2}} |S^{t_2}\left(\frac{(a_1,0,a_3)}{q}\right)|&\le\frac{1}{2^{j_2}}   \sum_{t_2\sim 2^{j_2}} \frac{1}{q}\left|\sum_{t_1=1}^q e^{2\pi i\frac{1}{q}t_1(a_3R_1(t_2)+a_1)}  \right| \nonumber
\\
&=\frac{1}{2^{j_2}}  \sum_{ t_2\sim 2^{j_2}: a_3R_1(t_2)+a_1\equiv 0 (\text{mod}\ q)}1.
\end{align}
For the inequality (\ref{q}), we shall count \begin{align*}
	\sharp(\{t_2\sim 2^{j_2}: a_3R_1(t_2)+a_1\equiv 0 (\text{mod}\ q)   \}).
\end{align*} Let $a_3=h\tilde{a}_3$ and $q=h\tilde{q}$ where $\tilde{q}\ge q^{\eta}$. In order to have $t_2$ satisfying $a_3R_1(t_2)+a_1\equiv 0 (\text{mod}\ q)$, it is necessary that  $h|a_1$. We write $a_1$ as $a_1=\tilde{a}_1h$.
Hence, the equation $a_3R_1(t_2)+a_1\equiv 0 (\text{mod}\ q) $ becomes
$\tilde{a}_3R_1(t_2)+\tilde{a}_1\equiv 0 (\text{mod} \ \tilde{q})$ where $\text{gcd}(\tilde{a}_3,\tilde{q})=1$. By $\tilde{q}\ge q^{\eta}$ and $\text{gcd}(\tilde{a}_3,\tilde{q})=1$ and the assumption $\tilde{q}\le q<2^{100d j_2}$, if we replace $R_m(t_2)$ and $p$  in (\ref{0055}) with $\tilde{a}_3R_1(t_2)+\tilde{a}_1$ and $\tilde{q}$ respectively, then  we obtain that  \begin{align}\label{t1}
	\sharp(\{t_2\sim 2^{j_2}: \tilde{a}_3R_1(t_2)+\tilde{a}_1\equiv 0 (\text{mod} \ \tilde{q})  \})\le C2^{j_2}/q^c \quad \text{(due to (\ref{0055}))}
	\end{align}
	for some $C,c>0$ depending only on the degree and the coefficients of $R_1(t_2)$ and $\eta$.
  \end{proof}
We now show (\ref{110}) and (\ref{110a}) under the condition $\text{gcd}(q,a_1)\le q^{1-\eta}$.\begin{proof}[The case of $\text{gcd}(a_1,q)\le q^{1-\eta}$]
Next, we claim that (\ref{110}) and (\ref{110a}) hold under the condition of  $\text{gcd}(a_1,q)=h\le q^{1-\eta}$.  By the previous generalization, it suffices to consider the case $\text{gcd}(a_3,q)> q^{1-\eta/2}$.  Then it holds that
 $\text{gcd}((a_1+a_3 R_1(t_2)), q)\le q^{1-\eta/2}$. Using this for the coefficient of the linear term of 
the one-parameter-Gauss sum $$S^{t_2}\left(\frac{(a_1,0,a_3)}{q}\right)=\frac{1}{q}\sum_{t_1=1}^q e^{2\pi i\left(\frac{(a_1+a_3R_1(t_2))}{q} t_1+\frac{a_3}{q}R_2(t_2) t_1^2+\cdots+\frac{a_3}{q}R_m(t_2)t_1^m\right)}$$  to obtain that $S^{t_2}\left(\frac{(a_1,0,a_3)}{q}\right)=O(q^{-\delta})$ for some $\delta>0$.  This completes a proof of Proposition \ref{lem055}.
\end{proof}
\subsection{Double Hilbert transforms; Continuous Cases}\label{Sec52}
Before estimating the major part, we give a   proof of Theorems \ref{th21} and \ref{th22}. Moreover, we obtain   the $L^p$ boundedness of the continuous double Hilbert transform along a polynomial surface on $D_{B}=\{x\in\mathbb{R}^2: |x_1|,|x_2|\ge 1\}$ where $B=\{-{\bf e}_1,-{\bf e}_2\}$ under the even vertex assumption.

\begin{lemma}\label{lem004}[Single-Parameter Oscillatory Singular Integral]
For $\xi\in \mathbb{R}^N$ and $m_\ell\in \mathbb{Z}_+$ for all $\ell=1,\cdots,N$, let
$$\mathcal{H}_j(\xi)=\int e^{2\pi i \sum_{\ell=1}^N\xi_\ell x^{m_\ell}} \chi\left(  \frac{x}{2^{j }}\right)\frac{dx}{x}\ \text{where $j\in\mathbb{Z}$}$$
where $\chi$  defined in Definition \ref{df31}.
Then there exist $c,C>0$ independent of $\xi$ such that
\begin{align}
\left|\mathcal{H}_j(\xi)\right|\lesssim \min\left\{\left[\sum_{\ell=1}^N |2^{m_\ell j} \xi_\ell|\right]^{-c}, \sum_{\ell=1}^N |2^{m_\ell j} \xi_\ell|\right\}\ \text{and}\ \sum_{j\in\mathbb{Z}} \left|\mathcal{H}_j(\xi)\right|+ \left|\mathcal{H}_j(\xi)\right|^{1/2}&\le C.\label{015}
\end{align}
\end{lemma}

\begin{proposition}[Two-parameter Oscillatory Singular Integral]\label{prop006}
We consider a real valued polynomial  $P(x_1,x_2)=\sum_{\mathfrak{m}\in\Lambda(P)}c_{\mathfrak{m}}x^{\mathfrak{m}}$  of two variables. Consider 
\begin{align}\label{2D}
\mathcal{H}^{\Lambda(P)}_j (\xi_1,\xi_2,\xi_3)&=\int e^{2\pi i( \xi_1x_1+\xi_2x_2+\xi_3 P(x_1,x_2))} \ \chi\left(  \frac{x_1}{2^{j_1}}\right) \chi\left(  \frac{x_2}{2^{j_2}}\right) \frac{dx_1}{x_1} \frac{dx_2}{x_2}.
\end{align}
If    every vertex vector $\mathfrak{m}=(m_1,m_2)$ of  $ {\bf N}(P,D_{\{-{\bf e}_1,-{\bf e}_2\}})$ has an even component, then
\begin{align}\label{025}
\sum_{j\in \mathbb{Z}_+^2}|\mathcal{H}_j^{\Lambda(P)}(\xi_1,\xi_2,\xi_3)|,\ \sum_{j\in \mathbb{Z}_+^2}|\mathcal{H}_j^{\Lambda(P)}(\xi_1,\xi_2,\xi_3)|^{1/2}\le C
\end{align}
  where  $C$ is  independent of $ \xi$, while it can depend on  the coefficients  of $P $. Moreover for any $a_j$ with $j\in\mathbb{Z}^2_+$, we have
\begin{align}\label{025h}
 \|\sum_{j\in \mathbb{Z}_+^2}[\mathcal{H}_j^{\Lambda(P)}(\xi)\phi(a_j\xi_3)]^{\vee}*f \|_{L^p(\mathbb{R}^3)}\le C\|f\|_{L^p(\mathbb{R}^3)}.
\end{align}

\end{proposition}

\begin{proof}
In view of (\ref{au0}), it suffices to consider the sum over $j\in \mathbb{Z}(\mathbb{F})=(-\mathbb{F}^*)\cap \mathbb{Z}_{\ge 0}^2$.
We shall consider the case $\mathbb{F}=\mathfrak{m}=(even,odd)$ only.  Let $$m(\xi_3,\mathfrak{m},\mathfrak{n})= \min\{ |\xi_32^{ j\cdot \mathfrak{m}}|^{-c}, |\xi_3  2^{j\cdot \mathfrak{m}}|, |\xi_32^{j\cdot \mathfrak{n}}|^{-c},  |\xi_32^{j\cdot \mathfrak{n}}|\}.$$  
Observe that $\Lambda(P)\setminus\{\mathfrak{m}\}= \left(\Lambda(P)\setminus \rm{cone}(\mathfrak{m})\right)\cup \left( (\Lambda(P)\setminus\{\mathfrak{m}\})\cap \rm{cone}(\mathfrak{m})\right)$.  Thus,
apply the telescoping sum and use the decay from (\ref{jq501}) and the mean value theorem to have
\begin{align*}
 \sum_{j\in \mathbb{Z}(\mathbb{F})}|\mathcal{H}_j^{\Lambda(P)}(\xi)-\mathcal{H}_j^{\{\mathfrak{m}\} }(\xi) |&\lesssim  \sum_{\mathfrak{n}\in \Lambda(P)\setminus \rm{cone}(\mathfrak{m})} \left(\sum_{j\in \mathbb{Z}(\mathbb{F})}m(\xi_3,\mathfrak{m},\mathfrak{n})\right)\\
 &  + \sum_{\mathfrak{n}\in (\Lambda(P)\setminus\{\mathfrak{m}\})\cap \rm{cone}(\mathfrak{m}) }  \left(\sum_{j\in \mathbb{Z}(\mathbb{F})}m(\xi_3,\mathfrak{m},\mathfrak{n})\right).
\end{align*}
The first part is bounded by $C$   from $\text{rank}(\{\mathfrak{m}\}\cup ( \Lambda(P)\setminus \rm{cone}(\mathfrak{m})))= 2$
and the second part is bounded by $ \sum_j 2^{-\delta|j|}\le C$ from $|\xi_32^{j\cdot \mathfrak{n}}|\le |\xi_3 2^{(1-\delta)j\cdot \mathfrak{m}}| $ for $\mathfrak{n}\in (\Lambda(P)\setminus\{\mathfrak{m}\})\cap \rm{cone}(\mathfrak{m})$. So it suffices to treat $\mathcal{H}_j^{\{\mathfrak{m}\} }(\xi)$. From $ \mathcal{H}_j^{\{\mathfrak{m}\} }(0,\xi_2,\xi_3)=\mathcal{H}_j^{\{\mathfrak{m}\} }(0,0,\xi_3)=0$ due to $\int \frac{dt_1}{t_1}=0$ for $\mathfrak{m}=(even,odd)$, one has
\begin{align*}
&\sum_{j\in \mathbb{Z}(\mathbb{F})}|[\mathcal{H}_j^{\{\mathfrak{m}\} }(\xi)-\mathcal{H}_j^{\{\mathfrak{m}\} }(0,\xi_2,\xi_3)]-[\mathcal{H}_j^{\{\mathfrak{m}\} }(\xi_1,0,\xi_3)-\mathcal{H}_j^{\{\mathfrak{m}\} }(0,0,\xi_3)]| \phi(\xi_22^{j_2})\\
&\qquad \lesssim \sum_{j\in \mathbb{Z}(\mathbb{F})}\min\{  |\xi_12^{j_1}|, |\xi_12^{j_1}|^{-c},|\xi_22^{j_2}|\}\phi(\xi_22^{j_2}) \lesssim 1\ \text{and}\ \\
&\sum_{j\in \mathbb{Z}(\mathbb{F})}|[\mathcal{H}_j^{\{\mathfrak{m}\} }(\xi)-\mathcal{H}_j^{\{\mathfrak{m}\} }(0,\xi_2,\xi_3)]  \phi^c(\xi_22^{j_2}) |\\
&\qquad\lesssim \sum_{j\in\mathbb{Z}(\mathbb{F})}\min\{  |\xi_12^{j_1}|, |\xi_12^{j_1}|^{-c},|\xi_22^{j_2}|^{-c}\}\phi^c(\xi_22^{j_2}) \lesssim 1.
\end{align*}
Finally, one knows that  $$\sum_{j\in \mathbb{Z}(\mathbb{F})}|\mathcal{H}_j^{\{\mathfrak{m}\} }(\xi_1,0,\xi_3)  |\lesssim \sum_{j\in\mathbb{Z}(\mathbb{F})}\min\{  |\xi_12^{j_1}|, |\xi_12^{j_1}|^{-c},|\xi_32^{j\cdot \mathfrak{m}}|,|\xi_32^{j\cdot \mathfrak{m}}|^{-c}\} \lesssim 1.$$  
Therefore, we obtain the first part of (\ref{025}). Next by the same argument, one can show the second part of (\ref{025}). Finally, one can apply the Littlewood-Paley inequalities for $|\xi_32^{ j\cdot \mathfrak{n}}|\approx 2^{\ell_3}$ to handle the sum of $\mathcal{H}_j^{\Lambda(P)}(\xi)-\mathcal{H}_j^{\{\mathfrak{m}\} }(\xi) $. Next  apply the Littlewood-Paley inequality again for $ |\xi_12^{j_1}|\approx 2^{\ell_1}$ and $ |\xi_22^{j_2}|\approx 2^{\ell_2}$ to handle the last two summation to obtain (\ref{025h}).   See Section 4 in \cite{CWW} and  see Chapter 7  in \cite{K} for the higher dimensional case.
\end{proof} 
By Proposition \ref{prop006}, we proved the sufficient parts of Theorems \ref{th21} and \ref{th22}.  We shall prove their necessity parts in  (\ref{9431}).

\subsection{Major-Minor Arc Estimate}\label{Sec6}
\begin{proposition}\label{prop001} [Major-Minor-Arc]
 There exist $C,c>0$ independent of $\xi$ such that  
  \begin{align}\label{082} 
  \sum_{j_1:j\in \mathbb{Z}^2(\mathbb{F})}  |H^{\rm{discrete}}_{j}(\xi_1,\xi_2,\xi_3) |\phi^{\rm{major,minor}}_j(\xi)&\le C2^{-cj_2}\ \text{if $P\in \mathbb{Z}^{\ge 2}[t_1,t_2]$},\nonumber\\
  \sum_{j_1:j\in \mathbb{Z}_{\approx}^2(\mathbb{F})}  |H^{\rm{discrete}}_{j}(\xi_1,\xi_2,\xi_3) |\phi^{\rm{major,minor}}_j(\xi)&\le C2^{-cj_2}\ \text{if $P\in \mathbb{Z}^{1}[t_1,t_2]$}. \end{align} 
\end{proposition}
\begin{proof}
Recall (\ref{010}) and (\ref{011}). 
By (\ref{640}) in Lemma \ref{lem002},
\begin{align*} 
  |H^{\rm{discrete}}_{j}(\xi)|\phi^{\rm{major,minor}}_j(\xi) & \lesssim \sum_{t_2\sim 2^{j_2}} \sum_{2^{j_2/10}\le q\le 2^{j_1/10}} \sum_{a: (a_3,q)=1} \phi_{j,q,a}(\xi) \nonumber \\
&   \times
|S^{t_2}\left(\frac{(a_1,0,a_3)}{q}\right)|\left(|\mathcal{H}^{t_2}_{j_1}(\beta_1,0,\beta_3) |+O( 2^{-cj_1}) \right).
\end{align*}
The     sum of the error $O(2^{-cj_1})$ over $j_1:j_1\ge j_2$ gives a bound $C 2^{-cj_2}$. 
Next one can control the sum of the main term over $j_1:j_1\ge j_2$  by
\begin{align}\label{pia}
& \sum_{j_1:j\in \mathbb{Z}^2(\mathbb{F})}  \sum_{2^{j_2/10}\le q\le 2^{j_1/10}} \sum_{a: (a_3,q)=1}  \frac{1}{2^{j_2}}\sum_{t_2\sim 2^{j_2}}|S^{t_2}\left(\frac{(a_1,0,a_3)}{q}\right)\phi_{j,q,a}(\xi)  \mathcal{H}^{t_2}_{j_1}(\beta_1,0,\beta_3)  |    \\
&\le  \sum_{2^{j_2/10}\le q\le 2^{j_1/10} } \sum_{a: (a_3,q)=1} \frac{1}{2^{j_2}}  \sum_{t_2\sim 2^{j_2}} | S^{t_2}\left(\frac{(a_1,0,a_3)}{q}\right) |   \phi_{j,q,a}(\xi)  \sum_{j_1\in \mathbb{Z}_+} |\mathcal{H}_{j_1}^{t_2}(\beta_1,0,\beta_3) |.\nonumber
\end{align}
Under the condition that (i) $P(t_1,t_2)\in \mathbb{Z}^{\ge 2}[t_1,t_2]$ or (ii) $P(t_1,t_2)\in \mathbb{Z}^1[t_1,t_2]$ with $ j_1\le 100d j_2$, we are able to apply  the Gauss-sum average decay in Proposition \ref{lem055}:
  $$\frac{1}{2^{j_2}}  \sum_{t_2\sim 2^{j_2}} |S^{t_2}\left(\frac{(a_1,0,a_3)}{q}\right) |\phi^c\left(\frac{q}{2^{j_2/10}}\right)  \le Cq^{-\delta}$$ 
  because the condition $j_1\le 100d j_2$ in (ii) implies that $  q<2^{j_1/10}\le 2^{10dj_2}$.
  By this combined with the uniform estimate $\sum_{j_1\in \mathbb{Z}_+} |\mathcal{H}_{j_1}^{t_2}(\beta_1,0,\beta_3) | \le C$   in $t_2$ in Lemma \ref{lem004}, we obtain that
\begin{align*}
\text{RHS of (\ref{pia})}&\le  \sum_{2^{j_2/10}\le q } \sum_{a: (a_3,q)=1}  q^{-c}
    \phi_{j,q,a}(\xi)    \le 2^{-j_2/10}
\end{align*} 
where  the second inequality follows from   the disjointness of $\xi$ in $a$ for each fixed $q$ and the lacunary distribution of $q$ in Lemma \ref{le32}.
Therefore, we proved (\ref{082}).
\end{proof}

 \section{Non-Major-Arc-Estimate when $P(t)=t_1R_1(t_2)+R_0(t_2)$ and $j_1\gg j_2$}\label{Sec72}
In view of (\ref{414b}), (\ref{082}), to finish the non-major arc estimate (\ref{nnmm}), there remains:
 $$ \sum_{j_1:j\in \mathbb{Z}_{\gg}^2(\mathbb{F})}  |H^{\rm{discrete}}_{j}(\xi_1,\xi_2,\xi_3) (\phi_j^{\rm{minor}}(\xi_3)+\phi^{\rm{major,minor}}_j(\xi_3))\le C2^{-cj_2}\ \text{if $P\in \mathbb{Z}^{1}[t_1,t_2]$} $$ 
 where $\phi_j^{\rm{minor}}(\xi_3)+\phi^{\rm{major,minor}}_j(\xi_3)=1-\phi_j^{\rm{major}}(\xi_3)$. Recall that
$$\phi_{j,q,a}(\xi)  =\Gamma_{a_1/q}\left(\xi_1\right)\Gamma_{a_2/q}\left(\xi_2\right)\phi\left(\frac{\xi_3-a_3/q }{2^{-j\cdot \mathfrak{m}}2^{j_1/10}}\right)$$ and set
\begin{align*} 
\psi^{\rm{variant1}}_{j}(\xi)&:=  1-\sum_{1\le q\le 2^{100dj_2}} \sum_{a : (a_3,q)=1}\phi_{j,q,a}(\xi),\\
\psi^{\rm{variant2}}_{j}(\xi)&:= \sum_{2^{j_2/10}\le q\le 2^{100dj_2}} \sum_{a: (a_3,q)=1} \phi_{j,q,a}(\xi).
\end{align*}
We note that
\begin{align*}
\psi^{\rm{variant1}}_{j}(\xi)+\psi^{\rm{variant2}}_{j}(\xi)\equiv  1-\phi^{\rm{major}}_{j}(\xi)\ \text{on $[0,1]^3$}.\end{align*}
\begin{proposition}\label{prop885}
Suppose that $P(t)=t_1 R_1(t_2)+R_0(t_2)$ where  $deg(R_1)<d$. Set
\begin{align*}
H^{\rm{discrete}}_j(\xi)&=\sum_{t_1\sim 2^{j_1},t_2\sim 2^{j_2}}e^{-2\pi i(\xi_1t_1+\xi_2t_2+\xi_3(t_1 R_1(t_2)+R_0(t_2)) )}.
\end{align*}
Then for all $\xi\in \mathbb{R}^3$, there exist  $C>0$ and $c>0$ such that
\begin{align}
\sum_{j_1:j\in \mathbb{Z}^2_{\gg}(\mathbb{F})} \psi^{\rm{variant1}}_{j}(\xi) |H^{\rm{discrete}}_j(\xi)|\le C2^{-cj_2},\label{mmex}\\
\sum_{j_1:j\in \mathbb{Z}^2_{\gg}(\mathbb{F})} \psi^{\rm{variant2}}_{j}(\xi) |H^{\rm{discrete}}_j(\xi)|\le C2^{-cj_2}.\label{9437}\end{align}
   \end{proposition}
   \begin{proof}[Proof of (\ref{9437})]
   We use the fact $2^{j_2/10}<q<2^{100dj_2}$ for  (\ref{110a}) to have
   \begin{align*} 
&  \sum_{j_1:j\in \mathbb{Z}^2_{\gg}(\mathbb{F})}|H^{\rm{discrete}}_{j}(\xi)|\psi^{\rm{variant2}}_j(\xi_3) \\
  & \lesssim  \sum_{2^{j_2/10}\le q \le 2^{100d j_2} } \sum_{a: (a_3,q)=1} \frac{1}{2^{j_2}}  \sum_{t_2\sim 2^{j_2}} | S^{t_2}\left(\frac{(a_1,0,a_3)}{q}\right) |   \phi_{j,q,a}(\xi)  \sum_{j_1\in \mathbb{Z}_+} |\mathcal{H}_{j_1}^{t_2}(\beta_1,0,\beta_3) |\\
  &  \lesssim  \sum_{2^{j_2/10}\le q \le 2^{100d j_2} } \sum_{a: (a_3,q)=1} q^{-\delta} \phi_{j,q,a}(\xi)  \sum_{j_1\in \mathbb{Z}_+} |\mathcal{H}_{j_1}^{t_2}(\beta_1,0,\beta_3) |\le C2^{-cj_2}.
\end{align*}
This shows (\ref{9437}).
   \end{proof}
To prove  (\ref{mmex}), we  now set up the double arc as follows.   
Firstly,  note that $\mathbb{F} = \mathfrak{m} = (1,n)$ with $n \leq d$ due to the condition $P \in \mathbb{Z}^{1}[t_1, t_2])$.  By the Dirichlet approximation and the support condition $\psi^{\rm{variant1}}_{j}(\xi) \ne 0$ for $j \in \mathbb{Z}_{\gg}^2(\mathbb{F})$, one can find  a pair $(q, a_3)$ such that $\text{gcd}(a_3, q) = 1$ satisfying the following properties:
\begin{align}\label{5910}
	\left| \xi_3 - \frac{a_3}{q} \right| \le \frac{1}{q 2^{(j \cdot \mathfrak{m} - j_1/10)}} \quad \text{and} \quad 2^{100 d j_2} < q < 2^{j \cdot \mathfrak{m} - j_1/10} \quad \text{where } \mathfrak{m} = (1, n).
\end{align}
Next, for a fixed $t_2$ we regard  $\xi\cdot (t_1,t_2,P(t))$  as the one variable function of $t_1$:
$$\xi_1t_1+\xi_2t_2+\xi_3(t_1 R_1(t_2)+R_0(t_2))=\xi(t_2)t_1+[\xi_2t_2+\xi_3R_0(t_2)] $$
which is a degree-1-polynomial of $t_1$,    whose coefficient of $t_1$ is
\begin{align}\label{584}
	\xi(t_2):= \xi_3R_1(t_2)+\xi_1.
\end{align}
Then our exponential sum becomes $$H^{\rm{discrete}}_j(\xi)=\sum_{t_2\sim 2^{j_2}}\frac{e^{-2\pi i  (\xi_3R_0(t_2)+\xi_2t_2)} }{t_2}\sum_{t_1\sim 2^{j_1}}\frac{ e^{-2\pi i  \xi(t_2)t_1 }}{t_1}.$$ 
 Our idea is to utilize  not only the original arc for $\xi_3$ in (\ref{5910}) but also a new arc  due to Dirichlet approximation  of  $\xi(t_2)$. Given $j=(j_1,j_2)$, $t_2\sim 2^{j_2}$ and $\xi_2(t_2)$ in (\ref{584}), we can find a pair $(a_{j}(t_2),q_{j}(t_2))\in \mathbb{Z}^2$ such that  
\begin{align}\label{0sd}
	\left|\xi(t_2)-\frac{a_j(t_2)}{q_j(t_2)}\right|<\frac{1}{q_j(t_2)2^{(1-\frac{1}{100})j_1}}
\end{align}
$$\ \text{where}\  gcd(a_j(t_2),q_j(t_2))=1\ \text{and}\ 1\le q_j(t_2)\le 2^{(1-\frac{1}{100})j_1}.$$
According to $q_j(t_2)$ in (\ref{0sd}), we split 
$$ \sum_{j_1:j\in \mathbb{Z}^2_{\gg}}|\psi^{\rm{variant1}}_j(\xi_3)H^{\rm{discrete}}_j(\xi)|\le  \sum_{j_1:j\in \mathbb{Z}^2_{\gg}} M^{\rm{good}}_j(\xi)+ \sum_{j_1:j\in \mathbb{Z}^2_{\gg}}M^{\rm{bad}}_j(\xi)$$ where
\begin{align}
M^{\rm{good}}_j(\xi)&:=\frac{\psi^{\rm{variant1}}_j(\xi_3)}{2^{j_2}}\sum_{t_2\sim 2^{j_2}: q_j(t_2)>1 }\left| \sum_{t_1\sim 2^{j_1}}\frac{ e^{-2\pi i  \xi(t_2)t_1 }}{t_1}\right|,\notag\\
M^{\rm{bad}}_j(\xi)&:=\frac{\psi^{\rm{variant1}}_j(\xi_3)}{2^{j_2}} \sum_{t_2\sim 2^{j_2}:q_j(t_2)=1} \left|\sum_{t_1\sim 2^{j_1}}\frac{ e^{-2\pi i  \xi(t_2)t_1 }}{t_1}\right|.\label{94b}
\end{align}
Now, we shall show that
\begin{align}\label{69v}
	\sum_{j_1:j\in \mathbb{Z}^2_{\gg}} M^{\rm{good}}_j(\xi)+ \sum_{j_1:j\in \mathbb{Z}^2_{\gg}}M^{\rm{bad}}_j(\xi)\le C2^{-cj_2}
\end{align}
\subsection{Estimate of $M^{\rm{good}}_j(\xi)$}
We  claim that there exists  $c>0$ such that
\begin{align}\label{ksm1}
	\left|\sum_{t_1\sim 2^{j_1}}\frac{ e^{-2\pi i  \xi(t_2)t_1 }}{t_1}\right| \le 2^{-cj_1}.
\end{align}
To show (\ref{ksm1}), we start with the case $2^{\frac{1}{10} j_1}\le q_j(t_2)\le  2^{j_1(1-\frac{1}{100})}$ in (\ref{0sd}). For this case, apply the Weyl sum estimate in (\ref{43k}) to have $$ \left| \sum_{t_1\sim 2^{j_1}} e^{-2\pi i  t_1\xi(t_2) }\right| \le 2^{-cj_1}2^{j_1}.$$ This with the summation by parts shows that $\left|\sum_{t_1\sim 2^{j_1}} e^{-2\pi i \xi(t_2)t_1 }\frac{1}{t_1}\right| \le 2^{-c'j_1}$.\\
There remains the case
$   2\le q_j(t_2)< 2^{\frac{1}{10} j_1}$ in (\ref{0sd}). Then rewrite the sum of (\ref{ksm1}) using the cutoff function $\chi_{j_1}$ in (\ref{aria}) and shall
 show that
\begin{align}\label{ksm}
 \sum_{t_1\in \mathbb{Z}}  \frac{ e^{-2\pi i  t_1\xi(t_2) }\chi_{j_1}(t_1)}{t_1}=O( 2^{-j_1}).
\end{align}
\begin{proof}[Proof of (\ref{ksm})]
Fix $t_2\sim 2^{j_2}$.	From (\ref{0sd}), remind that  $ (a_j(t_2),q_j(t_2)) $ and $\xi(t_2)=\beta_j(t_2) +\frac{a_j(t_2)}{q_j(t_2)}$, we have $ |\beta_j(t_2)|\le  \frac{2^{-(1-\frac{1}{10})j_1}}{q_j(t_2)} $ and $2\le q_j(t_2)<2^{\frac{1}{10} j_1}$.  For simplicity, we write $  (a_j(t_2),q_j(t_2),\beta_j(t_2))= (a',q',\beta')$. Next by taking $t_1=\mu q'+\ell$  for  some $\mu\in \mathbb{Z}$ with $\ell=1,\cdots,q'$,   express LHS of (\ref{ksm}) as the double sum
	$$ \sum_{\ell=1}^{q'}\sum_{\mu\in \mathbb{Z}}\frac{ e^{-2\pi i  (\mu q'+\ell)(\beta'+a'/q') } }{ \mu q'+\ell }   \chi_{j_1} (\mu q' +\ell) = \sum_{\ell=1}^{q'}  e^{-2\pi i a' \ell  /q'  } \left(\sum_{\mu\in \mathbb{Z}}\frac{ e^{-2\pi i  (\mu q'+\ell) \beta } }{ \mu q'+\ell }   \chi_{j_1}(\mu q'+\ell)\right).  $$
Due to $|\beta' |\le  \frac{2^{-(1-\frac{1}{10})j_1}}{q'}$ and $2\le q'\le 2^{j_1/10},$	one can apply the van der Corput theorem for $\sum_{\mu}$ to obtain that
	$$\sum_{\mu\in \mathbb{Z}}\frac{ e^{-2\pi i  (\mu q'+\ell) \beta' } }{ \mu q'+\ell }  \chi_{j_1}(\mu q' +\ell)=O(2^{-j_1}) +\frac{1}{q'}\mathcal{H}_{j_1}(\beta')\ \text{for}\ \mathcal{H}_{j_1}(\beta')=\int e^{-2\pi i \beta' x_1}  \frac{ \chi_j (x_1)}{x_1}dx_1.$$ 
	  Therefore,    LHS  of (\ref{ksm}) is
	$$ O(2^{-8j_1/10})+\mathcal{H}_{j_1}(\beta') \frac{1}{q'}\sum_{\ell=1}^{q'}  e^{-2\pi i a' \ell  /q'  } = O(2^{-8j_1/10})$$  because $  \sum_{\ell=1}^{q'}  e^{-2\pi i a' \ell /q' } =0$ due to $q'\ge 2$ and $\text{gcd}(q',a')=1$.  So we obtain (\ref{ksm}).  
\end{proof}
So, we have proved (\ref{ksm1}).
\subsection{Estimate of $M^{\rm{bad}}_j(\xi)$}
\begin{definition}
 Let $j\in \mathbb{Z}_{\gg}^2(\mathbb{F})$. Suppose that   $\xi$ satisfies 
(\ref{5910}) where $2^{100 dj_2}<q<2^{j\cdot \mathfrak{m} -j_1/10}$ for  $\mathfrak{m}=(1,n)$. Then given $1\le q_j(t_2)\le 2^{j_1(1-\frac{1}{100})}$ in (\ref{0sd}), we define $$V(j_1,j_2,\xi)=\{t_2\sim 2^{j_2}:q_{j}(t_2)=1\}.$$ 
\end{definition}
Then in (\ref{94b}),  one can express
\begin{align*} 
	\sum_{j_1:j\in \mathbb{Z}^2(\mathbb{F})} M^{\rm{bad}}_j(\xi)= \sum_{j_1:j\in \mathbb{Z}^2(\mathbb{F})} \frac{\phi^{\rm{variant1}}_j(\xi_3)}{2^{j_2}} \sum_{t_2\in V(j_1,j_2,\xi)} \left|\sum_{t_1\sim 2^{j_1}}\frac{ e^{-2\pi i  \xi(t_2)t_1 }}{t_1}\right|.
\end{align*}
To complete  (\ref{69v}),  we shall find $C >0$ independent of $\xi$ such that 
\begin{align}\label{0bb}
	\sum_{j_1:j\in \mathbb{Z}^2(\mathbb{F})} \frac{\phi^{\rm{variant1}}_j(\xi_3)}{2^{j_2}} \sum_{t_2\in V(j_1,j_2,\xi)} \left|\sum_{t_1\sim 2^{j_1}}\frac{ e^{-2\pi i  \xi(t_2)t_1 }}{t_1}\right|\le C2^{-j_2}.
\end{align}
Our proof of (\ref{0bb}) is based on the following lemma.
\begin{lemma}\label{lem944}
Consider   $\psi_{(j_1,j_2)}^{\rm{varinat1}}(\xi)\ne 0$ and  $\psi^{\rm{varinat1}}_{(j_1',j_2)}(\xi)\ne 0$ where $(j_1,j_2),(j_1',j_2)\in \mathbb{Z}_{\gg}^2(\mathbb{F}).$ Here $j_1,j_1'$ can be same.
 If $t_2\in V(j_1,j_2,\xi)\ \text{and}\ t_2'\in V(j_1',j_2,\xi)$ for a fixed $j_2$, then it holds that
\begin{align}\label{890}
 R_1(t_2)=R_1(t_2').
\end{align}
Moreover, we have
  for a given pair of  $j_2$ and $\xi$,
\begin{align}\label{0ff}
	\left|\bigcup_{j_1:j_1\ge 100dj_2}V(j_1,j_2,\xi) \right|\le \text{deg}(R_1) \ \text{where $R_1$ in (\ref{584}).}\ 
\end{align}
\end{lemma}

\begin{proof}[Proof of (\ref{890})] 
In our proof below, we denote $B=O(A)$ if $|B|\le |A|$.
	Let $t_2\in V(j_1,j_2,\xi)$ and $t_2'\in V(j_1',j_2,\xi)$. Then for $j=(j_1,j_2)$ and $j'=(j_1',j_2)$,
	$$q_{j}(t_2)=q_{j'}(t_2')=1\ \text{ for $t_2,t_2'\sim 2^{j_2}$}.$$
	 Let $j_1\le j_1'$ without loss of generality. Then from (\ref{0sd}) and $q_j(t_2)=q_{j'}(t_2')=1$, it holds that
\begin{align*} 
 \xi_3R_1(t_2)+\xi_1=\xi(t_2)=\frac{a_j(t_2)}{q_j(t_2)}+\beta_j(t_2)=\text{integer}+O(2^{-(1-\frac{1}{100})j_1}).
\end{align*}
 Here, we can notice from (\ref{0sd})  that  $O(2^{-(1-\frac{1}{100})j_1})\le 2^{-(1-\frac{1}{100})j_1} $.  Similarly, 
\begin{align*} 
\xi_3R_1(t_2')+\xi_1=\xi(t_2')=\frac{a_{j'}(t_2')}{q_{j'}(t_2')}+\beta_{j'}(t_2')= \text{integer}+O(2^{-(1-\frac{1}{100})j_1}).
\end{align*}
since $2^{-(1-\frac{1}{100})j'_1}\le 2^{-(1-\frac{1}{100})j_1}$. Their difference is 
\begin{align}\label{499c}
 \xi_3(R_1(t_2)-R_1(t_2') )= \text{integer}+O(2\times 2^{-(1-\frac{1}{100})j_1}).
\end{align}
Nexxt, invoke  Dirichlet approximation in (\ref{5910}) for the support condition of $\xi_3$ in $\psi_j^{\rm{variant1}}(\xi)$. Then one has  $(q,a_3)$ with $\text{gcd}(a_3,q)=1$ with
$\beta_3=\xi_3-\frac{a_3}{q} $   such that  
	\begin{align}\label{591}
|\beta_3|\le \frac{1}{ q\cdot2^{(j\cdot \mathfrak{m}-j_1/10)}}\ \text{and}\ 2^{100 dj_2}<q<2^{j\cdot \mathfrak{m} -j_1/10}\le 2^{j_1+dj_2-j_1/10}\le 2^{j_1\frac{91}{100}}
\end{align}
where the  last two inequalities follow  from   $\mathfrak{m}=(1,n)$ and  $2^{j_1}\ge 2^{100d j_2}$.
On the other hand
$ |R_1(t_2)-R_1(t_2')| \le  2^{2j_2d}$ due to (\ref{0122}) and $t_2\sim 2^{j_2}$. This combined with   (\ref{591}) and $2^{(j\cdot \mathfrak{m}-j_1/10)} =2^{j_1+nj_2-j_1/10}\ge 2^{j_1-j_1/10}$,
\begin{align}\label{kjh1}
|\beta_3 (R_1(t_2)-R_1(t_2'))| \le \frac{2^{2dj_2}}{ q\cdot 2^{j\cdot \mathfrak{m}- j_1/10}}\le \frac{2^{2dj_2}}{q2^{(1-1/10)j_1}}.
\end{align}
	With this and the   size conditions    (\ref{499c}) and (\ref{kjh1}) , we   have 
	\begin{align}\label{04gg}
	 \frac{a_3}{q} (R_1(t_2)-R_1(t_2') )&= \xi_3(R_1(t_2)-R_1(t_2') )-\beta_3 (R_1(t_2)-R_1(t_2'))\nonumber\\
&= \text{integer}+O(2\times 2^{-(1-\frac{1}{100})j_1}) +O\left(\frac{2^{2dj_2}}{q2^{(1-1/10)j_1}}\right).
	 \end{align}
Assume that $R_1(t_2)-R_1(t_2')\ne 0$. Then from $\text{gcd}(a_3,q)=1$  in (\ref{04gg}) and $q>2^{100d j_2}\gg  2^{2dj_2}\ge  |R_1(t_2)-R_1(t_2')|\ge 1$, one can find the two integers
$\tilde{a}$ and $\tilde{q}$ satisfying
$$ \frac{a_3}{q} (R_1(t_2)-R_1(t_2') )=\frac{\tilde{a}}{\tilde{q}}\ \text{such that}\ 
 \text{gcd}(\tilde{q},\tilde{a})=1\ \text{and}\  \frac{q}{2^{j_2 2d}}\le \tilde{q}\le q.$$
	Now, we insert  this fraction $\frac{\tilde{a}}{\tilde{q}}$  into the LHS of  (\ref{04gg}) and obtain that  
$$\left|\frac{\tilde{a}}{\tilde{q}}-\text{integer}\right|=O(2\times 2^{-(1-\frac{1}{100})j_1}) +O(\frac{2^{2dj_2}}{q2^{(1-1/10)j_1}}).$$
However,
$|\frac{\tilde{a}}{\tilde{q}}-\text{integer}|\ge \frac{1}{\tilde{q}}\ge \frac{1}{q}$.  Thus, we have
$$   \frac{1}{q}\le  2\times 2^{-(1-\frac{1}{100})j_1}+\frac{2^{2dj_2}}{ q2^{(1-1/10)j_1}}\ \text{with}\   2^{100d j_2}<q<2^{91j_1/100}\
\text{ in (\ref{591})}.$$
The above inequality  can be rewritten as $$ \frac{1}{2q}\le \frac{1}{q}\left(1-\frac{2^{2dj_2}}{ q2^{(1-1/10)j_1}}\right)\le  2\times 2^{-(1-\frac{1}{100})j_1},$$
so that
$$\frac{1}{q}\le  4\times 2^{- \frac{99}{100}j_1} \ \text{and}\ 2^{100d j_2}<q<2^{91j_1/100}   $$ which leads   $(1/4)\times 2^{\frac{99}{100}j_1}\le  q<2^{\frac{91}{100}j_1}$. This is   a contradiction. Therefore,
$R_1(t_2)-R_1(t_2')=0$. This  yields (\ref{890}).  
\end{proof}
Now we can prove (\ref{0ff}).
\begin{proof}[Proof of (\ref{0ff})]
Fix  one number  $t_2'\in V(j_1,j_2,\xi)$. If $t_2\in \bigcup_{j_1:j_1\ge 100dj_2}V(j_1,j_2,\xi)$, then $R_1(t_2)-R_1(t_2')=0$ by Lemma \ref{lem944}. Therefore,   $t_2\in \bigcup_{j_1:j_1\ge 100dj_2}V(j_1,j_2,\xi)$ is  a solution of the polynomial equation
$t_2: R_1(t_2)-R_1(t_2')=0$. Hence, we obtain that $|\bigcup_{j_1:j_1\ge 100dj_2}V(j_1,j_2,\xi)|\le \text{deg}(R_1).$
\end{proof}

\begin{proof}[Proof of (\ref{0bb})]
	Let $\psi_{j}^{\rm{varinat1}}(\xi)\ne 0$.  Set $U(j_2,\xi)=\bigcup_{j_1:j_1\ge 100dj_2}V(j_1,j_2,\xi)$.  Fix $t_2'\in U(j_2,\xi)$. Then by (\ref{890}), for every $t_2\in  U(j_2,\xi)$,  it holds that $R_1(t_2)=R_1(t_2')$, that is $\xi(t_2)=\xi(t_2')$ because  $\xi(t_2)=\xi_3R_1(t_2)+\xi_1$.  
	Thus, one can express\begin{align}\label{613w}
 \xi(t_2)=c(j_2) \ \text{which is independent of $t_2$ and $j_1$} 
 \end{align}	
where
$ \{\xi(t_2): t_2\in \bigcup_{j_1:j_1\ge 100d j_2}V(j_1,j_2,\xi)\}$ is   a singleton for a fixed $\xi_1,\xi_3$ and $j_2$.

	Let $\chi_A$ be a characteristic function on $A$.  Then by using
	 $|U(j_2,\xi)|\le \text{deg}(R_1) $ in (\ref{0ff}) and  $\xi(t_2)=c(j_2)$ for $t_2\in U(j_2,\xi)$ in (\ref{613w}), one can obtain that
	\begin{align*} 
		 &\frac{1}{2^{j_2}} \sum_{j_1: j_1\ge 100dj_2}  \sum_{t_2\in V(j_1,j_2,\xi) }  \left| \sum_{t_1\sim 2^{j_1}}\frac{e^{-2\pi i  \xi(t_2)t_1 }}{t_1}\right|\\
		 &\le
		 \frac{1}{2^{j_2}} \sum_{t_2\sim 2^{j_2} } \sum_{j_1: j_1\ge 100dj_2} \chi_{U(j_2,\xi)}(t_2) \left| \sum_{t_1\sim 2^{j_1}}\frac{e^{-2\pi i  \xi(t_2)t_1 }}{t_1}\right| \\
		  &\le
		 \frac{1}{2^{j_2}}  \sum_{t_2\sim 2^{j_2} }  \chi_{U(j_2,\xi)}(t_2) \sum_{j_1: j_1\ge 100dj_2}  \left|\sum_{t_1\sim 2^{j_1}}\frac{e^{-2\pi i c(j_2)t_1 }}{t_1}\right|  
		  \le  C \frac{\text{deg}(R_1)}{2^{j_2}}  
	\end{align*}
 because $\sum_{1\le j_1<\infty}\left|\sum_{t_1\sim 2^{j_1}} \frac{e^{-2\pi i  a t_1 }}{t_1}\right|\le C$ with $C$ independent of $a$. So our proof of (\ref{0bb}) is completed.
\end{proof} 

\section{Major Arc Estimate ($\rm{Major}^{\sharp}$ and $\rm{Major}^{\flat}$ Estimates)}\label{Sec7}
Recall in (\ref{33}) that
$$\phi^{\rm{major}}_j(\xi)= \sum_{ 1\le q\le 2^{j_2/10}} \sum_{a_3: (a_3,q)=1} \phi\left(\frac{\xi_3-a_3/q }{2^{-j\cdot \mathfrak{m}}2^{j_1/10}}\right).$$ In this section,  we shall show that
 $$\sum_{j\in \mathbb{Z}^2(\mathbb{F})} \phi^{\rm{major}}_j(\xi)|H^{\rm{discrete}}_{j}(\xi)| \le C.$$
  \subsection{
 Major-$\sharp$-Arc Estimate}
 
\begin{proposition}\label{prop002} $[\rm{Major}^{\sharp}]$ 
There exists $C>0$ and $c>0$ independent of $\xi$ such that
\begin{align}\label{0401}
\sum_{j_1:j\in \mathbb{Z}^2(\mathbb{F})} \phi^{\rm{major,\sharp}}_j(\xi)\left| H^{\rm{discrete}}_{j}(\xi) \right|   \le C2^{-cj_2}. 
\end{align}
\end{proposition}

\begin{proof}
Recall (\ref{010}) and (\ref{011}). By Lemma \ref{lem002} and  $\beta_i=\xi_i-\frac{a_i}{q}  $ for $i=1,2,3$, we have
\begin{align}\label{100p}
   \phi^{\rm{major,\sharp}}_j(\xi)\left| H^{\rm{discrete}}_{j}(\xi) \right|
&  \lesssim   \sum_{ 1\le q\le 2^{j_2/10}} \sum_{a: (a_3,q)=1}  \phi_{j,q,a}(\xi)  \phi^c\left( \frac{  \beta_3 2^{\mathfrak{m}\cdot j} }{2^{j_2/10}} \right)\\
&\times\sum_{t_2\sim 2^{j_2}} \frac{1}{2^{j_2}}\left(|\mathcal{H}^{t_2}_{j_1}(\beta_1,0,\beta_3) |+ 2^{-cj_1} \right). \nonumber
\end{align}
By utilizing 
$\phi^c\left(\frac{\beta_32^{j\cdot \mathfrak{m}}}{2^{j_2/10}}\right)  |\mathcal{H}^{t_2}_{j_1}(\beta_1,0,\beta_3) |=O(2^{-cj_2})$  for $c>0$ in the first part of (\ref{015}),  
\begin{align*} 
&   |\mathcal{H}^{t_2}_{j_1}(\beta_1,0,\beta_3) |  
 \le 2^{-cj_2/2}|\mathcal{H}^{t_2}_{j_1}(\beta_1,0,\beta_3) |^{1/2}\le  q^{-c}2^{-cj_2/4}|\mathcal{H}^{t_2}_{j_1}(\beta_1,0,\beta_3) |^{1/2}
\end{align*}
where $ q\le 2^{j_2/10}$.
By this  we obtain that  up to $O(2^{-cj_2})$, 
\begin{align*}
 \sum_{j_1:j\in \mathbb{Z}^2(\mathbb{F})}RHS\ \text{of}\ (\ref{100p})& \le \sum_{ 1\le q\le 2^{j_2/10}} \sum_{a: (a_3,q)=1}  \phi_{j,q,a}(\xi)  \frac{q^{-c}2^{-j_2}}{2^{cj_2/4}}\sum_{t_2\sim 2^{j_2}}  \sum_{j_1\ge 0}  |\mathcal{H}^{t_2}_{j_1}(\beta_1,0,\beta_3) |^{1/2} \\
& \le  \sum_{ 1\le q\le 2^{j_2/10}} \sum_{a: (a_3,q)=1}  \phi_{j,q,a}(\xi)   \frac{q^{-c}}{2^{ cj_2/4}}  \le C2^{-cj_2/4}
\end{align*}
where the second inequality  follows from the second part of (\ref{015}) and the third inequality follows from the dyadic distribution of $q$ for the major arc in the second part of Lemma \ref{le32}.
This is our desired result.
\end{proof}
 \subsection{
 Major-$\flat$-Arc Estimate}
To prove 
 $  \sum_{j\in \mathbb{Z}^2(\mathbb{F})} |H^{\rm{discrete}}_{j}(\xi)  |\phi^{\rm{major,\flat}}_j(\xi_3)\le C $,
 we are able to reduce matters to the two-parameter continuous case.
 \begin{proposition}\label{prop003} $[\rm{Major}^{\flat}]$ 
Let $ \beta_i = \xi_i-\frac{a_i}{q}  $ for $i=1,2,3$ and recall that $$\phi^{\rm{major,\flat}}_j(\xi)=\sum_{ 1\le q\le 2^{j_2/10}} \sum_{a: (a_3,q)=1}   \phi_{j,q,a}(\xi)  \phi\left( \frac{ \beta_3  2^{\mathfrak{m}\cdot j} }{2^{j_2/10}} \right). $$ If $\phi_{j,q,a}(\xi)  \phi\left( \frac{ \beta_3  2^{\mathfrak{m}\cdot j} }{2^{j_2/10}} \right)\ne 0$, then it holds that \begin{align}\label{029}
& H^{\rm{discrete}}_{j}(\xi)  =  S\left(\frac{a}{q}\right)\mathcal{H}_{j}^{\Lambda(P)}\left( \beta\right) +   E_j(a,q,\xi) 
\end{align}
where the  sum of the error terms is estimated by
\begin{align*}
 \sum_{j_1:j\in \mathbb{Z}^2(\mathbb{F})}\sum_{ 1\le q\le 2^{j_2/10}} \sum_{a: (a_3,q)=1} \phi_{j,q,a}(\xi)  \phi\left( \frac{ \beta_3  2^{\mathfrak{m}\cdot j} }{2^{j_2/10}} \right)|E_j(a,q,\xi)|=O(2^{-j_2/10}).
\end{align*}
Finally,   suppose that  ${\bf N}(P,D_{\{-{\bf e}_1,-{\bf e}_2\}})$ has the vertex-evenness property.  Then there exists a constant $C>0$ independent of $\xi$ such that
\begin{align}\label{943}
\sum_{j\in \mathbb{Z}^2(\mathbb{F})}|H^{\rm{discrete}}_{j}(\xi)  |\phi^{\rm{major,\flat}}_j(\xi) \le C.
\end{align}
\end{proposition}
 
\begin{proof}[Proof of (\ref{029}) ]
Recall
\begin{align*}
H^{\rm{discrete}}_{j}(\xi)=&  \sum_{t_1}\sum_{t_2} e^{-2\pi i (\xi_1t_1+\xi_2t_2+\xi_3 P(t_1,t_2) )}\frac{\chi_{j_1}(t_1)}{t_1}\frac{\chi_{j_2}(t_2)}{t_2}
\end{align*} 
and apply the change of variable $t_1=\mu_1q+\ell_1$  and $t_2=\mu_2q+\ell_2$. Then  for every $\xi$ such that  $ \phi_{j,q,a}(\xi)  \phi\left( \frac{ \beta_3  2^{\mathfrak{m}\cdot j} }{2^{j_2/10}} \right) \ne 0$, we have
\begin{align*}
H^{\rm{discrete}}_{j}(\xi)& =\sum_{\ell_1,\ell_2=1}^q\sum_{\mu_1,\mu_2}    e^{-2\pi i  (\beta +\frac{a }{q})\cdot (\mu_1q+\ell_1,\mu_2q+\ell_2,P(q\mu_1+\ell_1,q\mu_2+\ell_2) )} \\
&\times\frac{\chi_{j_1}( q\mu_1+\ell_1 )}{q\mu_1+\ell_1}\frac{\chi_{j_2}( q\mu_2+\ell_2 )}{q\mu_2+\ell_2}\\
& =\frac{1}{q^2} \sum_{\ell_1,\ell_2=1}^q   e^{-2\pi i    \frac{a }{q}\cdot (\ell_1,\ell_2, P( \ell_1 ,\ell_2)) }q^2  \sum_{(\mu_1,\mu_2)\in\mathbb{Z}^2}F(\mu_1,\mu_2) 
\end{align*}
where the function $F$  is given by $F(\mu_1,\mu_2) $:
$$   e^{-2\pi i   \beta \cdot  (\mu_1q+\ell_1,\mu_2q+\ell_2,P(q\mu_1+\ell_1,q\mu_2+\ell_2) )} \times\frac{\chi_{j_1}( q\mu_1+\ell_1 )}{q\mu_1+\ell_1}\frac{\chi_{j_2}( q\mu_2+\ell_2 )}{q\mu_2+\ell_2}.$$
Next, we are able to apply the Poisson summation formula for $F\in C_0^\infty(\mathbb{R}^2)$, $$\sum_{(\mu_1,\mu_2)\in\mathbb{Z}^2}F (\mu_1,\mu_2) =\sum_{ (w_1,w_2)\in\mathbb{Z}^2}F^{\vee} (w_1,w_2) $$
where $F^{\vee}(w_1,w_2) $,   via the coordinate change  $(q\mu_1+\ell_1, q\mu_2+\ell_2)\rightarrow (y_1,y_2)$, is
\begin{align*} 
  q^{-2} e^{2\pi i\frac{  ( w_1\ell_1 +w_2\ell_2)}{q} }  \int e^{-2\pi i \left( \frac{(w_1 y_1 +w_2 y_2)}{q} +\beta \cdot (y_1,y_2, P(y_1,y_2))\right)}\frac{\chi_{j_1}( y_1)}{y_1}\frac{\chi_{j_2}( y_2)}{y_2}
dy_1dy_2. 
\end{align*}
Thus, we can obtain  (\ref{029}) with
$$E_j(a,q,\xi): = \sum_{w\in \mathbb{Z}^2\setminus\{(0,0)\} } S\left(\frac{(a_1-w_1,a_2-w_2,a_3)}{q}\right)\mathcal{H}_{j}^{\Lambda(P)}\left(\frac{w_1}{q}+\beta_1,\frac{w_2}{q}+\beta_2,\beta_3 \right). $$
To show (\ref{943}), we shall claim that
\begin{align}\label{4times}
\left|\mathcal{H}_{j}^{\Lambda(P)}\left( \frac{w_1}{q}+\beta_1,\frac{w_2}{q}+\beta_2, \beta_3 \right)\right|^{1/2} \lesssim (|2^{j_2/10}(w_1,w_2)|+1)^{-5/2}.
  \end{align}
  \begin{proof}[Proof of (\ref{4times})]
Let $\Phi(w,\beta,y): =  (\frac{w_1}{q} +\beta_1) y_1  +(\frac{w_2}{q} +\beta_2)  y_2+\beta_3P( y_1, y_2)$. Then  
\begin{align*}
&\mathcal{H}_{j}^{\Lambda(P)}\left( \frac{w_1}{q}+\beta_1,\frac{w_2}{q}+\beta_2, \beta_3 \right)  =\int e^{2\pi i\Phi(w,\beta,y) }\frac{\chi_{j_1}(y_1)}{y_1}\frac{\chi_{j_2}(y_2)}{y_2}
dy_1dy_2.
\end{align*}
By $\beta_32^{\mathfrak{m}\cdot j} \le 2^{j_2/10}$ in  (\ref{029}) and $1\le q\le 2^{j_2/10}$,  one  has the lower bounds of partial derivatives of the phase:
\begin{align}\label{phase}
|\partial_{y_1} \Phi(w,\beta,y)|=\left| \frac{w_1}{q}+\beta_1+O( \frac{ \beta_3  2^{j\cdot \mathfrak{m}} }{2^{j_1}}) \right|\ge \frac{1}{4} \left| \frac{w_1}{q}\right|\ge \frac{1}{4} \left| \frac{w_1}{2^{j_2/10}}\right|\\
|\partial_{y_2}\Phi(w,\beta,y)|=\left| \frac{w_2}{q}+\beta_2+O( \frac{ \beta_3  2^{j\cdot \mathfrak{m}} }{2^{j_2}}) \right|\ge \frac{1}{4} \left| \frac{w_2}{q} \right|\ge \frac{1}{4}\left| \frac{w_2}{2^{j_2/10}}\right|\notag
\end{align}
 Moreover, by  (\ref{aria})  one has
\begin{align*}
\partial_{y_1}^{m}\left( \frac{\chi_{j_1}(y_1)}{y_1} \right)\lesssim \frac{  \chi_{[2^{j_1},2^{j_1}+1/2]}(|y_1|)+ \chi_{[2^{j_1-1}-1/2,2^{j_1-1}]}(|y_1|) }{2^{j_1}} +\frac{\chi_{j_1}(y_1)}{2^{(1+m)j_1}}\\
   \partial_{y_2}^{m}\left( \frac{\chi_{j_2}(y_2)}{y_2} \right)\lesssim \frac{  \chi_{[2^{j_2},2^{j_2}+1/2]}(|y_2|)+ \chi_{[2^{j_2-1}-1/2,2^{j_2-1}]}(|y_2|) }{2^{j_2}} +\frac{\chi_{j_2}(y_2)}{2^{(1+m)j_2}}\notag 
    \end{align*}
    where $\chi_{[a,b]}$ is a characteristic function. 
This with   (\ref{phase})  enables us to apply the   integration by parts $5$ times   for the above integrals $\int dy_1dy_2$ to have (\ref{4times}).
\end{proof}
By (\ref{4times}), we obtain that
\begin{align*} 
|E_j(a,q,\xi)|\le \sum_{w\in \mathbb{Z}^2\setminus\{(0,0)\}}  (|2^{j_2/10}(w_1,w_2)|+1)^{-5/2} \left|\mathcal{H}_{j}^{\Lambda(P)}\left( \frac{w_1}{q}+\beta_1,\frac{w_2}{q}+\beta_2, \beta_3 \right) \right|^{1/2}.
\end{align*}
Using this  with  $ \sum_{ 1\le q\le 2^{j_2/10}} \sum_{a: (a_3,q)=1}   \phi_{j,q,a}(\xi)  \phi\left( \frac{ \beta_3  2^{\mathfrak{m}\cdot j} }{2^{j_2/10}} \right) \le 1$ and the uniform boundedness of 
 $\sum_{j_1\in \mathbb{Z}_+}\left|\mathcal{H}_{j}^{\Lambda(P)}\left( \frac{w_1}{q}+\beta_1,\frac{w_2}{q}+\beta_2, \beta_3 \right) \right|^{1/2}\lesssim 1$ in $\beta,w,q$ and  $j_2$ in Proposition \ref{prop006}, we obtain the decay $2^{-cj_2}$ in the error term estimate of (\ref{029}). 
\end{proof}
\begin{proof}[Proof of (\ref{943})]
By (\ref{029}), it suffices to show that 
\begin{align}\label{494}
\sum_{j\in\mathbb{Z}^2(\mathbb{F})} \sum_{ 1\le q\le 2^{j_2/10}} \sum_{a: (a_3,q)=1}   \phi_{j,q,a}(\xi)  \phi\left( \frac{ \beta_3 2^{\mathfrak{m}\cdot j} }{2^{j_2/10}} \right)|S\left(\frac{a}{q}\right)\mathcal{H}_{j}^{\Lambda(P)}\left( \beta\right)|\le C
\end{align} 
where $\beta=\xi-a/q$. For the estimate (\ref{494}), we utilize the boundedness $\sum_{j} |\mathcal{H}_{j}^{\Lambda(P)}\left(  \beta  \right)  | $ of the continuous double Hilbert transform in Proposition \ref{prop006}    under the vertex evenness assumption. This is the only place where we need the     the vertex evenness assumption. Next by using   the Gauss sum decay $O(q^{-c})$ as above and the dyadic distribution of $q$  associated with major arc in Lemma \ref{le32}, to obtain (\ref{494}), which implies (\ref{943}).
\end{proof}
 This completes the proof of Proposition \ref{prop003}. So,
we have proved Propositions \ref{minor arc}, \ref{prop001}, \ref{prop885}, \ref{prop002} and \ref{prop003}. Therefore, we proved the sufficient part of Main Theorem 1 and also finished the proof of the $\ell^2$ boundedness of ${\bf H}^{\rm{discrete}}$ in Main Theorem 2.
 \section{$\ell^{p}$ boundedness of ${\bf H}^{\rm{discrete}}$}\label{Sec8}
Until  the last section, under the vertex-evenness assumption,  we have proved that
$$\sum_{j\in \mathbb{Z}_+^2} |H_{j}^{\rm{discrete}}(\xi_1,\xi_2,\xi_3)|\le C$$
which implies (\ref{15p}). 
By this with ${\bf H}^{\rm{discrete}}_{(N_1,N_2)}(f)(x) = [H^{\rm{discrete}}_{(N_1,N_2)}]^{\vee}*f(x)$  for  $x\in \mathbb{R}^3$, we have
$$ {\bf H}^{\rm{discrete}} f(x)= \int e^{2\pi i x\cdot \xi}\sum_{j\in \mathbb{Z}^2}H_{j}^{\rm{discrete}}(\xi_1,\xi_2,\xi_3)\widehat{f}(\xi)d\xi.$$
 In this section we  prove  Main Theorem \ref{main} by showing under the evenness hypothesis,
\begin{align*}
\|{\bf H}^{\rm{discrete}} \|_{\ell^p(\mathbb{Z}^3)\rightarrow \ell^p(\mathbb{Z}^3)}\le C\ \text{for all $1<p<\infty$}
\end{align*}
    by adapting the similar argument in \cite{IW}. 
    Assume that $m: \mathbb{R}^{3} \rightarrow \mathbb{C}$ is a bounded  function supported in the cube $[-1 / 2,1 / 2]^{3}$ and is periodically extended to $\mathbb{R}^3$ with the property that for any $p \in(1, \infty)$,
\begin{align}\label{IW1}
\left\|(m \cdot \widehat{g})^{\vee}\right\|_{L^{p}\left(\mathbb{R}^{d}\right)} \leq B_{p}\|g\|_{L^{p}\left(\mathbb{R}^{d}\right)}
\end{align}
for all Schwartz function $g: \mathbb{R}^{3} \rightarrow \mathbb{C}$. Given a finite set  $Y\subset \mathbb{N}$,  let $$\mathcal{R}(Y)=\left\{a/ q: q \in Y, a=(a_1,a_2,a_3)  \in \mathbb{Z}^{3},\text{gcd}(a, q)=1\right\}.$$
Given this $Y$ and  $\eta \in(0,1]$, we define 
\begin{align}\label{91k}
m_{\eta, Y}(\xi)=\sum_{a / q \in \mathcal{R}(Y)} m\left(\frac{\xi-a / q}{\eta}\right),
\end{align}
which is periodic as $m_{\eta, Y}(\xi+n)=m_{\eta, Y}(\xi)$ for all $n\in \mathbb{Z}^3$. Denote $Z_N=\mathbb{Z}\cap [1,N].$
 \begin{lemma}\label{lem82}
 \cite{IW} For any $\delta>0$ and $p \in(1, \infty)$, there are constants $A_{\delta}$ and $C_{p, \delta}$ having the following property: for any $N \geq A_{\delta}$ there is a set of integers $Y_{N}=Y_{N, \delta}$,
 $$
 Z_{N} \subset Y_{N} \subset Z_{e^{N^{\delta}}},
 $$
satisfying that  the operator $T_{N}=T_{N, \eta}$, for any $\eta \leq e^{-N^{2 \delta}}$, defined by the Fourier multiplier $m_{\eta, Y_{N}}$ in (\ref{91k}) extends to a bounded operator on $\ell^{p}\left(\mathbb{Z}^{d}\right)$, with
\begin{align*}
 \left\|T_{N}(f)\right\|_{\ell^{p}\left(\mathbb{Z}^{d}\right)} \leq C_{p, \delta}(\ln N)^{2 / \delta}\|f\|_{\ell^{p}\left(\mathbb{Z}^{d}\right)} .
\end{align*}
 The constant $A_{\delta}$ may depend only on $\delta$ and $d$; the constant $C_{p, \delta}$ may depend only on $\delta, d$, the exponent $p$, and the constants $B_{p}$ in (\ref{IW1}).
\end{lemma}
  \begin{remark}
  This result was further developed in a more general setting in \cite{Tao}, where an improved bound was established. Nonetheless, for our purposes, the original version in \cite{IW} is sufficient.
  \end{remark}
 \subsection{Asymptotic Estimate}
Given $j=(j_1,j_2)$ and $(q,a)\in \mathbb{Z}\times \mathbb{Z}^3$ with $q\ge 1$,  set  
\begin{align}\label{36aa}
\Phi_{j}(\xi-a/q):=\phi\left( \frac{\xi_1-a_1/q}{2^{-j_1(9/10)}} \right) \phi\left( \frac{\xi_2-a_2/q}{2^{-j_2(9/10)}} \right) \phi\left(\frac{\xi_3-a_3/q }{2^{-j\cdot \mathfrak{m}} 2^{j_2(1/10)}}\right).
\end{align}
  Then for $\xi \in \mathbb{R}^3$,
 \begin{align}\label{8hb}
\sum_{1\le q\le 2^{j_2/10}}\sum_{a:(a,q)=1}\Phi_{j}(\xi-a/q)\le 1 .\end{align}
The above inequality follows since the supports  of   $\xi\rightarrow\Phi_{j} (\xi-a/q)$  given by $$ (-2^{-(9/10)j_1}, 2^{-(9/10)j_1}) \times  (-2^{-(9/10)j_2}, 2^{-(9/10)j_2}) \times  (-2^{-j\cdot \mathfrak{m}}2^{j_2/10},2^{-j\cdot \mathfrak{m}}2^{j_2/10})+\frac{a}{q}$$  are  disjoint  in $(q,a)$. Here, we used the fact that $|a/q-a'/q'|\ge 1/|qq'|\ge 2^{-j_2/5}$ while  every side length of the box is smaller than   $2^{-(9/10)j_2}$.  
Moreover, 
\begin{align}\label{8hb1}
\sum_{1\le q\le 2^{j_2/10}}\sum_{a:(a,q)=1}  \Phi_{j}(\xi-a/q)\le \sum_{1\le q\le 2^{j_2/10}}\sum_{a:(a,q)=1}\prod_{i=1}^3\phi\left( \frac{\xi_i-a_i/q}{q^{-4}} \right). \end{align}
In the RHS of (\ref{8hb1}), observe that, as in Lemma \ref{le32},  $q> q'$ implies  $q>5q'$.  
We start with proving the following   asymptotic formula.
\begin{proposition}\label{prop812}
Let $\Phi^{\rm{major}}_j(\xi):=  \sum_{1\le q\le 2^{j_2/10}}\sum_{a:(a,q)=1}\Phi_{j}(\xi-a/q)$, which is defined due to (\ref{8hb}). Then
\begin{align}\label{59}
	&\sum_{j_1:j \in \mathbb{Z}^2(\mathbb{F})} |H^{\rm{discrete}}_{j}(\xi)| (1-\Phi_j^{\rm{major}}(\xi))=O(2^{-cj_2}).
\end{align}

\end{proposition}
\subsection{Proof of Proposition \ref{prop812}} 
Recall    $\phi^{\rm{major},\flat}_j(\xi)$ in (\ref{34p}) and rewrite it as
\begin{align}\label{mfy}
 \phi^{\rm{major},\flat}_{3,j}(\xi_3):= \sum_{ 1\le q_3\le 2^{j_2/100}} \sum_{a_3: (a_3,q_3)=1}\phi\left( \frac{ (\xi_3-a_3/q_3) 2^{\mathfrak{m}\cdot j} }{2^{j_2(\eta/2)}} \right)
\end{align}
where $\eta$ is taken to be a small positive number later.  Then we are able to apply (\ref{0053}),(\ref{082}),(\ref{mmex}),(\ref{9437}) and (\ref{0401}) to observe that for some $c>0$,
\begin{align} \label{8181}
&	 \sum_{j_1:j \in \mathbb{Z}^2(\mathbb{F})} |H^{\rm{discrete}}_{j}(\xi)| (1-\phi^{\rm{major},\flat}_{3,j}(\xi_3))=
O(2^{-cj_2}).
\end{align}
Set
\begin{align}\label{85ss}
\phi^{\rm{major},\flat}_{1,j_1}(\xi_1)= \sum_{ 1\le q_1\le 2^{j_2/50}} \sum_{a_1: (a_1,q_1)=1} \phi\left(\frac{\xi_1-a_1/q_1}{2^{-(9/10)j_1}}\right).
\end{align}
Fix $j_2$. Then we firstly claim that there exists $C,c>0$ independent of $j,\xi$ such that
\begin{align}\label{5009}
\sum_{j_1:j \in \mathbb{Z}^2(\mathbb{F})} |H^{\rm{discrete}}_{j}(\xi)| (1-\phi^{\rm{major},\flat}_{1,j_1}(\xi_1)) \phi^{\rm{major},\flat}_{3,j}(\xi_3)\le 
C2^{-cj_2}.
\end{align}
\begin{proof}[Proof of (\ref{5009})]
Using the Dirichlet approximation to $\xi_1$ satisfying $1-\phi^{\rm{major},\flat}_{1,j_1}(\xi_1)\ne 0$ in (\ref{85ss}),  
  choose $q_1=q_1(\xi_1,j)\in [2^{j_2/50}, 2^{(9/10)j_1}]$ and $a_1=a_1(\xi_1,j)$ so that 
\begin{align}\label{pjk}
 |\xi_1-a_1/q_1|<2^{-(9/10)j_1}\ \text{and}\ \text{gcd}(q_1,a_1)=1.
\end{align}
Consider the case  $q_1  \in [2^{j_1/10},2^{j_1(9/10)}]$  in (\ref{pjk}). For this case with  $t_2\sim 2^{j_2}$ fixed, we shall show the 1D-Weyl-sum estimate:
\begin{align}\label{8898}
\sum_{t_1\sim 2^{j_1}} e^{2\pi i[\xi_1t_1+\xi_2t_2+\xi_3 P(t_1,t_2) ]}\le 2^{j_1-cj_1}
\ \text{ for some $c>0$.}\end{align} 
To claim (\ref{8898}), from $P(t_1,t_2)=\sum_{m\ge 0} R_m(t_2)t_1^m$, we first write the phase as
$$\xi_1t_1+\xi_2t_2+\xi_3 P(t_1,t_2)= \sum_{m\ge 2} \xi_3 R_m(t_2) t_1^m+( \xi_1+\xi_3R_1(t_2))t_1+\xi_2t_2+\xi_3R_0(t_2).$$
Using  (\ref{mfy}) and the fact $|R_1(t_2)t_1|\lesssim  \sum_{\mathfrak{n}\in\Lambda(P)}|t^{\mathfrak{n}}|\sim 2^{j\cdot \mathfrak{m}}$ in (\ref{jq501}), we have
$$\left|\left( \xi_3-\frac{a_3}{q_3}\right)R_1(t_2)\right|\le C\frac{2^{\eta j_2/2}}{2^{j\cdot \mathfrak{m}}}\frac{2^{j\cdot \mathfrak{m}}}{2^{j_1}}\le 2^{-9j_1/10}.$$
By this and (\ref{pjk}),    approximate   the coefficient $\xi_1+\xi_3R_1(t_2)$  of $t_1^1$ for the above phase,   $$\left|\left(\xi_1+\xi_3R_1(t_2)\right)-\left(\frac{a_1}{q_1}+\frac{a_3}{q_3}R_1(t_2)\right)\right|\le 2\times 2^{-9j_1/10} .$$
Since $2^{j_1/10}\le q_1\le 2^{9j_1/10}$ and $q_3\le 2^{j_2/100}$ with $j_1\ge j_2$ and $\text{gcd}(a_1,q_1)=1$, one can choose an integer pair  $\tilde{a},\tilde{q}$ with $\text{gcd}(\tilde{a},\tilde{q})=1$ satisfying that 
 $$ \frac{a_1}{q_1}+\frac{a_3}{q_3}R_1(t_2)=\frac{a_1q_3+ q_1 a_3R_1(t_2)}{q_1q_3}=\frac{\tilde{a}}{\tilde{q}}\ \text{and}\ 2^{j_1/20}\le \tilde{q}\le 2^{j_1-j_1/20}.$$
 This yields the desired 
Weyl sum decay in (\ref{8898}). \\
Next, consider the case  $q_1 \in [2^{j_2/50},2^{j_1/10}] $  in (\ref{pjk}). Recall    $(q_3,a_3) $     in (\ref{mfy}):
\begin{align}\label{krom1}
|\xi_3-a_3/q_3|<2^{-j\cdot \mathfrak{m}}2^{(\eta/2)j_2}\ \text{and}\ \text{gcd}(q_3,a_3)=1\ \text{with}\ 1\le q_3\le 2^{j_2/100}. 
\end{align}
Take $\tilde{q}=\text{lcm}(q_1,q_3)$. Indeed, write $\tilde{q}=q_1\alpha_1,\tilde{q}= q_3 \alpha_3$ with $1\le \alpha_1\le 2^{j_2/100}$ and $1\le \alpha_3$.  Next put $(\tilde{q},\tilde{a}_1)=\alpha_1(q_1,a_1)$   and $(\tilde{q},\tilde{a}_3)=\alpha_3(q_3,a_3)$  so that $\text{gcd}(\tilde{q},\tilde{a}_1,\tilde{a}_3)=1$. Then in (\ref{pjk}) and (\ref{krom1}), it holds that
\begin{align}\label{kom1}
 |\xi_1-\frac{\tilde{a}_1}{\tilde{q}}|<2^{-(9/10)j_1}\ \text{and}\ |\xi_3-\frac{\tilde{a}_3}{\tilde{q}}|<2^{-j\cdot \mathfrak{m}}2^{(\eta/2)j_2}\ \text{with} \ q_1\le \tilde{q}\le q_1^2.
\end{align} 
 For fixed $\xi,j$ with $\tilde{q},\tilde{a}_1,\tilde{a}_3$ in  (\ref{kom1}) where $\tilde{\beta}_\nu= \xi_\nu- \frac{\tilde{a}_{\nu}}{\tilde{q}}=\xi_\nu- \frac{a_{\nu}}{q}$ ($\nu=1,3$), by modification of Lemma \ref{lem002}, one has
 \begin{align}
  H^{\rm{discrete}}_{j}(\xi)   =  \label{bu9} 
 \sum_{t_2\sim 2^{j_2}} \frac{ e^{2\pi i\xi_2t_2}}{t_2}  S^{t_2}\left(\frac{(\tilde{a}_1,0,\tilde{a}_3)}{\tilde{q}}\right)\mathcal{H}_{j_1}^{t_2}(\tilde{\beta}_1,0,\tilde{\beta}_3)+O(2^{-cj_1}).  
\end{align}
Moreover, one has the average Gauss-sum estimate:
\begin{align}\label{cco}
\frac{1}{2^{j_2}}\sum_{t_2\sim 2^{j_2}}|S^{t_2}\left(\frac{(\tilde{a}_1,0,\tilde{a}_3)}{\tilde{q}}\right)|\le O(|\tilde{q}|^{-\delta})\ \text{with $q_1\le \tilde{q}\le q_1^2$}.
\end{align}
\begin{proof}[Proof of (\ref{cco})]
Let $P(t_1,t_2)\in \mathbb{Z}^{\ge 2}[t_1,t_2]$. Then by the condition $ \text{gcd}(\tilde{a}_1,\tilde{q})=\alpha_1\le 2^{j_2/100}$ with $\tilde{q}\ge 2^{j_2/50}$ satisfying the hypothesis above (\ref{110}), one can  apply (\ref{110})   to obtain (\ref{cco}). Let $P(t_1,t_2)\in \mathbb{Z}^{1}[t_1,t_2]$. If $\tilde{q}\le 2^{100dj_2 }$, then by the condition $ \text{gcd}(\tilde{a}_1,\tilde{q})=\alpha_1\le 2^{j_2/100}$ with $ 2^{j_2/50}\le \tilde{q}\le 2^{100dj_2 }$, we are able to apply (\ref{110a})   to obtain (\ref{cco}).  If  $\tilde{q}> 2^{100dj_2 }$, then $q_1>2^{99dj_2 }$ and $1\le q_3<2^{j_2/100}$, which implies that $$\frac{\tilde{a}_1+\tilde{a}_3R_1(t_2)}{\tilde{q}}=\frac{a_1}{q_1}+\frac{a_3}{q_3}R_1(t_2)=\frac{a_1q_3+a_3q_1R_1(t_2)}{q_1q_3}=\frac{\tilde{a}(t_2)}{\tilde{q}(t_2)}\ \text{with $\text{gcd}(\tilde{a}(t_2),\tilde{q}(t_2))=1$
}$$ 
satisfying $\tilde{q}^{1/2}\le \tilde{q}(t_2)\le \tilde{q}$ and $\tilde{q}(t_2)|\tilde{q}$.  For this case, 
$$ S^{t_2}\left(\frac{(\tilde{a}_1,0,\tilde{a}_3)}{\tilde{q}}\right)=\frac{1}{\tilde{q}}\sum_{t_1=1}^{\tilde{q}} e^{2\pi i t_1  \left(\frac{\tilde{a}_1+\tilde{a}_3R_1(t_2)}{\tilde{q}} \right)}=0.$$
This completes the proof of   (\ref{cco}).\end{proof}
Recall that $q_1 \in [2^{j_2/50},2^{j_1/10}]$   and $q_3 \in [1,2^{j_2/100}]$ and  $\tilde{\beta}_1 , \tilde{\beta}_3$ depend on $q_1,q_3$ such that    $\tilde{\beta}_1 =\xi_1-a_1/q_1$ and $\tilde{\beta}_3=\xi_3-a_3/q_3$ satisfying the support condition (\ref{kom1}).
Then by similar arguments in Section \ref{Sec6} combined with (\ref{bu9}),(\ref{cco}) and (\ref{015}), we can obtain that for some $c>0$,
\begin{align*} 
 \text{LHS of (\ref{5009}) }&\lesssim  \sum_{q_1,q_3; dyadic} 2^{-cj_2} q_1^{-\delta/4} q_3^{-\delta/2} \lesssim 2^{-cj_2}
\end{align*}
which  yields (\ref{5009}).
\end{proof} 
Next, we define
\begin{align}\label{82tt}
\phi^{\rm{major},\flat}_{2,j_2}(\xi_2):= \sum_{ 1\le q_2\le 2^{j_2/20}} \sum_{a_2: (a_2,q_2)=1} \phi\left(\frac{\xi_2-a_2/q_2}{2^{-(9/10)j_2}}\right) 
\end{align}
and
  claim that for some $c>0$,
\begin{align}\label{5449}
\sum_{j_1:j \in \mathbb{Z}^2(\mathbb{F})} |H^{\rm{discrete}}_{j}(\xi)| (1-\phi^{\rm{major},\flat}_{2,j_2}(\xi_2))\phi^{\rm{major},\flat}_{1,j_1}(\xi_1) \phi^{\rm{major},\flat}_{3,j}(\xi_3)=
O(2^{-cj_2}).
\end{align} \begin{proof}[Proof of (\ref{5449})]
Apply Lemma \ref{lem002} for the support condition  $\phi^{\rm{major},\flat}_{1,j_1}(\xi_1) \phi^{\rm{major},\flat}_{3,j}(\xi_3)$ with $\tilde{q}=\text{lcm}(q_1,q_2)$ where $q_1<2^{j_2/50}$ and $q_3<2^{j_2/100}$ satisfying (\ref{pjk}) and (\ref{krom1}). Then it suffices to consider the following sums:
 \begin{align}
&\sum_{t_2\sim 2^{j_2}} \frac{ e^{2\pi i\xi_2t_2}}{t_2}  
S^{t_2}\left(\frac{(\tilde{a}_1,0,\tilde{a}_3)}{\tilde{q}}\right)\mathcal{H}_{j_1}^{t_2}(\tilde{\beta}_1,0,\tilde{\beta}_3)\nonumber\\
&=\frac{1}{\tilde{q}}\sum_{\ell_1=1}^{\tilde{q}}e^{2\pi i [\frac{\tilde{a}_1}{\tilde{q}}\ell_1+\frac{\tilde{a}_3}{\tilde{q}}Q_0(\ell_1)]} \left(\sum_{t_2\sim 2^{j_2}} e^{2\pi i [(\xi_2+\frac{\tilde{a}_3}{\tilde{q}}Q_1(\ell_1))t_2+()t_2^2+\cdots]} \frac{  \mathcal{H}_{j_1}^{t_2}(\tilde{\beta}_1,0,\tilde{\beta}_3)}{t_2}\right)\label{co24}
\end{align}
where $ \frac{\tilde{a}_3}{\tilde{q}}=\frac{a_3}{q_3}$ and $P(\ell_1,t_2)= Q_0(\ell_1)+Q_1(\ell_1)t_2+\cdots+Q_n(\ell_1)t_2^n$. 
If $(1-\phi^{\rm{major},\flat}_{2,j_2}(\xi_2))\ne 0$, then we can take  $q_2=q_2(\xi_2,j)$ and $a_2=a_2(\xi_2,j)$ such that
\begin{align}\label{pjk2}
 |\xi_2-a_2/q_2|<2^{-(9/10)j_2}\ \text{and}\ \text{gcd}(q_2,a_2)=1\ \text{with}\ q_2\in [2^{j_2/20}, 2^{(9/10)j_2}].\end{align}
Fix $\ell_1$ in (\ref{co24}). By (\ref{pjk2}), one knows that $$\left|(\xi_2+\frac{\tilde{a}_3}{\tilde{q}}Q_1(\ell_1))-\left(\frac{a_2}{q_2}+\frac{a_3}{q_3}Q_1(\ell_1)\right)\right|<2^{-(9/10)j_2}$$
where $q_2>2^{j_2/20}$ and $q_3<2^{j_2/100}$ give us the following  reduced fraction:
$$\left(\frac{a_2}{q_2}+\frac{a_3}{q_3}Q_1(\ell_1)\right)=\frac{a_2(\ell_2)}{q_2(\ell_2)}\ \text{with}\ q_2(\ell_2)>2^{j_2/30}\ \text{and}\ \text{gcd}(a_2(\ell_2),q_2(\ell_2))=1.$$
So, by applying 1D Weyl-sum estimate one has 
\begin{align}\label{ko45}
W_{j_2,\ell_1}(t_2):=\sum_{2^{j_2}/2\le s\le t_2} e^{2\pi i \left( (\xi_2+\frac{\tilde{a}_3}{\tilde{q}}Q_1(\ell_1))s+()s^2+\cdots+()s^n\right)}=O(2^{j_2}2^{-\eta j_2}).
\end{align}
Next, fix $\ell_1$ and 
  apply the summation by parts  for the $t_2$-sum in (\ref{co24}) as  \begin{align*}
&\sum_{t_2\sim 2^{j_2}} [W_{j_2,\ell_1}(t_2)- W_{j_2,\ell_1}(t_2-1) ] \frac{  \mathcal{H}_{j_1}^{t_2}(\tilde{\beta}_1,0,\tilde{\beta}_3)}{t_2}\notag  \\
&=  \sum_{t_2\sim 2^{j_2}} W_{j_2,\ell_1}(t_2) \left(\frac{  \mathcal{H}_{j_1}^{t_2}(\tilde{\beta}_1,0,\tilde{\beta}_3)}{t_2}- \frac{  \mathcal{H}_{j_1}^{t_2-1}(\tilde{\beta}_1,0,\tilde{\beta}_3)}{t_2-1}\right)+O(\text{boundary}).
\end{align*}
where the boundary part is controlled by $2^{-\eta j_2} |\mathcal{H}_{j_1}^{t_2^*}(\tilde{\beta}_1,0,\tilde{\beta}_3)|$ for the end points $t_2^*$ in $[2^{j_2-1},2^{j_2}]$. So,  the above
is bounded by $U(j_1,j_2,\xi,q_1, q_3)$ defined as
\begin{align*}
	&2^{-\eta j_2} |\mathcal{H}_{j_1}^{t_2^*}(\tilde{\beta}_1,0,\tilde{\beta}_3)|\\
	&+\sum_{t_2\sim 2^{j_2}} |W_{j_2,\ell_1}(t_2) |\left(\frac{ | \mathcal{H}_{j_1}^{t_2}(\tilde{\beta}_1,0,\tilde{\beta}_3)- \mathcal{H}_{j_1}^{t_2-1}(\tilde{\beta}_1,0,\tilde{\beta}_3)|}{2^{j_2}} +\frac{ | \mathcal{H}_{j_1}^{t_2}(\tilde{\beta}_1,0,\tilde{\beta}_3) |}{2^{2j_2}}\right)  
\end{align*}
As $|e^{2\pi i \tilde{\beta_3}P(x_1,t_2)} -e^{2\pi i \tilde{\beta_3}P(x_1,t_2-1)}|\lesssim  2^{-j_2}| 2^{j\cdot \mathfrak{m}} \tilde{\beta_3}| $, one take the sum over $j_1$ satisfying $|2^{j\cdot \mathfrak{m}} \tilde{\beta_3}|\le 2^{(\eta/2)j_2}$  in (\ref{kom1}), to obtain that  
\begin{align*} 
	\sum_{j_1 }| \mathcal{H}_{j_1}^{t_2}(\tilde{\beta}_1,0,\tilde{\beta}_3)- \mathcal{H}_{j_1}^{t_2-1}(\tilde{\beta}_1,0,\tilde{\beta}_3)|\le \sum_{j_1 } |2^{-j_2}2^{j\cdot \mathfrak{m}} \tilde{\beta_3}|\le C2^{-j_2}2^{(\eta/2)j_2}
\end{align*}
and $\sum_{j_1 }\max_{t_2\sim 2^{j_2}}| \mathcal{H}_{j_1}^{t_2}(\tilde{\beta}_1,0,\tilde{\beta}_3) |\le C$  from the Hilbert transform estimate (\ref{015}).  This combined with (\ref{ko45}) yields that
\begin{align}\label{094g}
	\sum_{j_1:j\in \mathbb{Z}(\mathbb{F}) }U(j_1,j_2,\xi,q_1, q_3)\le C 2^{-(\eta/2)j_2}.
\end{align}
Recall (\ref{mfy}) and (\ref{85ss}). Then $\xi_1,\xi_3$ with $\phi^{\rm{major},\flat}_{1,j_1}(\xi_1) \phi^{\rm{major},\flat}_{3,j}(\xi_3)\ne0$ satisfy that  
$$|\tilde{\beta}_\nu|=|\xi_\nu-a_\nu/q_\nu|\le 1/(10q_{\nu}^2)\ \text{ for $\nu=1,3$ where  $1\le q_1,q_3\le 2^{j_2/(50)}$}$$   where
such $q_\nu$'s for fixed $\xi_\nu$   are dyadic such that $q_\nu\ge 5 q_\nu' $ if $q_\nu\ge q_\nu'$.  
By this together with  (\ref{094g}) and the fact that $|1-\phi^{\rm{major},\flat}_{2,j_2}(\xi_2)|\le 2$ for fixed $j_2$, we obtain that for some $c>0$,
\begin{align*}
	\text{LHS of (\ref{5449})}&\lesssim \sum_{q_1,q_3;dyadic}(q_1q_3)^{-\eta}2^{-(\eta/4)j_2}  =O(2^{-(\eta/4)j_2}).
\end{align*}
We now take $\eta$ to be $10^{-3}$.
\end{proof}
\begin{proof}[Proof of (\ref{59})]
By (\ref{8181}),(\ref{5009}) and (\ref{5449}), we have
\begin{align*} 
	&\sum_{j_1:j \in \mathbb{Z}^2(\mathbb{F})} |H^{\rm{discrete}}_{j}(\xi)| (1-\Phi_j^{\rm{major}}(\xi))\\
&=\sum_{j_1:j \in \mathbb{Z}^2(\mathbb{F})} |H^{\rm{discrete}}_{j}(\xi)|  \phi^{\rm{major},\flat}_{2,j_2}(\xi_2)\phi^{\rm{major},\flat}_{1,j_1}(\xi_1) \phi^{\rm{major},\flat}_{3,j}(\xi_3)(1-\Phi_j^{\rm{major}}(\xi)) +O(2^{-cj_2}).
\end{align*}
Here, one can observe that $\phi^{\rm{major},\flat}_{2,j_2}(\xi_2)\phi^{\rm{major},\flat}_{1,j_1}(\xi_1) \phi^{\rm{major},\flat}_{3,j}(\xi_3)(1-\Phi_j^{\rm{major}}(\xi))\equiv 0$ from the support condition of (\ref{mfy}),(\ref{85ss}) and (\ref{82tt}).
Thus, we proved
 (\ref{59}).
\end{proof}
\begin{proposition}\label{IW3} 
Suppose that every vertex in ${\bf N}(P,D_{\{-{\bf e}_1,-{\bf e}_2\}})$ has an even component. Then
 given  a positive integer $k$ and  $D \geq 2$,  there is an asymptotic formula,
\begin{align}\label{IW5}
	&\sum_{j \in \mathbb{Z}^2(\mathbb{F}): j_2\ge k} H^{\rm{discrete}}_{j}(\xi) \nonumber\\
	&\qquad=\sum_{q\le  k^{D/s}} \sum_{a;(a,q)=1} S (a/ q) m_{k}( \xi -a/q)  +E_{k}(\xi)
\end{align}
where     $\left|E_{k}(\xi)\right| \leq C_{D} k^{-9D/10}$   and the multiplier $ m_{k}( \xi )$ is
\begin{align*}
  \sum_{j=(j_1,j_2)\in\mathbb{Z}^2(\mathbb{F}); j_2\ge k}\phi\left( \frac{\xi_1}{2^{-(1-1/10)k}} \right) \phi\left( \frac{\xi_2}{2^{-(1-1/10)k}} \right) \phi\left(\frac{\xi_3 }{2^{-(k,k)\cdot \mathfrak{m}}2^{k/10}}\right)\mathcal{H}^{\Lambda(P)}_{j}(\xi).
\end{align*}
\end{proposition}

\begin{proof}
Recall  $\Phi^{\rm{major}}_j(\xi)=  \sum_{1\le q\le 2^{j_2/10}}\sum_{a:(a,q)=1}\Phi_{j}(\xi-a/q)$. 
By Propositions \ref{prop812}, we have 
\begin{align*}
	&\sum_{j_1: (j_1,j_2)\in \mathbb{Z}^2(\mathbb{F})}  H^{\rm{discrete}}_{j}(\xi)= \sum_{j_1: (j_1,j_2)\in \mathbb{Z}^2(\mathbb{F})} \Phi^{\rm{major}}_j(\xi)H^{\rm{discrete}}_{j}(\xi) +O(2^{-cj_2})\\
&=\sum_{j_1:j\in\mathbb{Z}^2(\mathbb{F})} \sum_{ 1\le q\le 2^{j_2/10}} \sum_{a: (a,q)=1}  \Phi_{j}(\xi-a/q) S\left(\frac{a}{q}\right)\mathcal{H}_{j}^{\Lambda(P)}\left( \xi-a/q\right) +O(2^{-cj_2})
\end{align*}
where $ \Phi_{j}(\xi-a/q) $ is in (\ref{36aa}).
Moreover, by using $S(a/q)=O(q^{-s})$ again, observe that
$$
\sum_{j_1:j\in\mathbb{Z}^2(\mathbb{F})} \sum_{  j_2^{D/s}\le  q\le 2^{j_2/10}} \sum_{a: (a,q)=1}  \Phi_{j}(\xi-a/q) S\left(\frac{a}{q}\right)\mathcal{H}_{j}^{\Lambda(P)}\left( \xi-a/q\right) +O(2^{-cj_2})
 $$ is bounded by $\sum_{q\ge j_2^{D/s}} q^{-s}\le Cj_2^{-D}$.
 Hence,
\begin{align*}
	&\sum_{j_1: (j_1,j_2)\in \mathbb{Z}^2(\mathbb{F})}  H^{\rm{discrete}}_{j}(\xi)\\
&=\sum_{j_1:j\in\mathbb{Z}^2(\mathbb{F})} \sum_{ 1\le q\le  j_2^{D/s}} \sum_{a: (a,q)=1}  \Phi_{j}(\xi-a/q) S\left(\frac{a}{q}\right)\mathcal{H}_{j}^{\Lambda(P)}\left( \xi-a/q\right)+O( j_2^{-D}).
\end{align*}
Next, by taking the sum of the above terms over $j_2\ge k$ as
\begin{align*}
	&\sum_{(j_1,j_2):j_2\ge k\  \text{and}\ j\in \mathbb{Z}^2(\mathbb{F})}  H^{\rm{discrete}}_{j}(\xi)= O(k^{-D})\\
&+\sum_{(j_1,j_2):j_2\ge k\  \text{and}\ j\in \mathbb{Z}^2(\mathbb{F})}\sum_{ 1\le q<j_2^{D/s} } \sum_{a: (a,q)=1}   \Phi_{j}(\xi-a/q) S\left(\frac{a}{q}\right)\mathcal{H}_{j}^{\Lambda(P)}\left( \xi-a/q\right).
\end{align*}
If $k^{D/s}\le q\le j_2^{D/s}$, then   $S(a/q)=O(q^{-s})=O(q^{-s/10}k^{-9D/10})$.
This combined with (\ref{8hb1}) and the continuous multiplier estimate $\sum_{j_1,j_2}|\mathcal{H}_{j}^{\Lambda(P)}\left( \beta\right)|\le C$  under the  evenness vertex hypothesis of ${\bf N}(P,D_{\{-{\bf e}_1,-{\bf e}_2\}})$, yields that
\begin{align*}
&\sum_{(j_1,j_2):j_2\ge k\  \text{and}\ j\in \mathbb{Z}^2(\mathbb{F})}\sum_{ k^{D/s}\le q<j_2^{D/s} } \sum_{a: (a,q)=1}   \Phi_{j}(\xi-a/q) S\left(\frac{a}{q}\right)\mathcal{H}_{j}^{\Lambda(P)}\left( \xi-a/q\right) \\
&\qquad\qquad=O(k^{-9D/10}).
\end{align*}
Therefore,  we obtain that
\begin{align}\label{pio}
	&\sum_{(j_1,j_2):j_2\ge k\  \text{and}\ j\in \mathbb{Z}^2(\mathbb{F})}  H^{\rm{discrete}}_{j}(\xi)= O(k^{-9D/10})\nonumber\\
&+\sum_{(j_1,j_2):j_2\ge k\  \text{and}\ j\in \mathbb{Z}^2(\mathbb{F})}\sum_{ 1\le q<k^{D/s} } \sum_{a: (a,q)=1}  \Phi_{j}(\xi-a/q) S\left(\frac{a}{q}\right)\mathcal{H}_{j}^{\Lambda(P)}\left( \xi-a/q\right).
\end{align}
Let $\beta_i=\xi_i-a_i/q$ for $i=1,2,3$. 
Note $2^{j\cdot\mathfrak{m}} \ge 2^{(k,k)\cdot\mathfrak{m}} $ if $j_2\ge k$, and $ \mathcal{H}_{j}^{\Lambda(P)}\left( \beta\right)=(|\beta_12^{j_1}|+|\beta_22^{j_2}|+|\beta_32^{j\cdot\mathfrak{m}}|)^{-2c} $. Then 
\begin{align*}
&\text{(i) If $ |\beta_32^{(k,k)\cdot\mathfrak{m}}|\ge 2^{k/10}$, then}\ |\mathcal{H}_{j}^{\Lambda(P)}\left( \beta\right)|^{1/2} =O( |\beta_32^{(k,k)\cdot\mathfrak{m}})=O(2^{-ck/10} ),\\
&\text{(ii) If $ |\beta_12^{k}|\ge 2^{k/10}$, then}\ | \mathcal{H}_{j}^{\Lambda(P)}\left( \beta\right)|^{1/2} =O( |\beta_12^{k}  |^{-c})=O(2^{-ck/10} ),\\
&\text{(iii) If $ |\beta_22^{k}|\ge 2^{k/10}$, then}\  |\mathcal{H}_{j}^{\Lambda(P)}\left( \beta\right)|^{1/2} =O( |\beta_22^{k}  |^{-c})=O(2^{-ck/10} ).
\end{align*}
This combined with (\ref{8hb1}) and $\sum_{j_1} |\mathcal{H}_{j}^{\Lambda(P)}\left( \beta\right)|^{1/2} \le C$ in (\ref{025}), implies that
\begin{align*}
&\sum_{j_2:j_2\ge k\  \text{and}\ j\in \mathbb{Z}^2(\mathbb{F})}\sum_{ 1\le q<k^{D/s} } \sum_{a: (a,q)=1}   \Phi_{j}(\beta) S\left(\frac{a}{q}\right)|\mathcal{H}_{j}^{\Lambda(P)}\left( \beta\right)|\\
&\times\left( 1-\phi\left( \frac{ \beta_1}{2^{-(1-1/10)k}} \right) \phi\left( \frac{ \beta_2}{2^{-(1-1/10)k}} \right) \phi\left( \frac{ \beta_3 2^{\mathfrak{m}\cdot (k,k)} }{2^{k/10}} \right) \right)
=O(2^{-ck})=O(k^{-D}).
\end{align*}
So, we  replace $\phi\left( \frac{ \beta_1}{2^{-j_1(9/10)}} \right) \phi\left( \frac{ \beta_2}{2^{j_2(9/10)}} \right) \phi\left( \frac{ \beta_3 2^{\mathfrak{m}\cdot j} }{2^{j_2/10}} \right)$   in (\ref{pio}) with the narrower support-function: $$\phi\left( \frac{ \beta_1}{2^{-(1-1/10)k}} \right) \phi\left( \frac{ \beta_2}{2^{-(1-1/10)k}} \right) \phi\left( \frac{ \beta_3 2^{\mathfrak{m}\cdot (k,k)} }{2^{k/10}} \right).$$ 
  This yields  the desired asymptotic formula (\ref{IW5}).
\end{proof}
For each vertex $\mathbb{F}$ of ${\bf N}(P,D_{\{-{\bf e}_1,-{\bf e}_2\}})$, we set 
$${\bf H}^{\rm{discrete}}_{\mathbb{F}}f(x):=\int e^{2\pi i x\cdot \xi}\sum_{j\in \mathbb{Z}^2(\mathbb{F})}H_{j}^{\rm{discrete}}(\xi_1,\xi_2,\xi_3)\widehat{f}(\xi)d\xi. $$
 In view of page 380 of \cite{IW},  to prove 
  $\|{\bf H}^{\rm{discrete}}_{\mathbb{F}}  \|_{\ell^{p}(\mathbb{Z}^3)\rightarrow \ell^p(\mathbb{Z}^3)} \le C  $   for  $1<p<\infty$,
  it suffices to prove the following Lemma.
\begin{lemma}\label{IW2}
	Given $\epsilon \in(0,1]$, for any $\lambda \in(0, \infty)$ there are two linear operators $A_{\lambda, \epsilon}$ and $B_{\lambda, \epsilon}$ such that ${\bf H}^{\rm{discrete}}_{\mathbb{F}}f=A_{\lambda, \epsilon}f+B_{\lambda, \epsilon}f$ satisfying
	$$
	\left\|A_{\lambda, \epsilon}\right\|_{\ell^{2} \rightarrow \ell^{2}} \leq C_{\epsilon} / \lambda,
	$$
	and
	$$
	\left\|B_{\lambda, \epsilon}\right\|_{\ell^{r} \rightarrow \ell^{r}} \leq C_{r, \epsilon} \lambda^{\epsilon}\ \text{for any $r \in[2, \infty)$. }
	$$
		The constants  $C_{\epsilon}$ and $C_{r, \epsilon}$ may depend on the coefficients   of $P(t)$. 
\end{lemma}
 
\begin{proof}
If $\lambda< 1$, then  choose $A_{\lambda, \epsilon}f={\bf H}^{\rm{discrete}}_{\mathbb{F}}f$ and $B_{\lambda, \epsilon}={\bf 0}$. Then it satisfies the above conditions due to the case $p=2$. We next consider $\lambda\ge 1$. We shall choose $$A_{\lambda, \epsilon}=A_{\lambda, \epsilon}^{1}+A_{\lambda, \epsilon}^{2} +A_{\lambda, \epsilon}^{3}\ \text{and}\ B_{\lambda, \epsilon}=B_{\lambda, \epsilon}^{1}+B_{\lambda, \epsilon}^{2}.$$
 We shall determine in the order $B_{\lambda, \epsilon}^{1},$  $A^1_{\lambda, \epsilon},A^2_{\lambda, \epsilon},B^2_{\lambda, \epsilon}$ and lastly $A^3_{\lambda, \epsilon}$.
Take $$k=[\lambda^{\epsilon}]  $$ where $[c]$ is the largest integer $\le c$.    Define
$B_{\lambda, \epsilon}^{1}$ as an operator whose multiplier is
$$m_{B_{\lambda,\epsilon}^1}(\xi)=\sum_{j\in \mathbb{Z}^2(\mathbb{F}): 1\le j_2\le  k+1} H^{\rm{discrete}}_{j_1,j_2}(\xi) 
$$  
where
$ B_{\lambda,\epsilon}^1f=\sum_{j\in \mathbb{Z}^2(\mathbb{F}):1\le j_2\le k+1}  [H^{\rm{discrete}}_{j_1,j_2}]^{\vee}*f$.
By applying the   $\ell^p$ bound of the one-parameter
discrete Hilbert transform along polynomial curve uniformly in their coefficients,  we have $$\left\|B_{\lambda, \epsilon}^{1}\right\|_{\ell^{r} \rightarrow \ell^{r}} \leq Ck\le   C\lambda^{\epsilon}.$$
First, set $D=\frac{(3/2)}{\epsilon}$      in  (\ref{IW5}). Choose  the operator $A_{\lambda, \epsilon}^{1}$ whose multiplier is $$m_{A^1_{\lambda,\epsilon}}(\xi)=E_{k}(\xi)\ \text{ in (\ref{IW5}) satisfying}\ \left\|A_{\lambda, \epsilon}^{1}\right\|_{\ell^{2} \rightarrow \ell^{2}} \leq  \frac{C}{k^{9D/10}}\le \frac{C}{\lambda}.$$
Hence, from (\ref{IW5}) there remains to control
\begin{align}\label{aa}
 \sum_{q\le  k^{D/s}} \sum_{a;(a,q)=1} S (a/ q) m_{k}( \xi -a/q)
\end{align}
where $m_k$ is defined below (\ref{IW5}).
Let $N:=[k^{D/s}]$  to treat the sum $  \sum_{q\le  k^{D/s}} $.  In order to utilize   Lemma \ref{lem82} later,  we shall take $m_k$ in (\ref{IW5}) to play the role of $m((\xi-a/q)/\eta)$   in that lemma where   $m(\xi)$ is  expressed in (\ref{skc1}) with $\eta=2^{-9k/10}$.  Recall $Y_{N}=Y_{N, \delta}$ as the   set constructed in Lemma \ref{lem82} satisfying that $ Z_{N} \subset Y_{N} \subset Z_{e^{N ^{\delta}}}$ and the operator whose multiplier is
$  \sum_{a / q \in \mathcal{R}(Y_N)} m\left(\frac{\xi-a / q}{\eta}\right)$ is  bound in $\ell^p$.
To treat $Y_N$ and  our summation range $Z_N$ of $1\le q\le N=[k^{D/s}]$ in (\ref{aa})   such that $\text{gcd}(a,q)=1$,   split the range of $q$ in $Y_N$ so that
 $\sum_{Z_N}=\sum_{Y_N}-\sum_{Y_N\setminus Z_N}$. Take $A_{\lambda, \epsilon}^{2}$ whose  multiplier is 
\begin{align*}
m_{A_{\lambda, \epsilon}^{2}}(\xi)=	-\sum_{q \in Y_{N} \backslash Z_{N}}\sum_{a; (a,q)=1} S(a/ q) m_{k}   (\xi-a/q)\ \text{where $m_k$ in (\ref{IW5})}.
\end{align*}
Then from $q\ge N\ge k^{\frac{D}{s}}$, we have  $|S(a / q)|=O(q^{-s})=O(q^{-s/3}k^{-2D/3})=O(q^{-c}\lambda^{-1})$ since $[\lambda^{\epsilon}]=k$ and $  D=3/(2\epsilon)$. By this  with (\ref{025}) and the disjointness of the major arcs for each fixed $q$,  
$$
\left\|A_{\lambda, \epsilon}^{2}\right\|_{\ell^{2} \rightarrow \ell^{2}} \leq C / \lambda .
$$
Hence, there  remains the main part:
\begin{align}
\sum_{q \in Y_{N}}\sum_{a; (a,q)=1} S(a/ q) m_{k}  (\xi-a/q)\ \text{where $m_k$ is in  (\ref{IW5})}.\label{05k}
\end{align}
Take
$$v_{k}(\xi):=\frac{1}{[2^{k/2}]^2} \sum_{t\in \{1,\cdots,[2^{k/2}]\}^2} e^{-2 \pi i (\xi_1t_1+\xi_2t_2+\xi_3P(t_1,t_2))}$$ and splits (\ref{05k}) as the sum of the following two functions (multipliers):
\begin{align}
&m_{B_{\lambda, \epsilon}^{2} }(\xi):=	\sum_{q \in Y_{N}}\sum_{a; (a,q)=1}  v_{k}(\xi)m_{k} (\xi -a/q), \label{IW11}
\\
&m_{A_{\lambda, \epsilon}^{3} }(\xi):=\sum_{q \in Y_{N}}
\sum_{a:  (a,q)=1} ( S(a / q) -v_k(\xi))m_{k}(\xi-a/q).\label{pol2}
\end{align}
\begin{proof}[Proof of $\ell^r$ boundedness for $B_{\lambda, \epsilon}^{2}$]
We can write   $B_{\lambda, \epsilon}^{2}=U_{N} \circ V_{k}$ in (\ref{IW11}) where the operator $V_k$ has the bounded multiplier $v_k(\xi)$  such that $\left\|V_{k}\right\|_{\ell^{r} \rightarrow \ell^{r}} \leq C_{r}$ and
  the operator $U_{N}$ has  the multiplier:
\begin{align*}
m_{U_N}(\xi)= \sum_{q \in Y_{N}}\sum_{a; (a,q)=1} m_{k}  (\xi-a/q)\ \text{with $m_k$ in (\ref{IW5})}.
\end{align*}
Then we can express
\begin{align}\label{pol1}
m_{U_N}(\xi)= \sum_{q \in Y_{N}}\sum_{a; (a,q)=1} m \left(\frac{\xi-a/q}{\eta}\right)
\end{align}
where  $\eta=2^{-(1-1/10)k}$ and $m$ is given by
\begin{align}\label{skc1}
m(\xi)= \sum_{j=(j_1,j_2)\in\mathbb{Z}^2(\mathbb{F}); j_2\ge k}\phi\left( \xi_1\right) \phi\left( \xi_2\right) \phi\left(\frac{2^{-(1-1/10)k}\xi_3 }{2^{-(k,k)\cdot \mathfrak{m}}2^{k/10}}\right)\mathcal{H}^{\Lambda(P)}_{j}(2^{-(1-1/10)k} \xi)
\end{align}
which is  a bounded multiplier  in $L^r$ whose norm is bounded uniformly in $k$. This follows as
 $m( \xi)$ is a bounded multiplier in $L^r$ whose bound is independent of  $k$, which is proved in  (\ref{025h}).
Before applying  Lemma \ref{lem82}, take $$\delta=\frac{s}{4D}\ \text{so that}\ N^{2\delta}\sim (k^{\frac{D}{s}})^{2\delta}=k^{1/2}$$  
which implies $\eta=2^{-9k/10} \leq e^{-N^{2 \delta}}$.
 By applying
Lemma \ref{lem82}  for (\ref{pol1}), we have
\begin{align}\label{IW10}
	\left\|U_{N}\right\|_{\ell^{r} \rightarrow \ell^{r}}\le C_{r, \epsilon}(\ln  N)^{2/\delta}\lesssim C_{r,\epsilon}(\ln \lambda)^{2 / \delta} .
\end{align}
Due to (\ref{IW10}) and $\left\|V_{k}\right\|_{\ell^{r} \rightarrow \ell^{r}} \leq C_{r}$, we obtain that
$ 
\ \|B_{\lambda, \epsilon}^{2} \|_{\ell^{r} \rightarrow \ell^{r}} \leq C_{r, \epsilon} \lambda^{\epsilon}.
$ 
\end{proof}
\begin{proof}[Proof of $\ell^2$ boundedness for $A_{\lambda, \epsilon}^{3}$]
 On the other hand,   for  $\xi,q$  in  (\ref{pol2}),  check that
\begin{itemize}
\item
$|\xi_3-a_3/q| \leq 2 \cdot 2^{-(k,k)\cdot \mathfrak{m}}2^{k/10}$ and
 $|\xi_1-a_1/q|, |\xi_2-a_2/q|\le 2^{-9k/10}$,
\item  $q \leq e^{N^{\delta}} \le Ce^{k^{1 / 2}}$ due to $  Y_{N} \subset Z_{e^{N ^{\delta}}} $ and  $0\le t_1,t_2\le 2^{k / 2}$. 
\end{itemize}
Using the above size   after dividing $\{1,\cdots,2^{k/2}\}^2$ into smaller $q\times q$ sized boxes,  \begin{align*} 
	v_{J}(\xi)-S(a / q)=O(2^{-4k/10})\ \text{leading} \  \left\|A_{\lambda, \epsilon}^{3}\right\|_{\ell^{2} \rightarrow \ell^{2}} \leq C_{\epsilon,c} 2^{-c' k} \leq C_{\epsilon} / \lambda 
\end{align*}  
from Plancherel's theorem. 
\end{proof}
 Therefore, we proved
$$
 \ \sum_{i=1}^3\|A_{\lambda, \epsilon}^{i}\|_{\ell^2\rightarrow \ell^2}\le  C_{\epsilon} / \lambda\ \text{and}\ \sum_{i=1}^2\| B_{\lambda, \epsilon}^{i} \|_{\ell^r\rightarrow \ell^r}\le   C_{r,\epsilon}\lambda^{\epsilon}. $$
So, we  finish the proof of Lemma \ref{IW2}.
\end{proof}

	\section{Necessity}\label{Sec9}
Assume that a vertex $\mathbb{F}  \in\mathcal{F}^0({\bf N}(P,D_{\{-{\bf e}_1,-{\bf e}_2\}} ))$ is  $\mathfrak{m}=(m_1,m_2)=(\text{odd},\text{odd})$.  Choose $\mathfrak{q}=(q_1,q_2)\in (\mathbb{F}^*)^{\circ}$ such that 
\begin{align*}
q_1,q_2<0\ \text{where}\  |q_1|\ge |q_2| \ \text{and}\ |\mathfrak{q}|=1.
\end{align*}
Then one can find a supporting plane $\pi_{\mathfrak{q}}$    of  ${\bf N}(P,D_B)$ such that  $\mathbb{F}=\{\mathfrak{m}\}\in \pi_{\mathfrak{q}}$ and  $\Lambda(P)\setminus\{\mathfrak{m}\}\subset (\pi^+_{\mathfrak{q}} )^{\circ}$. By this, one can observe that  there exists a $h>0$ depending only on $\Lambda(P)$ satisfying that $$
 \mathfrak{q}\cdot (\mathfrak{n}- \mathfrak{m}) \ge  h>0 \ \text{ for all $\mathfrak{n}\in \Lambda(P)\setminus\{\mathfrak{m}\}$. }
$$
Put $j=(j_1,j_2):= (-\alpha q_1,-\alpha q_2)=\alpha(- \mathfrak{q})$ where $-q_1,-q_2>0$. Then it holds that 
  \begin{align}\label{92}
j\cdot ( \mathfrak{m}- \mathfrak{n})    \ge \alpha h>0, \ \text{that is}\  2^{j\cdot \mathfrak{m}} \ge 2^{\alpha h }  2^{j\cdot \mathfrak{n}}\ \text{  $\forall\mathfrak{n}\in \Lambda(P)\setminus\{\mathfrak{m}\}$. }
\end{align}
Define some constants:
  \begin{itemize}
  \item $N:=(N_1,N_2)=(2^{j_1},2^{j_2})$,
  \item $ b:=|q_2|/|q_1| $ and $\epsilon=\min\{b, |h| \}/10$ so that $h\alpha\ge h |j| \ge 10\epsilon |j| $,
\item $\xi_3:=\frac{ |N|^{\epsilon}}{N_1^{m_1} N_2^{m_2}}.$
\end{itemize}
 Then
\begin{align*}
|N|/2\le N_1\le |N|\ \text{and}\ N_1^b= N_2.
\end{align*}
We now claim that
\begin{align}\label{0521}
\left|\sum_{k\in \mathcal{R}(N)}    \frac{e^{2\pi i\xi_3 P(t_1,t_2)}}{t_1 t_2}\right|  \ge c\log |N|\ \text{for $c>0$ independent of $N$}
\end{align}
which yields the necessity part of Main Theorem 1.
\begin{proof}[Proof of (\ref{0521})]
To show (\ref{0521}),   start with 
	$$G_N(t_1,t_2):=e^{2\pi i\xi_3 P(t_1,t_2)}\psi_{N_1}(t_1)\psi_{N_2}(t_2)\frac{1}{t_1}\frac{1}{t_2}.$$
Here, we defined a smooth function $\psi_{M} $ such that
\begin{itemize}
\item  $\text{supp}\left( \psi_{M}\right)=\{s: 1/2\le |s|\le M+1/2\}$ with  $\psi_M(s)=1$ on $ [-M,-1]\cup [1,M]$,
\item
$ |\frac{d}{ds}  \psi_{M}(s) |\le 1$   and  $\text{supp}\left(\frac{d}{ds}  \psi_{M}\right)\subset\{s: 1/2\le |s|\le 1\}\cup \{s: M\le |s|\le M+1/2\}$.
\end{itemize}
Then
	$$\sum_{t\in \mathcal{R}(N)}    \frac{e^{2\pi i\xi_3 P(t_1,t_2)}}{t_1 t_2}  =\sum_{k\in \mathbb{Z}^2} G_N(k_1,k_2).$$
		By the Poisson-summation formula, it holds that
\begin{align*}
&\sum_{k\in \mathbb{Z}^2} G_N(k_1,k_2)=\sum_{k\in \mathbb{Z}^2} \widehat{G_N}(k_1,k_2)\end{align*}
where $\widehat{G_N}(k_1,k_2) =\int e^{2\pi i k\cdot t} G_N(t_1,t_2) dt$ which is given by the integral
\begin{align*}
\widehat{G_N}(k_1,k_2) = \int e^{2\pi i k\cdot t} e^{2\pi i\xi_3 P(t_1,t_2)}\psi_{N_1}(t_1)\psi_{N_2}(t_2)\frac{dt_1}{t_1}\frac{dt_2}{t_2}.
		\end{align*}
	We split
	$$\sum_{k\in \mathbb{Z}^2} \widehat{G_N}(k_1,k_2)=\widehat{G_N}(0,0)+\sum_{k\ne (0,0)} \widehat{G_N}(k_1,k_2)$$
	and observe that 
	$$\sum_{k\ne (0,0)} \widehat{G_N}(k_1,k_2)=O(1)$$
	from $ \widehat{G_N}(k_1,k_2)=O(1/|k|^{5/2})$, which follows from the derivative conditions $$|\nabla_{t} \left(k\cdot t+\xi_3P(t_1,t_2)\right)|\approx |k|.$$
 We express $\widehat{G_N}({\bf 0})$ as
\begin{align*}
&  \int  e^{2\pi i\xi_3 P(t_1,t_2)}\psi_{N_1}(t_1)\psi_{N_2}(t_2)\frac{dt_1}{t_1}\frac{dt_2}{t_2} = \int  e^{2\pi i\xi_3 t^{\mathfrak{m}} }\psi_{N_1}(t_1)\psi_{N_2}(t_2)\frac{dt_1}{t_1}\frac{dt_2}{t_2}+ \mathcal{E}_N.
\end{align*}
Here, by the mean value theorem and (\ref{92}),
\begin{align*}
\mathcal{E}_N&\le C \sum_{\mathfrak{n}\in \Lambda(P)\setminus\{\mathfrak{m}\}} \int  |\xi_3 t^{\mathfrak{n}}|\psi_{N_1}(t_1)\psi_{N_2}(t_2)\frac{dt_1}{t_1}\frac{dt_2}{t_2}\\
&\le C \sum_{\mathfrak{n}\in \Lambda(P)\setminus\{\mathfrak{m}\}}  |\xi_3|2^{j\cdot \mathfrak{n}}\le  C \sum_{\mathfrak{n}\in \Lambda(P)\setminus\{\mathfrak{m}\}} | \xi_3|  2^{-h\alpha } 2^{  j\cdot \mathfrak{m}}\\
&\le C \frac{2^{-h \alpha}|N|^{\epsilon}}{N_1^{m_1}N_2^{m_2}} N_1^{m_1}N_2^{m_2}\le C|N|^{-\epsilon}
\end{align*}
due to the definition of constants above. We now show that
\begin{align}\label{9431}
\left|\int  e^{2\pi i\xi_3 t^{\mathfrak{m}} }\psi_{N_1}(t_1)\psi_{N_2}(t_2)\frac{dt_1}{t_1}\frac{dt_2}{t_2}\right|\ge c\log |N|.
\end{align}
Split   the region $$ \{1/2\le |t_i|\le1\}\cup\{1\le |t_i|\le N_i\}\cup \{N_i\le |t_i|\le N_i+1/2\}.$$
By using the uniform boundedness  $\left|\int e^{2\pi i\xi_3 t^{\mathfrak{m}} }\psi_{N_1}(t_i)\frac{dt_i}{t_i}\right|\le 10$ combined with  $\psi_{N_i}(t_i)=1$ on $ [-N_i,-1]\cup [1,N_i]$, we obtain that
\begin{align*}
&\int  e^{2\pi i\xi_3 t^{\mathfrak{m}} }\psi_{N_1}(t_1)\psi_{N_2}(t_2)\frac{dt_1}{t_1}\frac{dt_2}{t_2}\\
 &=
\int_{1\le |t_1|\le N_1} \int_{1\le |t_2|\le N_2}  e^{2\pi i\xi_3 t^{\mathfrak{m}} }    \frac{dt_1}{t_1}\frac{dt_2}{t_2}+O(1).
\end{align*}
As $m_1,m_2$ are odd,  it suffices to show that for $N\gg 1$,
\begin{align}\label{0mm4}
4\left|\int_{1}^{N_1}\int_{1}^{N_2}  \sin(2\pi i\xi_3 t_1^{m_1}t_2^{m_2}) \frac{dt_1}{t_1}\frac{dt_2}{t_2}\right|\ge (\epsilon/2^{100}) \log |N|.
\end{align}
Hence, to finish the proof of (\ref{0521}), we shall prove (\ref{0mm4}) below.
\end{proof}
\begin{proof}[Proof of (\ref{0mm4})]
  The integral on the LHS of (\ref{0mm4}) is comparable to
  \begin{align*}
\int_{1}^{N_2^{m_2}}  \int_{1}^{N_1^{m_1}}  \sin(2\pi i\xi_3 t_1 t_2 ) \frac{dt_1}{t_1}\frac{dt_2}{t_2}&=A+B+C
\end{align*}
where
\begin{align*}
A&=\int_{N_2^{m_2} |N|^{-\epsilon}2^{100}}^{N_2^{m_2}}  \int_{1}^{N_1^{m_1}}  \sin(2\pi \xi_3 t_1 t_2 ) \frac{dt_1}{t_1}\frac{dt_2}{t_2},\\
B&=\int^{N_2^{m_2} |N|^{-\epsilon} 2^{100}}_{N_2^{m_2}|N|^{- \epsilon } 2^{-100}}  \int_{1}^{N_1^{m_1}}  \sin(2\pi \xi_3 t_1 t_2 ) \frac{dt_1}{t_1}\frac{dt_2}{t_2},\\
C&=\int^{N_2^{m_2} |N|^{- \epsilon}2^{-100} }_{1}  \int_{1}^{N_1^{m_1}}  \sin(2\pi \xi_3 t_1 t_2 ) \frac{dt_1}{t_1}\frac{dt_2}{t_2}.
  \end{align*}
	 By applying the mean value property with the size $\xi_3=\frac{|N|^{\epsilon}}{N_1^{m_1}N_2^{m_2}}$, we obtain that $|C|\lesssim 1$.  By utilizing the uniform boundedness of $$\left|\int_{1}^{N_1^{m_1}}  \sin(2\pi \xi_3 t_1 t_2 ) \frac{dt_1}{t_1}\right|,$$ we have
	 $|B|\lesssim 1$. Finally, to treat $A$, use the change of variable $\xi_3 t_1t_2\rightarrow s_1 $, then
	 \begin{align*}
	  \int_{1}^{N_1^{m_1}}  \sin(2\pi \xi_3 t_1 t_2 ) \frac{dt_1}{t_1}&= \int_{s_1=\frac{|N|^{\epsilon}}{N_1^{m_1}N_2^{m_2}}  t_2}^{s_1=\frac{|N|^{\epsilon}N_1^{m_1}}{N_1^{m_1}N_2^{m_2}} t_2}  \sin(2\pi s_1) \frac{ds_1}{s_1} \\
	  &=\pi/2 +O(1/100)
	 \end{align*}
because of
$  N_2^{m_2} |N|^{-\epsilon}2^{100} \le t_2\le N_2^{m_2}  $  and the observation
 $$\left|\int_0^M \sin (2\pi s_1)\frac{ds_1}{s_1}-\int_0^{\infty} \sin (2\pi s_1)\frac{ds_1}{s_1}\right|\le 2^{-90}\ \text{ whenever $M>2^{100}$.} $$ 
 Therefore, take $\int dt_2$ for the above approximation to obtain  
 $A\approx \epsilon ( \pi/2)\log |N|$. 
 \end{proof}

\end{document}